\documentclass[reqno]{amsart}
\usepackage{amsfonts}
\usepackage{graphicx}
\usepackage{amsfonts,amsmath, amssymb}
\usepackage{hyperref}
\usepackage{marginnote}
\usepackage{color}
\usepackage{mathtools}
\usepackage{cancel}
\usepackage{cite}
\usepackage{accents}
\newcommand{\ubar}{\underaccent{\bar}}
\numberwithin{equation}{section}

\newcommand{\R}{\mathbb{R}}

\newtheorem{theorem}{Theorem}[section]

\newtheorem{lemma}[theorem]{Lemma}
\newtheorem{proposition}[theorem]{Proposition}
\newtheorem{remark}[theorem]{Remark}

\def\m{\mu}

\def\s{\sigma}

\def\f{\frac}

\def\b{\bar}
\usepackage{tikz}

\newcommand{\dd}{{\rm d}}

\newcommand{\CD}{\mathcal {D}}
\newcommand{\CE}{\mathcal{E}}

\newcommand{\CH}{\mathcal{H}}
\newcommand{\CI}{\mathcal{I}}
\newcommand{\CJ}{\mathcal{J}}

\newcommand{\CN}{\mathcal{N}}

\newcommand{\CR}{\mathcal{R}}

\newcommand{\la}{\lambda}
\newcommand{\de}{\delta}

\newcommand{\pa}{\partial}

\newcommand{\vep}{\varepsilon}

\begin{document}

\title[Ion-acoustic shock in a collisional plasma]{
Ion-acoustic shock in a collisional plasma
%Shock waves for one-dimensional Navier-Stokes-Poisson system
%On the shock structure problem in a fully ionized plasma
}

\author[R.-J. Duan]{Renjun Duan}
%\thanks{}
\address[RJD]{Department of Mathematics, The Chinese University of Hong Kong,
Shatin, Hong Kong, P.R.~China}
\email{rjduan@math.cuhk.edu.hk}

\author[S.-Q. Liu]{Shuangqian Liu}
\address[SQL]{Department of Mathematics, Jinan University, Guangzhou 510632, P.R.~China}
\email{tsqliu@jnu.edu.cn}

\author[Z. Zhang]{Zhu Zhang}
%\thanks{}
\address[ZZ]{Department of Mathematics, The Chinese University of Hong Kong,
Shatin, Hong Kong, P.R.~China}
\email{zzhang@math.cuhk.edu.hk}

\begin{abstract}
The paper is concerned with the propagation of ion-acoustic shock waves in a collision dominated plasma. We firstly establish the existence and uniqueness of a small-amplitude smooth travelling wave, then justify its approximation to the shock profile of the  KdV-Burgers equations in a suitable asymptotic regime where dissipation in terms of viscosity coefficient is much stronger than dispersion by the Debye length, and  prove in the end the large time asymptotic stability of travelling waves under suitably small smooth perturbations.
\end{abstract}

%old: 35Q20, 35B20, 35B35, 35B45  		


\date{\today}
%\thanks{}
\maketitle

\setcounter{tocdepth}{1}
\tableofcontents

\thispagestyle{empty}

\section{Introduction}

With a rapid phase transition occurring, shock wave is regarded as one of fundamental nonlinear phenomena in both gas dynamics and charged plasma. Historically, there have been a huge number of physical literatures and numerical experiments to understand the propagation of plasma shock waves \cite{NBS,JP,J,S}. In this paper, we carry out a mathematical study of existence and stability of a smooth shock profile for a model system of the Navier-Stokes-Poisson equations used to describe the dynamics of ions with the viscosity effect in the absence of magnetic fields. We also justify that the propagation of shock profiles is governed by the KdV-Burgers equations in a suitable regime, which coincides with those results in numerical simulations, cf.~\cite{GsKG,NBS}.

%\begin{subsection}{Equations of motion}
\subsection{Equations of motion}

Under the influence of a self-consistent electrostatic field, the dynamics of ions can be described by the following one-dimensional  Navier-Stokes-Poisson system:
\begin{equation}\label{1.1}
\left\{
\begin{aligned}
&\partial_t n+\partial_x(nu)=0,\\
&\partial_t(nu)+\partial_{x}(nu^{2}+Tn)=\mu\partial_{xx}u-n\partial_x\phi,\\
&-\lambda^{2}\partial_{xx}\phi=n-e^{\phi}.
\end{aligned}\right.%\qquad t>0,x\in \mathbb{R}
\end{equation}
Here $n=n(t,x)>0$ and $u=u(t,x)$  are respectively  the density and velocity for ions with $t\geq 0$ and $x\in \R$. The constants $T\geq0$, $\mu>0$ and $\lambda>0$ stand for the absolute temperature, viscosity coefficient and Debye length, respectively. Particularly, when $T=0$, the momentum equation is pressureless and \eqref{1.1} is usually used to model the motion of cold plasma. The electric potential $\phi=\phi(t,x)$ is induced by the total charge of ions %$n$
and electrons. %$n_e=e^{\phi}.$
We have assumed that the density of electrons is determined by the Boltzmann relation $n_e=e^{\phi}$, % called ,
%The law of electron density $n_e=e^{\phi}$ is called isothermal Boltzmann relation,
cf.\cite{C, KT}. Such relation is a physical assumption according to the fact that lighter electrons get close to the equilibrium state at a much faster rate than heavier ions in plasma, and also it can be formally derived from the two-fluid model by taking the velocity of electrons as zero,
%density of electron $n_e=0$ and viscosity of electron flow $\mu_e=0$,
see \cite{H,GP}. To solve \eqref{1.1},
initial data are given by
\begin{align}\label{1.1-1}
[n,u](0,x)=[n_0,u_0](x),
\end{align}
with
\begin{align}\label{1.1-2}
\lim_{x\rightarrow\pm\infty}[n_0,u_0](x)=[n_{\pm},u_{\pm}].
\end{align}
The far-field data of $\phi$ are given by
\begin{align}\label{1.1-3}
\lim_{x\rightarrow\pm\infty}\phi(t,x)=\phi_{\pm},
\end{align}
with the  quasi-neutral condition $\phi_\pm=\log n_\pm$ at $x=\pm\infty$.
%: %on the far fields $[n_{\pm},\phi_{\pm}]$:
%\begin{align}\label{1.4-1}
%n_{\pm}=e^{\phi_{\pm}}.
%\end{align}

%\end{subsection}

%\begin{subsection}{Existence of shock profile}
\subsection{Existence of shock profile}

In general the large time behavior of solutions to the one-dimensional Cauchy problem \eqref{1.1} and \eqref{1.1-1} is determined by the far-field data $[n_\pm,u_\pm,\phi_\pm]$ as given in \eqref{1.1-2} and \eqref{1.1-3}. In the paper, we are concerned with only the shock profile under the quasi-neutral condition $\phi_\pm=\log n_\pm$.
%\eqref{1.4-1}.
We first recall the definition of shock profiles briefly. %We start with
Let  $[n,u,\phi]$ be a smooth travelling wave solution
%$[n,u,\phi]$
of \eqref{1.1} which depends only on the variable $\xi=x-st$ with the wave speed $s$ to be determined later and connects the far-fields $[n_{\pm},u_{\pm},\phi_{\pm}]$ at $x=\pm\infty$ respectively. The profile is thereby governed by the following system of ODEs: % system:
\begin{equation}\label{1.2}
\left\{\begin{aligned}
&-s\f{\dd n}{\dd \xi}+\f{\dd (nu)}{\dd \xi}=0,\\
&-s\f{\dd(nu)}{\dd \xi}+\f{\dd (nu^2+Tn)}{\dd \xi}=\mu \f{\dd^2 u}{\dd \xi^2}-n\f{\dd \phi}{\dd \xi},\\
&-\lambda^2\f{\dd^2\phi}{\dd \xi^2}=n-e^{\phi},
\end{aligned}\right.
\end{equation}
with
\begin{align}\label{1.3}
\lim_{\xi\rightarrow \pm \infty}[{n}(\xi), {u}(\xi), {\phi}(\xi) ]=[n_{\pm}, u_{\pm}, \phi_{\pm}].
\end{align}
%Due to $\eqref{1.2}_3$, we should set the following compatibility condition on the far fields $[n_{\pm},\phi_{\pm}]$:
%\begin{align}\label{1.4-1}
%n_{\pm}=e^{\phi_{\pm}}.
%\end{align}
Note that by using the third equation of \eqref{1.2},
%multiplying $\f{\dd\phi}{\dd \xi}$ to $\eqref{1.2}_3$,
we can rewrite the last term on the right hand side of the momentum equation of  \eqref{1.2} as a conservative form,
%: %$n\f{\dd\phi}{\dd \xi}$ as
%$$
%-n\f{\dd\phi}{\dd \xi}=\f{\dd}{\dd \xi}\left[\f{\lambda^2}{2}\left(\f{\dd \phi}{\dd \xi}\right)^2-e^{\phi}\right].
%$$
so that system \eqref{1.2} is equivalent to
\begin{equation}\label{1.4}
\left\{
\begin{aligned}
&-s\f{\dd n}{\dd \xi}+\f{\dd(nu)}{\dd \xi}=0,\\
&-s\f{\dd(nu)}{\dd \xi}+\f{\dd(nu^2+Tn)}{\dd \xi}=\mu \f{\dd^2u}{\dd \xi^2}+\f{\dd}{\dd\xi}\left[\f{\lambda^2}{2}\left(\f{\dd\phi}{\dd \xi}\right)^2-e^{\phi}\right],\\
&-\lambda^2\f{\dd^2\phi}{\dd \xi^2}=n-e^{\phi},
\end{aligned}
\right.
\end{equation}
which consists of two conservation laws together with the Poisson equation for the electric potential.
By integrating the first two equations in \eqref{1.4} over $-\infty<\xi<\infty$, then using $\phi_\pm=\log n_\pm$,
% \eqref{1.4-1},
and further assuming that all the first-order derivative terms vanish as $\xi\to\pm\infty$, we obtain %yields
the Rankine-Hugoniot conditions as follows:
\begin{equation}\label{1.6}
\left\{
\begin{aligned}
&-s(n_+-n_-)+n_+u_+-n_-u_-=0,\\
&-s(n_+u_+-n_-u_-)+n_+u^2_+-n_-u^2_-+(T+1)(n_+-n_-)=0.\\
\end{aligned}
\right.
\end{equation}
We point out that the R-H condition \eqref{1.6} is the same as the one for the following quasi-neutral Euler system:
\begin{equation}\label{1.7}
\left\{
\begin{aligned}
&\partial_t n+\partial_x(nu)=0,\\
&\partial_t(nu)+\partial_{x}\left(nu^{2}+(T+1)n\right)=0,
%,\quad t>0,  x\in \mathbb{R}\\
%&n=e^{\phi}.
\end{aligned}\right.
\end{equation}
%Note that \eqref{1.7}
which can be %in
formally obtained by letting $\mu=\lambda=0$ in \eqref{1.1}.%\\

The 1(2)-shock profile is the travelling wave with compressibility, namely, %\textcolor[rgb]{1.00,0.00,0.00}{\cancel{$s<0$ and} }
 $n_+>n_-$ (%\textcolor[rgb]{1.00,0.00,0.00}{\cancel{$s>0$ and}}
 $n_+<n_-$, respectively). In this paper, we will concern  only the 2-shock profile (shock profile henceforth) and assume without loss of generality that the upstream data are constants given by %that
$
[n_-,u_-,\phi_-]=[1,0,0].
$
%\end{subsection}
%\begin{subsection}{Existence of shock profile}

First of all, as for the existence of shock profiles of \eqref{1.1} with fixed constants $\mu>0$ and $\lambda>0$, we have
%the following result.

\begin{theorem}\label{thm1.1}
Let $T\geq 0$. For given data $[n_-,u_-]$ with $n_->0$, there exist positive constants %$\delta_0$,
$\hat{\vep}_0$, $\b{C}$, $\ubar{C}$, $\theta$ and $C_k$ %and $\theta_k$
$(k=0,1,\cdots)$
such that if $[n_+,u_+]$ satisfies \eqref{1.6} with $n_+<n_-$ and
$|n_+-n_-|\leq %\delta_0
\hat{\vep}_0$, the problem \eqref{1.2} or equivalently \eqref{1.4} has a smooth traveling wave solution of the form $[\b{n},\b{u},\b{\phi}](x-st)$ connecting the far fields $[n_{\pm},u_{\pm},\phi_{\pm}]$ with $\phi_{\pm}=\log n_{\pm}$, which is unique up to a spatial shift and satisfies the following properties:
\begin{align}
\b{n}_x(x-st)=%s^{-1}\b{n}^2\b{u}_x(x-st)<0
{\f{\b{n}^2\b{u}_x(x-st)}{n_-|u_--s|}}
,\quad \ubar{C}\b{n}_x\leq\b{\phi}_{x}\leq \b{C}\b{n}_x<0, %\quad
\label{1.1.1}
\end{align}
for any $x\in \mathbb{R}$ and $t>0$, and
\begin{equation}
\left|\f{
\dd^k}{\dd x^k}\left[\b{n}-n_{\pm},\b{u}-u_{\pm},\b{\phi}-\phi_{\pm}\right]\right|\\
\leq C_k|n_+-n_-|^{k+1}e^{-%\theta
{\theta}|n_+-n_-|\cdot|x-st|},
%\quad x-st\lessgtr0.
\label{1.1.2}
\end{equation}
for $x-st\lessgtr0$ and $k=0,1,\cdots$.
\end{theorem}

The existence of shock waves is a fundamental issue in the context of conservation laws and has attracted a lot of attentions. In particular, there have been a rather complete theory for the structure of classical shock waves. We mention the pioneer work by Gilbarg \cite{Gil} where the shock profile of the Navier-Stokes equations was constructed and the shock structure for small viscosity and heat-conductivity coefficients was also investigated. For system of conservation laws with uniform viscosity, a topological approach to construct the weak shock profile was discussed by Smoller and Conley \cite{SC1}. Later they adopted this approach to the MHD \cite{SC2}. In 1985, by the aid of center manifold theorem,  Majda-Pego \cite{M-P} built the connection between admissibility and structure of viscosity matrix and constructed the weak shock profile for a large class of viscous conservation laws. Their approach can be also used to deal with the relaxation shock structure problem, see \cite{YZ}. Indeed, our approach for the proof of Theorem \ref{thm1.1} is also based on the center manifold theorem as in \cite{M-P}.

Moreover, there have also been a huge number of literatures to study the structure of non-classical shocks in the context of combustion and MHD due to their physical significance and mathematical interests of analysis. In this direction, we mention two interesting works \cite{FS} and \cite{W}. In the end we point out that a nice and detailed introduction to the history of the shock structure problem can be found in the book by Dafermos \cite{D}.

%\end{subsection}

%\begin{subsection}{%Derivation for KdV-Burgers Structure
%Formal KdV-Burgers approximation}

\subsection{Formal KdV-Burgers approximation}

It is well-known (eg.~see Chapter 2 in \cite{LZ2}) that the propagation of shock profiles for viscous conservation laws such as the classical Navier-Stokes equations can be approximately described by the rather simple viscous Burgers equation in the weak shock regime. However, when the dispersive effect is involved, %including in,
the generation of dispersive plasma shock waves has been observed in physical experiments and investigated by numerical simulations, for instance, \cite{GsKG,NBS}.
%When the nonlinearity is balanced by a joint force
Due to the balance between dissipation and dispersion effects, plasma waves propagate like the profiles determined by the KdV-Burgers equations.
% with a monotone structure.
%For this and other reasons,
Motivated by this, we are further interested in understanding the structure of the obtained ion-acoustic shock waves through the KdV-Burgers approximation
%this kind of phenomenon
from the mathematical point of view.

We re-set up  the problem \eqref{1.2} in the regime where both viscosity coefficient $\mu$ and Debye length $\la$ are small and depend on a small parameter $\vep>0$ by %in the sense of
\begin{equation}
\label{def.scal}
\mu=\vep\b{\mu}, \quad %$ and $
\lambda=\vep^{1/2}\b{\lambda}
\end{equation}
%with $\vep>0$ in \eqref{1.2}
for two constants $\b{\mu}$ and $\b{\lambda}$ of the same order as a typical length. Then the propagation of the shock profile $[n_{\vep},u_{\vep},\phi_{\vep}]$ obtained in Theorem \ref{thm1.1} can be described by the following system of ODEs with rescaled viscosity and Debye length:
\begin{equation}\label{1.1.2-1}
\left\{
\begin{aligned}
&-s_{\varepsilon}\f{\dd n_{\varepsilon}}{\dd \xi}+\f{\dd(n_{\varepsilon} u_{\varepsilon})}{\dd \xi}=0,\\
&-s_{\varepsilon}\f{\dd(n_{\varepsilon} u_{\varepsilon})}{\dd \xi}+\f{\dd(n_{\varepsilon} u_{\varepsilon}^{2}+Tn_{\varepsilon})}{\dd \xi}=\varepsilon\bar{\mu} \f{\dd^2 u_{\varepsilon}}{\dd \xi^2}-n_{\varepsilon}\f{\dd \phi_{\varepsilon}}{\dd \xi},\\
&-\varepsilon\bar{\lambda}^2\f{\dd^2\phi_{\varepsilon}}{\dd \xi^2}=n_{\varepsilon}-e^{\phi_{\varepsilon}},
\end{aligned}
\right.
\end{equation}
supplemented by the corresponding far-field data
\begin{align}\label{1.1.2-2}
\lim_{\xi\rightarrow \pm \infty}[n_{\varepsilon}(\xi), u_{\varepsilon}(\xi), \phi_{\varepsilon}(\xi) ]=[n_{\varepsilon,\pm}, u_{\varepsilon,\pm}, \phi_{\varepsilon,\pm}].
\end{align}
%Here $\b{\mu}$ and $\b{\lambda}$ are of the same order as typical length.
For simplicity, we further denote
\begin{equation}
\label{def.ratio}
\delta=
\frac{\bar{\lambda}^2}{\bar{\mu}^2}
\end{equation}
and introduce a scaled variable $z=%\f{
%y
%\textcolor[rgb]{1.00,0.00,0.00}
{\xi}/%}{
\bar{\mu}
%}
$ to the end.
% in what follows.
Hence the system \eqref{1.1.2-1} can be equivalently rewritten as
% the following:
\begin{equation}\label{1.1.2-3}
\left\{
\begin{aligned}
&-s_{\varepsilon}n_{\varepsilon}'+(n_{\varepsilon} u_{\varepsilon})'=0,\\
&-s_{\varepsilon}(n_{\varepsilon} u_{\varepsilon})'+(n_{\varepsilon} u_{\varepsilon}^{2}+Tn_{\varepsilon})'=\varepsilon u_{\varepsilon}''-n_{\varepsilon}\phi_{\varepsilon}',\\
&-\varepsilon\delta\phi_{\varepsilon}''=n_{\varepsilon}-e^{\phi_{\varepsilon}},
\end{aligned}
\right.
\end{equation}
where $'$ stands for the differential operator $\f{\dd}{\dd z}$. Note that the R-H condition \eqref{1.6} for the far fields remains the same %does not change
in this formulation, and \eqref{1.1.2-3} is supplemented with the same far-field data as in \eqref{1.1.2-2}.  %regime. Inf

Formally, we assume to have the following expansion of $[n_{\varepsilon},u_{\varepsilon},\phi_{\varepsilon}]$ near
the upstream constant equilibrium $[n_{\varepsilon,-},u_{\varepsilon,-},\phi_{\varepsilon,-}]=[1,0,0]$: % at $x=-\infty$:
\begin{equation}\label{3.1.1}
\left\{\begin{aligned}
&n_{\varepsilon}=1+\varepsilon n_{1}+\varepsilon^{2}n_{2}+\cdots,\\
&u_{\varepsilon}=\varepsilon u_{1}+\varepsilon^{2}u_{2}+\cdots,\\
&\phi_{\varepsilon}=\varepsilon \phi_{1}+\varepsilon^{2}\phi_{2}+\cdots.%\nonumber%\\
\end{aligned}\right.
\end{equation}
To make the far fields %to be
compatible with the expansion in $\vep$, the shock speed
%of shock
$s_{\varepsilon}$ and the downstream constant equilibrium $[n_{\varepsilon,+},u_{\varepsilon,+},\phi_{\varepsilon,+}]$ are %is
also supposed to have the following asymptotic expansion:
\begin{equation}\label{3.1.2}
\left\{\begin{aligned}
s_{\varepsilon}&=\sqrt{T+1}+\vep s_1+\vep^2s_2+\cdots,\\
n_{\vep,+}&=1+\vep n_{1,+}+\vep^2 n_{2,+}+\cdots,\\
u_{\vep,+}&=\vep u_{1,+} +\vep^2 u_{2,+}+\cdots,\\
\phi_{\vep,+}&=\vep \phi_{1,+}+\vep^2 \phi_{2,+}+\cdots.
\end{aligned}\right.
\end{equation}
Here $\sqrt{T+1}$ in the first equation of \eqref{3.1.2} is equal to
the acoustic speed of the second family of characteristic field %valued
at $[n,u]=[1,0]$ for the quasi-neutral Euler system \eqref{1.7}, and it is the exact %a natural
leading term of the asymptotic expansion of $s_{\vep}$ as $\vep\to0$.
%Due to \eqref{3.1.2}, the amplitude of shock profile is indeed assumed to be of $O(\vep).$ Therefore, without loss of generality, we choose
For brevity, in what follows we shall make a simple choice of $s_\vep$ by
\begin{equation}\label{3.1.3}
%\varepsilon=|s_{\vep}-\sqrt{T+1}|=\sqrt{T+1}-s_{\vep}>0.
s_\vep=\sqrt{T+1}-\vep,
\end{equation}
and parametrize $[n_{\varepsilon,+},u_{\varepsilon,+},\phi_{\varepsilon,+}]$ in terms of $\vep>0$ through the R-H condition \eqref{1.6}.
Here, due to $\vep>0$, one has $s_\vep<\sqrt{T+1}$, which is consistent with
%the sign in \eqref{3.1.3} is due to
the compressibility property of shock profiles. In fact, by substituting \eqref{3.1.3} into the R-H condition \eqref{1.6}, one can parametrize the downstream data $n_{\vep,+}$, $u_{\vep,+}$ and $\phi_{\vep,+}$ in terms of $\vep$ as follows:
%the following
%functions of $\vep$ in the following way:
\begin{equation}\label{3.1.4}
\left\{
\begin{aligned}
n_{\vep,+}&=\f{s^2_{\vep}}{T+1}=1-\vep\left(\f{2}{\sqrt{T+1}}-\f{\vep}{T+1}\right),\\
u_{\vep,+}&=s_\vep(1-\f1{n_{\vep,+}})=-\vep\left(2+\f{\vep}{\sqrt{T+1}-\vep}\right),\\
\phi_{\vep,+}&=\log n_{\vep,+}=-\vep\left(\f2{\sqrt{T+1}}-a_\vep\right).
\end{aligned}
\right.
\end{equation}
Here $a_\vep$ in the last line of  \eqref{3.1.4} is denoted by
\begin{equation*}
%\label{ }
a_\vep=\frac{1}{\vep}\log\f{s^2_{\vep}}{T+1}+\f{2}{\sqrt{T+1}}.
\end{equation*}
By \eqref{3.1.3}, it is straightforward to verify that
\begin{equation*}
%\label{ }
a_\vep=\frac{2}{\vep}\log \left(1-\frac{\vep}{\sqrt{T+1}}\right)+\frac{2}{\sqrt{T+1}} =-\frac{\vep}{T+1}-O(1)\vep^2.
\end{equation*}

Now we are in the position of deriving the equations for $[n_1,u_1,\phi_1]$ corresponding to the first-order terms in \eqref{3.1.1}. In fact, substituting \eqref{3.1.1} into \eqref{1.1.2-3} immediately yields to the vanishing zeroth order in $\vep$ and the subsequent orders as follows:
\begin{align}
\varepsilon^{1}:&-\sqrt{T+1}n_{1}'+u_{1}'=0,\label{3.1.5-1}\\
&-\sqrt{T+1}u_{1}'+Tn_{1}'=-\phi_{1}',\label{3.1.5-2}\\
&n_{1}-\phi_{1}=0,\label{3.1.5-3}\\
\varepsilon^{2}:&-\sqrt{T+1}n_{2}'+u_{2}'+n_{1}'+(n_{1}u_{1})'=0,\label{3.1.5-4}\\
&-\sqrt{T+1}u_{2}'+T n_{2}+u_1'-\sqrt{T+1}(n_{1}u_{1})'+2u_{1}u_1'\notag\\
&\qquad\qquad\qquad=-\phi_{2}'+ u_{1}''-n_{1}\phi_{1}',\label{3.1.5-5}\\
&n_{2}-\phi_{2}-\frac{1}{2}\phi_{1}^{2}=-\delta\phi_{1}'',\label{3.1.5-6}\\
\vep^3:&\cdots. \nonumber
\end{align}
From \eqref{3.1.5-1}, \eqref{3.1.5-2} and \eqref{3.1.5-3}, we solve $u_1$ and $\phi_1$ in terms of $n_1$ as %follows.
\begin{equation}\label{3.1.6}
u_{1}=\sqrt{T+1}n_{1},\qquad \phi_{1}=%-n_1
%\textcolor[rgb]{1.00,0.00,0.00}
{n_1}.
\end{equation}
To further derive the equation for $n_1$, we differentiate \eqref{3.1.5-6} with respect to $z$, then multiply \eqref{3.1.5-4} by $\sqrt{T+1}$, and further add these two resultant equations to \eqref{3.1.5-5}, so it follows that  %Then it holds that
\begin{align}\label{3.1.7}
\sqrt{T+1}n_1'+u_1'+2u_1u_1'-u_1''+n_1\phi_1'-\phi_1\phi_1'+\delta\phi_1'''=0.
\end{align}
Substituting \eqref{3.1.6} into \eqref{3.1.7}, one derives the equation for $n_1$:
\begin{equation}\label{3.1.8-1}
2\sqrt{T+1}n_1'+2(T+1)n_1n_1'-\sqrt{T+1}n_1''+\delta n_1'''=0,
\end{equation}
with the far fields
\begin{align}\label{3.1.8-1.1}
\lim_{z\rightarrow+\infty}n_{1}(z):=n_{1,+}=-\frac{2}{\sqrt{T+1}},\quad
\lim_{z\rightarrow -\infty}n_1(z):=n_{1,-}=0.
\end{align}
%according to \eqref{3.1.4}.
Note that \eqref{3.1.8-1.1} above matches \eqref{3.1.4} at the first order of $\vep$.
By \eqref{3.1.6} and \eqref{3.1.4} again, the equations for $u_1$ and $\phi_1$ are given by
\begin{equation}\label{3.1.8-2}
2u_1'+2u_1u_1'-u_1''+\f{\delta}{\sqrt{T+1}}u_1'''=0, \end{equation}
with
\begin{align}\label{3.1.8-2.1}
\lim_{z\rightarrow +\infty}u_1(z):=u_{1,+}=-2,\quad \lim_{z\rightarrow -\infty}u_1(z):=u_{1,-}=0,
\end{align}
and
\begin{equation}\label{3.1.8-3}
\sqrt{T+1}\phi_1'+2(T+1)\phi_1\phi_1'-\sqrt{T+1}\phi_1''+\delta \phi_1'''=0,
\end{equation}
with
\begin{align}\label{3.1.8-3.1}
\lim_{z\rightarrow +\infty}\phi_1(z):=\phi_{1,+}=-\f{2}{\sqrt{T+1}},\quad\lim_{z\rightarrow-\infty}\phi_1(z):=\phi_{1,-}=0,
\end{align}
respectively.

Note that if the dissipation is dominated, namely,
$
\delta=\b{\lambda}^2/\b{\mu}^2
$ is small enough,
then the unique (up to a shift) monotone shock profiles $n_1$, $u_1$ and $\phi_1$ of the KdV-Burgers equations
%i.e., \eqref{3.1.8-1} and \eqref{3.1.8-1.1} for $n_1$, \eqref{3.1.8-2} and \eqref{3.1.8-2.1} for $u_1$, and \eqref{3.1.8-3} and \eqref{3.1.8-3.1} for $\phi_1$, %has been
can be constructed as in \cite{BS}. For completeness, we will list the related results in Lemma \ref{lem3.1.0} in the Appendix.
%\end{subsection}

%\begin{subsection}{Justification of KdV Burgers approximation}
\subsection{Rigorous justification of KdV Burgers approximation}

To make a rigorous justification of  the KdV-Burgers approximation to the ion-acoustic shock profiles obtained in Theorem \ref{thm1.1}, we start from the rescaled system \eqref{1.1.2-3} with the far fields \eqref{1.1.2-2}, where we have chosen $s_\vep$ as in \eqref{3.1.3}, the upstream equilibrium $[n_{\varepsilon,-},u_{\varepsilon,-},\phi_{\varepsilon,-}]=[1,0,0]$, and  the downstream equilibrium $[n_{\varepsilon,+},u_{\varepsilon,+},\phi_{\varepsilon,+}]$ as in \eqref{3.1.4}.
Note that by comparing the far fields $[n_{1,+},u_{1,+},\phi_{1,+}]$ to $[n_{\vep,+},u_{\vep,+},\phi_{\vep,+}]$, one has
\begin{equation*}
%\label{}
\frac{1}{\vep^2}\left\{[n_{\vep,+},u_{\vep,+},\phi_{\vep,+}]-[1,0,0]-\vep [n_{1,+},u_{1,+},\phi_{1,+}]\right\}
=[\f{1}{T+1},-\f{1}{\sqrt{T+1}-\vep},\frac{a_\vep}{\vep}]
\end{equation*}
%Here we  notice that
with
\begin{equation*}
%\label{ }
%\phi_{R,+}=\vep^{-2}\{\log\f{s_{\vep}^2}{T+1}+
%\f{2\vep}{\sqrt{T+1}}\}.
\frac{a_\vep}{\vep}=-\frac{1}{T+1}-O(1)\vep.
\end{equation*}
%$$
%n_{\vep,+}=1+\vep n_{1,+}+\vep^2 n_{R,+},\quad u_{\vep,+}=\vep u_{1,+}+\vep^2 u_{R,+},\quad\phi_{\vep,+}=1+\vep \phi_{1,+}+\vep^2 \phi_{R,+}, $$
%with
%$$[n_{R,+},u_{R,+},\phi_{R,+}]=[\f{1}{T+1},-\f{1}{\sqrt{T+1}-\vep},\vep^{-2}\{\log\f{s_{\vep}^2}{T+1}+
%\f{2\vep}{\sqrt{T+1}}\}]\sim O(1).
%$$
%To justify the limit to  the KdV-Burgers equations for the ion-acoustic shock profile obtained from \eqref{1.1.2-3},
Therefore, one can see that it may not be a good ansatz to directly take $[1,0,0]+\vep [n_{1},u_{1},\phi_{1}]$ as the approximation of $[n_{\vep},u_{\vep},\phi_{\vep}]$ up to the first order for making the energy estimates on remainders in $L^2$ setting, because their far-field data cannot be matched. To overcome this trouble, we introduce the modified first-order approximation
%Therefore, we introduce the following corrections
$[n_{1,\vep},u_{1,\vep},\phi_{1,\vep}]$ satisfying
%to satisfy the far fields conditions.
\begin{equation}\label{3.1.9}
\left\{
\begin{aligned}
&(2\sqrt{T+1}-\vep)n_{1,\vep}'+2(T+1)n_{1,\vep}n_{1,\vep}'-\sqrt{T+1}n_{1,\vep}''+\delta n_{1,\vep}'''=0,\\
&\left(2+\f{\vep}{\sqrt{T+1}-\vep}\right)u_{1,\vep}'+2u_{1,\vep}u_{1,\vep}'-u_{1,\vep}''+\f{\delta}{\sqrt{T+1}}u_{1,\vep}'''=0,\\
&\left(2\sqrt{T+1}-a_\vep (T+1)\right)\phi_{1,\vep}'+2(T+1)\phi_{1,\vep}\phi_{1,\vep}'-\sqrt{T+1}\phi_{1,\vep}''+
\delta\phi_{1,\vep}'''=0,
\end{aligned}
\right.
\end{equation}
with the far-field data:
\begin{equation}
\label{ad-ffdl}
\lim\limits_{z\to -\infty}[n_{1,\vep},u_{1,\vep},\phi_{1,\vep}](z)=[0,0,0],
\end{equation}
and
%\begin{equation}
\begin{multline}
\lim\limits_{z\to \infty}[n_{1,\vep},u_{1,\vep},\phi_{1,\vep}](z)=\frac{1}{\vep}\{[n_{\vep,+},u_{\vep,+},\phi_{\vep,+}]-[1,0,0]\}\\
=[-\frac{2}{\sqrt{T+1}}+\frac{\vep}{T+1},-2-\f{\vep}{\sqrt{T+1}-\vep},-\frac{2}{\sqrt{T+1}}+a_\vep],\label{ad-ffdr}
\end{multline}
%\end{equation}
in terms of \eqref{3.1.4}. Note that compared to \eqref{3.1.8-1}, \eqref{3.1.8-2} and \eqref{3.1.8-3}, we have modified  the coefficients of $n_{1,\vep}'$, $u_{1,\vep}'$ and  $\phi_{1,\vep}'$ in \eqref{3.1.9}, respectively, according to the far-field conditions \eqref{ad-ffdl} and \eqref{ad-ffdr}.
Moreover, since the shock profile is invariant under a spatial shift, we further set
\begin{equation*}%\label{ }
n_{1,\vep}(0)=n_1(0),\ u_{1,\vep}(0)=u_1(0),\ \phi_{1,\vep}(0)=\phi_1(0)
\end{equation*}
without loss of generality.

We seek for the shock profile solution %to \eqref{1.1.2-3}
in the form:
\begin{equation}\label{3.1.10}
\left\{\begin{aligned}
&n_{\vep}=1+\vep n_{1,\vep}+\vep^2 n_{R}=1+\vep n_1+\vep^2(n_2+n_R),\\
&u_{\vep}=\vep u_{1,\vep}+\vep^2 u_{R}=\vep u_{1}+\vep^2(u_2+u_R),\\
&\phi_{\vep}=\vep\phi_{1,\vep}+\vep^2\phi_{R}=\vep\phi_1+\vep^2(\phi_2+\phi_R),
\end{aligned}\right.
\end{equation}
where $[n_2,u_2,\phi_2]$ is defined by
\begin{align}\label{n2}
[n_2,u_2,\phi_2]:=\vep^{-1}[n_{1,\vep}-n_1,u_{1,\vep}-u_1,\phi_{1,\vep}-\phi_1].
\end{align}
Note that $[n_2,u_2,\phi_2]$ is %of order
$O(1)$ in terms of Lemma \ref{lm3.1.1} in the Appendix, and
$$
\lim\limits_{z\to\pm \infty} [n_R,u_R,\phi_R]=[0,0,0].
$$
The key point is to establish uniform-in-$\vep$ estimates for the remainder $[n_R,u_R,\phi_R]$. For this purpose, we first derive the equation for $[n_R,u_R,\phi_R]$ from \eqref{1.1.2-3} as follows. In fact, integrating the first equation of $\eqref{1.1.2-3}$ from $-\infty$ to $z$ yields that
\begin{align}\label{3.1.11}
u_{\vep}=\f{s_{\vep}(n_{\vep}-1)}{n_{\vep}}.
\end{align}
Then from \eqref{3.1.11}, one can solve $u_R$ in terms of $n_R$ as
\begin{align}\label{3.1.u}
u_R=n_{\vep}^{-1}(s_{\vep}-\vep u_{1,\vep})n_{R}+\vep^{-1}n^{-1}_{\vep}(s_{\vep}n_{1,\vep}-u_{1,\vep}-\vep n_{1,\vep}u_{1,\vep}).
\end{align}
Similar for obtaining \eqref{1.4}, $[n_{\vep},u_{\vep},\phi_{\vep}]$ also satisfies the following system with two conservation laws:
\begin{equation}\label{3.0.3-1}
\left\{
\begin{aligned}
&-s_{\vep}n_{\vep}'+(n_{\vep}u_{\vep})'=0,\\
&-s_{\vep}(n_{\vep}u_{\vep})'+(n_{\vep}u_{\vep}^2+Tn_{\vep})'=\vep
%\textcolor[rgb]{1.00,0.00,0.00}{\cancel{\b{\mu}}}
u_{\vep}''+
\big(\f{1}{2}\vep\delta(\phi_{\vep}')^2-e^{\phi_{\vep}}\big)',\\
&-\vep\delta\phi_{\vep}''=n_{\vep}-e^{\phi_{\vep}}.
\end{aligned}\right.
\end{equation}
Substituting \eqref{3.1.11} into the second equation of \eqref{3.0.3-1} and integrating the resultant equation, one has
\begin{equation}\label{3.1.12}
\vep s_{\vep}n_{\vep}^{-2}n_{\vep}'=(T+1-s_{\vep}^2)(n_{\vep}-1)+s_{\vep}^2n^{-1}_{\vep}(n_{\vep}-1)^2
-\f{1}{2}\vep\delta(\phi_{\vep}')^2+\vep\delta\phi_{\vep}''.
\end{equation}
Here we have used the third equation of \eqref{3.0.3-1} to replace $e^{\phi_{\vep}}$ by $n_{\vep}+\vep\delta\phi_{\vep}''$. Plugging \eqref{3.1.10} into  equation \eqref{3.1.12} and the third equation of \eqref{3.0.3-1} simultaneously, one has the following system for $[n_{R},\phi_{R}]$:
\begin{equation}\label{3.1.13}
\left\{
\begin{aligned}
n_{R}'&=2\left(1+\sqrt{T+1}n_{1}(z)\right)n_{R}+\f{\delta}{\sqrt{T+1}}\phi_{R}''+r_1+r_2+r_3,\\
-\vep\delta\phi_{R}''&=n_{R}-\phi_{R}+r_{4}+r_{5}+r_{6},
\end{aligned}
\right.
\end{equation}
where the inhomogeneous terms $r_i$ ($1\leq i\leq 6$) are given by %\marginpar{\gr{Missed two $\delta$ below.}}
%\textcolor[rgb]{1.00,0.00,0.00}{
\begin{equation}\label{def.r1-6}
\left\{\begin{aligned}
&r_1:=(\frac{1}{\sqrt{T+1}}+n_{1,\vep})n_{1,\vep}'+\f{\delta(1+\vep n_{1,\vep})}{\vep\sqrt{T+1}}\{(1+\vep n_{1,\vep})\phi_{1,\vep}''-n_{1,\vep}''\}\\
&\qquad\ -\f{\delta(1+\vep n_{1,\vep})^2}{2\sqrt{T+1}}(\phi_{1,\vep}')^2,\\
&r_2=r_2[n_R,\phi_R]=O(\vep)\{|n_R|+{\delta}|\phi_R''|+|\phi_R'|+|n_R'|\},\\
&r_{3}=r_{3}[n_{R},\phi_R]
%O(1)\vep\{n_R^2(1+n_R)+n_R(\phi_R''+\phi_R')+(\phi_R')^2(1+n_R+n_R^2)\},\\
=O(\vep)|n_{R}|^2+O(\vep^2)\left\{|n_R|^2\cdot \big[|n_R|+|\phi_R'|+{\delta}|\phi_R''|\big]\right.\\
&\qquad\qquad\qquad\qquad\qquad\left.+{|n_R|\cdot\big[|\phi_R'|+\delta|\phi_R''|\big]}+|\phi_R'|^2\cdot \big[{1}+|n_R|^2\big]\right\},\\
&r_4=\vep^{-2}(1+\vep n_{1,\vep}-e^{\vep \phi_{1,\vep}})=O(1),\\%\quad
&r_5=r_5[\phi_R]:=(1-e^{\vep\phi_{1,\vep}})\phi_R=O(\vep) |\phi_R|,\\
&r_6=r_6[\phi_{R}]:=\vep^{-2}e^{\vep\phi_{1,\vep}}(1-e^{\vep^2\phi_R}+\vep^2\phi_R)=O(\vep^2)|\phi_R|^2.
\end{aligned}\right.
\end{equation}
%}
%\textcolor[rgb]{1.00,0.00,0.00}
For brevity of presentation, we put the explicit formulas of $r_2$ and $r_3$ into the Appendix, see \eqref{5.1} and \eqref{5.2} respectively. Moreover, one can see that solutions to the remainder system \eqref{3.1.13} are not unique due to the translation invariance of the shock profile.  Thus, to the end we set $n_{R}(0)=0$ without loss of generality.

Define a weight function
$$
w_{\alpha}=w_{\alpha}(z)=\exp\left\{\alpha\sqrt{1+|z|^{2}}\right\},
$$
for $\alpha>0$, and denote the weighted Sobolev space $H^{k}_{\alpha}$ for $k=0,1,\cdots$ as
\begin{equation*}%\label{Norm}
H^{k}_{\alpha}=\left\{f=f(z)\in H^{k}\left|w_{\alpha}\frac{\dd^i f}{\dd z^i}\in L^{2}, \ 0\leq  i\leq k\right.\right\},
\end{equation*}
associated with the norm
$$
\|f\|_{H^{k}_{\alpha}}=\left\{\sum_{i=0}^{k}\left\|w_{\alpha}\frac{\dd^i f}{\dd z^i}\right\|_{L^2}^2\right\}^{1/2}.
$$
%\textcolor[rgb]{1.00,0.00,0.00}
%{Let $k\geq 2$ be an integer. For any $\alpha>0$, $\vep>0$ and $\delta>0$,
For an integer $k\geq 2$ and $0<\alpha<2$,
we also define the following solution space for the remainder equations \eqref{3.1.13}:
\begin{align}\label{space}
\mathbf{X}_{\alpha,k}=\bigg\{U(z)=[n(z),\phi(z)]\bigg| n(0)=0, \|n\|_{H^{k}_{\alpha}}+\|\phi\|_{H^{k+2}_{\alpha}}<\infty\bigg\}
\end{align}
with the norm
\begin{align*}%\label{space1}
\|U\|_{\mathbf{X}_{\alpha,k}}\doteq \|[n,\phi]\|_{H^{k}_{\alpha}}+\sqrt{\varepsilon\delta}\left\|\f{\dd^{k+1}\phi}{\dd z^{k+1}}\right\|_{L^{2}_{\alpha}}+\varepsilon\delta\left\|\f{\dd^{k+2}\phi}{\dd z^{k+2}}\right\|_{L^{2}_{\alpha}}
\end{align*}
%}

With the above notations on hand, the main result concerning  the shock structure in terms of the KdV-Burgers approximation is stated as follows.

\begin{theorem}\label{thm1.2}
Let $T\geq 0$
%$\b{\mu}>0$, %and $\b{\lambda}>0$  with
%$\b{\lambda}>0$
and $0<\alpha<2$. There exist positive constants $\vep_0$, $\delta_0$, $\hat{c}_1$ and $\hat{c}_2$ such that if
\begin{equation}
\label{ad.ass.re}
0<\vep\leq\vep_0, \quad 0<\delta=\f{\b{\lambda}^2}{\b{\mu}^2}\leq \delta_0, \quad \hat{c}_1\vep\leq 1-n_{\vep,+}\leq \hat{c}_2\vep,
\end{equation}
%Assume that
%\begin{equation}
%\label{ad.ass.re}
%0<\delta=\f{\b{\lambda}^2}{\b{\mu}^2}\ll1.
%and $\hat{c}_1\vep\leq 1-n_{\vep,+}\leq \hat{c}_2\vep$ for two positive constants ${\hat{c}_1}$ and ${\hat{c}_2}$,
then \eqref{1.1.2-3} with \eqref{1.1.2-2} admits a unique shock profile solution $[n_{\vep},u_{\vep},\phi_{\vep}]$ of the form %there exists a unique shock profile solution $[n_{\vep},u_{\vep},\phi_{\vep}]$ to \eqref{1.1.2-3} with \eqref{1.1.2-2}.
%Furthermore,
%for each integer $k\geq 2$, there is a constant %$C>0$
%$C_k>0$ independent of $\delta$ and $\vep$
%%\red{$\cancel{0<\beta<2}$}
%such that
%the shock profile $[n_{\vep},u_{\vep},\phi_{\vep}]$ can be represented as
\begin{align}\label{1.1.3}
[n_{\vep},u_{\vep},\phi_{\vep}]=[1,0,0]+\vep[n_1,u_1,\phi_1]+\vep^2[n_2+n_R,u_2+u_R,\phi_2+\phi_R],
\end{align}
where the first-order KdV-Burgers shock profiles $n_1$, $u_1$ and $\phi_1$ solve \eqref{3.1.8-1}, \eqref{3.1.8-2} and \eqref{3.1.8-3} respectively, the second-order correction $[n_2,u_2,\phi_2]$ is defined by \eqref{n2} in terms of $[n_1,u_1,\phi_1]$ and its modified approximation $[n_{1,\vep},u_{1,\vep},\phi_{1,\vep}]$ given by \eqref{3.1.9}, \eqref{ad-ffdl} and \eqref{ad-ffdr}, and the remainder $[n_R,u_R,\phi_R]$ is solved by \eqref{3.1.13} and %, and $u_R$ is defined in
\eqref{3.1.u} satisfying the estimates:
%\textcolor[rgb]{1.00,0.00,0.00}{
\begin{equation}
\label{thm1.2.estnp}
\left\|[n_R,\phi_R]\right\|_{\mathbf{X}_{\alpha,k}}\leq {C_k},
\end{equation}
and  %with
\begin{equation}
\label{thm1.2.estu}
\|u_R\|_{H^k_{\alpha}}\leq {C_k},
\end{equation}
for any integer $k\geq 2$, where each $C_k>0$ is a  generic constant independent of $\vep$ and $\delta$.
\end{theorem}

%
%\Red{ZHU: As we agreed, we better add a remark to clarify that Theorem \ref{thm1.2} is also true in the case of $T=0$. However, it is still unknown for us to prove the results similar to Theorem \ref{thm1.1} and Theorem \ref{thm1.3} when $T=0$. RJ}

%\textcolor[rgb]{1.00,0.00,0.00}{
\begin{remark}
From \eqref{thm1.2.estnp} and \eqref{thm1.2.estu} we establish the uniform-in-$\vep$ estimates on derivatives of remainder $[n_R,u_R,\phi_R]$ up to any order, which justifies the KdV-Burgers approximation of the ion-acoustic shock profile for $\vep>0$ small enough with fixed $\delta>0$. Notice that our estimates are also uniform in $\delta>0$. Hence one can recover the classical Burgers approximation to the shock profile of the Navier-Stokes equations by letting $\delta\to 0+$ for fixed suitably small $\vep>0$.
\end{remark}
%}

\begin{remark}
The assumption \eqref{ad.ass.re}
%$\delta=\f{\b{\lambda}^2}{\b{\mu}^2}\ll1$
 implies that the dissipation dominates over the dispersion, which leads to a monotone structure of smooth shock waves, see
%\eqref{3.1.9-0} in
 Lemma \ref{lem3.1.0}. This is also consistent with the case of
 %Note that even for
 the KdV-Burgers equation for which the stable monotone shock profile exists only when the ratio $\b{\lambda}/\b{\mu}$ is suitably small, cf.~\cite{BS,Pego}. On the other hand, when the dispersion surpasses the dissipation, the oscillatory transition can be also observed in the plasma shock wave, cf.~\cite{NBS,GsKG}. In fact, it was proved
 %by Bona and Schonbek
in \cite{BS} that the shock for the KdV-Burgers equation has a one-side oscillatory tail when $\delta$ is bigger than a threshold value. Thus it may be interesting to study whether one can construct the shock wave for the Navier-Stokes-Poisson system with an oscillatory structure through an approximation by the KdV-Burgers equations.
\end{remark}

\begin{remark}
Due to the fact that $[n_1,u_1,\phi_1]$ does not satisfy the far fields condition, a layer correction could be required to be additionally included if one would expect a much more accurate approximation beyond the KdV-Burgers equations. This will be left for our future work.
\end{remark}

%\Red{ZHU: Can you re-write this paragraph by adding more words to explain the technique and idea for the proof of Theorem \ref{thm1.3}. RJ}
We remark that the problem considered in Theorem \ref{thm1.2} can be regarded as an analogy of the singular perturbation problem. In fact, there has been some interesting studies to analyze the structure of shock profiles for the MHD or the combustion model with some small physical parameters from a dynamical system point of view. We list some related works \cite{BHLZ,FS,GSW} and references therein. In the end we also point out that when there is no viscosity, the KdV approximation to the Euler-Poisson equations was studied by Guo and Pu \cite{GPu}. %\\

%\begin{subsection}{Dynamical stability of shock profile}
\subsection{Dynamical stability of shock profile}

We shall also investigate the large time asymptotic stability of the smooth travelling shock profile obtained in Theorem \ref{thm1.1}. For convenience, we  formulate the Navier-Stokes-Poisson system in the {\it Lagrangian} coordinates which reads
\begin{equation}\label{4.2}
\left\{\begin{aligned}
&\pa_{t}v-\pa_{x}u=0,\\
&\pa_tu+T\pa_{x}v^{-1}=\mu\pa_{x}(v^{-1}\pa_{x}u)-v^{-1}\pa_{x}\phi,\\
&-\lambda^2\pa_{x}(v^{-1}\pa_{x}\phi)=1-ve^{\phi},\quad  t>0,x\in \mathbb{R}.
%\\
%&\lim_{x\rightarrow \pm\infty }\phi(t,x)=\phi_{\pm}, \qquad t>0,x\in \mathbb{R}.
\end{aligned}\right.
\end{equation}
Here, $v=\frac{1}{n}$ is the specific volume. We keep on using variables $t$ and $x$ in the Lagrangian coordinates for brevity.
%$u$ is velocity and $\phi$ is electric potential.
Initial data are given by
\begin{align}\label{4.3}
[v,u](0,x)=[v_0,u_0](x)\rightarrow[v_{\pm},u_{\pm}]\quad (x\to \pm\infty),
\end{align}
and the far fields of $\phi(t,x)$ are given by
\begin{equation}
\label{sta-pf}
\lim_{x\rightarrow \pm\infty }\phi(t,x)=\phi_{\pm}.
\end{equation}
%as $x\rightarrow \pm\infty$.
Similar to the  case of Eulerian coordinates, we can also write the momentum equation into a conservative form. In fact, multiplying the third equation of \eqref{4.2} by $v^{-1}\pa_x\phi$ gives
% and rewrite the non-conserved term $v^{-1}\partial_x \phi$ in $\eqref{4.2}_2$ as
\begin{align}
v^{-1}\pa_x\phi=\big[-\lambda^2\pa_x(v^{-1}\pa_x\phi)+ve^{\phi}\big]v^{-1}\pa_x\phi=\pa_x\left[-\f{\lambda^2}{2}(v^{-1}\pa_x\phi)^2+e^{\phi}\right].\nonumber
\end{align}
Substituting the above identity into the second equation of \eqref{4.2},  the system for $[v,u,\phi]$ is rewritten as
\begin{equation}\label{4.3c}
\left\{\begin{aligned}
&\pa_{t}v-\pa_{x}u=0,\\
&\pa_tu+(T+1)\pa_{x}v^{-1}=\mu\pa_{x}(v^{-1}\pa_{x}u)+%\pa_x\left[\f{\lambda^2}{2}(v^{-1}\pa_x\phi)^2-e^{\phi}\right],\\
%\textcolor[rgb]{1.00,0.00,0.00}
{\lambda^2\pa_x\left[\f{(v^{-1}\pa_x\phi)^2}{2}-v^{-1}\pa_x(v^{-1}\pa_x\phi)\right]},\\
&-\lambda^2\pa_{x}(v^{-1}\pa_{x}\phi)=1-ve^{\phi}.
\end{aligned}\right.
\end{equation}
From \eqref{4.3c}, the R-H %Rankine-Hugoniot
condition is given by
\begin{equation}\label{4.3-1}
\left\{
\begin{aligned}
&-s(v_+-v_-)-(u_+-u_-)=0,\\
&-s(u_+-u_-)+(T+1)(v_+^{-1}-v_-^{-1})=0.
\end{aligned}\right.
\end{equation}
Since we  concern only the 2-shock profile in the paper, we assume the  compressibility $v_+>v_-$ in what follows.

In the context of gas dynamics and kinetic theory, the equivalence of shock profiles in the Eulerian coordinates and Lagrangian coordinates has been shown in \cite{HWWY}. The case of the Navier-Stokes-Poisson system is similar. Therefore, we re-state the properties of the shock profile  in the Lagrangian coordinates  in terms of Theorem \ref{thm1.1} as follows.

\begin{proposition}\label{prop1.5}
{Let $T\geq 0$.}
For given $[v_-,u_-]$ with $v_->0$, there exists  a positive constant %$\tilde{\delta}_1$
{$\tilde{\vep}$} such that if $[v_+,u_+]$ satisfies \eqref{4.3-1} with $0<v_+-v_-\leq {\tilde{\vep}}$
%\tilde{\delta}_1$,
then \eqref{4.2} has a unique (up to a shift) shock profile solution $[\b{v},\b{u},\b{\phi}](x-st)$ connecting the far fields $[v_{\pm},u_{\pm},\phi_{\pm}]$ with $\phi_{\pm}=-\log v_{\pm}$. Moreover, there exist  positive constants $\tilde{C}_1$, $\tilde{C}_2$, $C_k$ %and $\theta_k$
($k=0,1,2,\cdots$) and $\theta$, such that for each $t>0$ and $x\in \mathbb{R}$,
\begin{equation}\label{4.4}
\left\{
\begin{aligned}
&s\b{v}_x=-\b{u}_x>0, \quad -\tilde{C}_1\b{v}_x\leq\b{\phi}_{x}\leq -\tilde{C}_2\b{v}_x, \\
&\left|\f{
\dd^k}{\dd x^k}\left[\b{v}-v_{\pm},\b{u}-u_{\pm},\b{\phi}-\phi_{\pm}\right]\right|\leq C_k|v_+-v_-|^{k+1}e^{-{\theta}|v_+-v_-|\cdot |x-st|},\  x-st\lessgtr0.
\end{aligned}\right.
\end{equation}
\end{proposition}

%\textcolor[rgb]{1.00,0.00,0.00}
{Define new coordinates $(t,y)\doteq (t,x-st)$ and
formally introduce the anti-derivative variables as
\begin{equation}\label{1.1.4}
\left\{
\begin{aligned}
&\Phi(t,y)=\int_{-\infty}^{y}(v(t,y')-\b{v}(y'))\dd y',\\
&\Psi(t,y)=\int_{-\infty}^{y}(u(t,y')-\b{u}(y'))\dd y',
\end{aligned}
\right.
\end{equation}
and set
\begin{equation}
\label{def.id.PP}
\Phi_0(y)=\int_{-\infty}^y (v_0-\b{v})(y')\dd y',\quad \Psi_0(y)=\int_{-\infty}^y (u_0-\b{u})(y')\dd y'.
\end{equation}}
The main result about the dynamical stability of shock profiles is stated as follows.

\begin{theorem}\label{thm1.3}
%Assume
Let $T>0$. Let $[v_{\pm},u_{\pm}]$ satisfy \eqref{4.3-1} with $\phi_{\pm}=-\log v_{\pm}$, and let the shock profile $[\b{v},\b{u},\b{\phi}]$ with the shock speed $s>0$ be obtained as in Proposition \ref{prop1.5}.  Assume
$
[v_0-\b{v},u_0-\b{u}]\in L^1\cap H^1
$
with
$$
\int_{-\infty}^{+\infty}(v_0-\b{v})(y)\dd y=\int_{-\infty}^{+\infty}(u_0-\b{u}) (y)\dd y=0.
$$
Let
$$
E_0\equiv\|[v_0-\b{v},u_0-\b{u}]\|_{H^1}^2+\|[\Phi_0,\Psi_0]\|_{L^2}^2.
$$
There exist positive constants %$\varepsilon_1$, $\delta_1$ and $C$
$\tilde{\vep}_0$ and $e_0$ such that if
\begin{equation*}
%\label{ }
{0<v_+-v_-}\leq {\tilde{\vep}_0}%\delta_1
,\quad E_0\leq{e_0},
%\varepsilon_1,
\end{equation*}
then the Cauchy problem \eqref{4.2} together with \eqref{4.3} and \eqref{sta-pf} admit a unique global-in-time solution $[v,u,\phi](t,x)$ with
\begin{equation}
\label{thm1.3.est}
\CE(t)+\int_0^t\CD(\tau)\dd \tau\leq CE_0,
\end{equation}
for all $t\geq 0$,
where $\CE(t)$ and $\CD(t)$ are respectively denoted by
\begin{align}\label{E}
\CE(t)\equiv\|[\Phi,\Psi,\phi-\b{\phi}](t)\|_{H^2}^2
\end{align}
and
\begin{equation}\label{D}
\CD(t)\equiv\|%\sqrt{s\b{v}\bar{v}_y}
(s\b{v}\bar{v}_y)^{\frac{1}{2}}\Psi(t)\|_{L^2}^2+
\|[\Phi_y,\phi_t-\b{\phi}_t](t)\|_{H^1}^2+\|[\Psi_y,\phi-\b{\phi}](t)\|_{H^2}^2.
%+\|(\phi-\b{\phi})(t)\|_{H^2}^2
%+\|(\phi-\b{\phi})_t(t)\|_{H^1}^2.
\end{equation}
Moreover, the solution $[v,u,\phi]$
tends in large time to the shock profile  in the following sense
\begin{align}\label{LT}
\lim_{t\rightarrow\infty}\sup_{y\in \mathbb{R}}\Big|\left[v(t,y)-\b{v}(y),u(t,y)-\b{u}(y),\phi(t,y)-\b{\phi}(y)\right]\Big|=0.
\end{align}
\end{theorem}

\begin{remark}
The dynamical stability of the ion-acoustic shock profile in the case of $T=0$ remains left, see \eqref{ad.qfT} in the proof of Lemma \ref{lem4.2}.
\end{remark}

Besides the existence and structure stability, the dynamical stability of shock waves is another fundamental issue in viscous conservation laws and has attracted lots of mathematical interests. There have been extensive studies about this problem. Here we mention only a few related to our interest. In the context of gas dynamics, it was first studied by Matsumura and Nishihara \cite{MN} for the isentropic Navier-Stokes equations and later extended by Kawashima and Matsumura \cite{KM} to the heat-conductive case. Goodman \cite{G} developed a weighted energy method to treat the difficulties induced by interaction of the shock profile with waves in other families, and succeeded in proving the asymptotic stability of shock profiles for a general system of convex viscous conservation laws. In the work by {Szepessy and Xin} \cite{XZ}, the more accurate long-time ansatz was constructed in the case when the initial perturbation carry a non-zero mass. We also mention series of works \cite{L,LZ1,LZ2,Yu} collaborated by Liu, Yu and Zeng where the Green's functions of the linearized equations around the shock profiles have been constructed. Such approach in terms of  Green's functions not only  can establish the fine pointwise structure of solutions through the wave propagation, but also be widely applied to a variety of physical situations where complicated bifurcations may occur, for instance, \cite{LY} for the analysis of kinetic boundary layers in the regime of the critical Mach number. On the other hand, Zumbrun and his collaborators have developed a theory for stability or instability criteria of shock profiles, see \cite{MZ} and \cite{GK} for instance. These works have built a connection between the nonlinear stability and spectrum stability of shock profiles, and the criteria for the latter one can be verified or disproved by the aid of combination of the  classical energy method and the numerical or analytical computation of Evans functions. This approach has been justified to be quite effective in understanding a large class of non-classical shocks, such as detonation \cite{GK1}, viscoelasticity \cite{BMK}, MHD \cite{Z} and so on.

\subsection{Idea of the proof of main results}
We briefly state the ideas for the proof of Theorems \ref{thm1.1}, \ref{thm1.2} and \ref{thm1.3}. As mentioned before, the proof of Theorem \ref{thm1.1} is based on the centre manifold theorem as in \cite{M-P}. As for Theorem \ref{thm1.3},  the main efforts have been made to treat the extra effect of the self-consistent force on the energy estimates, compared to the case of the classical Navier-Stokes equations.

For the proof of Theorem \ref{thm1.2}, the key step is to obtain the uniform-in-$\vep$ estimates on the $k$th derivatives $(k\geq 2)$ of the solution to the linearized remainder system \eqref{3.2.1}. The difficulties come from the second order derivative term ${\delta\phi''}/{\sqrt{T+1}}$ on the right-hand side of the first equation of \eqref{3.2.1}. Specifically, when estimating ${\dd^kn}/{\dd z^k}$, the trouble terms like
$$
\left(\f{\delta}{\sqrt{T+1}}\f{\dd^{k+1}\phi}{\dd z^{k+1}}, {w_{\alpha}^2}\f{\dd^kn}{\dd z^k}\right)
$$
are hard to handle, due to the degeneracy of the Poisson equation when $\vep\rightarrow 0.$ To resolve them, we make an essential use of the structure of the Poisson equation. Indeed, our strategy is to use the Poisson equation to represent ${\dd^kn}/{\dd z^k}$ in terms of $\phi$, which leads to a crucial cancellation in this inner product term.  Also, this strategy is crucially used in our later stability analysis. Unfortunately, for the dynamical stability  problem, the principle part of the similar trouble term is involved in the energy functional $\CE_1(t)$, see $\eqref{E1}$. And, the restriction $T>0$ is essentially required to assure the positivity of $\CE_1(t)$. This is the reason why  the condition $T>0$  is necessary in Theorem \ref{thm1.3}.

%\medskip
\subsection{Organization and notations}
The rest of this paper is organized as follows. In Section \ref{sec2}, we shall prove Theorem \ref{thm1.1} for the construction of the shock profile in terms of \eqref{1.4} for the fixed viscosity coefficient $\mu>0$ and the fixed Debye length $\la>0$. In Section \ref{sec3}, we shall prove Theorem \ref{thm1.2} for the KdV-Burgers approximation to the shock profile. One key point there is to show Proposition \ref{lm3.2.1} for the existence of solutions to the linear inhomogeneous problem, particularly to obtain the estimate \eqref{3.2.2}. The smallness of $\de=\bar{\la}^2/\bar{\mu}^2$ plays an essential role in the analysis. In Section \ref{sec4}, we shall prove Theorem \ref{thm1.3} concerning the dynamical stability of the shock profile. In the Appendix, for completeness, we first list a lemma about the property of the KdV-Burgers shock profile, give the estimates on the error between $[n_1,u_1,\phi_1]$ and $[n_{1,\vep},u_{1,\vep},\phi_{1,\vep}]$, write down the explicit formula of remainder $r_2$ and $r_3$ defined in \eqref{def.r1-6}, and in the end show Lemmas \ref{lem4.4} and \ref{lem4.5} related to the higher order energy estimates.

\medskip
\noindent{\it Notations.} Throughout this paper, $C$ denotes some generic positive (generally large) constant and
$c$ denotes some generic positive (generally small) constant, where both C and $c$ may take different
values in different places. $\|\cdot\|_{L^p}$ stands for the $L^p_x$-norm $(1\leq p\leq \infty)$. Sometimes, we denote $(\cdot,\cdot)$ to be the inner product in $L^2_x$ for convenience. We also
use $H^k$ $(k\geq 0)$ to denote the usual Sobolev space with respect to $x$ variable.
%\end{subsection}

\section{Existence for shock profiles of small amplitude}\label{sec2}

In this section, we construct the shock profile solution
$$
[\b{n}, \b{u}, \b{\phi}](x-st)
$$ to the system \eqref{1.1} for the fixed constant viscosity $\mu>0$ and Debye length $\lambda>0$. The starting point is to rewrite the Poisson equation %$\eqref{1.4}_3$
as  a first-order ODE system for $[\phi,\phi']$ so that one can use the center manifold approach (cf.~\cite{M-P}) to treat the problem.

\medskip
\noindent{\it Proof of Theorem \ref{thm1.1}}. From the first equation of $\eqref{1.4}$, one can solve $u$ by $n$ as $u=s(1-n^{-1})$. Plugging it back into the second equation of $\eqref{1.4}$ and integrating the resultant equation from $-\infty$ to $\xi$, one obtains the following equivalent system
\begin{equation}\label{2.2}
\left\{
\begin{aligned}
\f{\dd n}{\dd \xi}&=\f{n^2}{\mu s}\bigg\{(T-s^2)(n-1)+\f{s^2(n-1)^2}{n}-\f{\lambda^2W^2}{2}+Z-1\bigg\},\\
\f{\dd Z}{\dd \xi}&=ZW,\\
\f{\dd W}{\dd \xi}&=-\lambda^{-2}(n-Z),
\end{aligned}
\right.
\end{equation}
with the far fields given by
\begin{align*}%\label{2.3}
\lim_{\xi\rightarrow \pm\infty}[n(\xi), Z(\xi), W(\xi)]=[n_{\pm}, Z_{\pm}, W_{\pm}]=[n_{\pm}, n_{\pm}, 0].
\end{align*}
Here we have denoted $Z=e^{\phi}$ and $W=\f{\dd\phi}{\dd \xi}$. For convenience, we introduce $U=[n,Z,W]$, $U_{\pm}=[n_{\pm},Z_{\pm},W_{\pm}]$ and write \eqref{2.2} in the form $\dot{U}=F(U)$, where $F(\cdot)$ denotes the vector field on the right hand of \eqref{2.2}. %in what follows.
Borrowing the idea from \cite{M-P}, we introduce the following extended ODE system
\begin{equation}\label{2.4}
\left\{
\begin{aligned}
\dot{n}&=\f{n^2}{\mu \tau}\bigg\{(T-\tau^2)(n-1)+\f{\tau^2(n-1)^2}{n}-\f{\lambda^2W^2}{2}+Z-1\bigg\},\\
\dot{Z}&=ZW,\\
\dot{W}&=-\lambda^{-2}(n-Z),\\
\dot{\tau}&=0.
\end{aligned}\right.
\end{equation}
One can see that %$[n_-,Z_-,W_-,\sqrt{T+1}]$,
$[n_{\pm},Z_{\pm},W_{\pm},s]$ are the only two critical points of \eqref{2.4}. Now we fix $[n_-,Z_-,W_-,\sqrt{T+1}]$ as a reference state and construct the center manifold of \eqref{2.4} around this reference state. To do so, we calculate the Jacobian of \eqref{2.4} at the critical point $[n_-,Z_-,W_-,\sqrt{T+1}]$ as
$$
J=\left(
\begin{array}{cccc}
-\frac{1}{\mu\sqrt{T+1}} &\f{1}{\mu\sqrt{T+1}}& 0 & 0\\
0& 0& 1& 0\\
-\f{1}{\lambda^2}& \f{1}{\lambda^2} & 0 & 0\\
0 & 0 & 0 & 0
\end{array}
\right).
$$
The eigenvalues of $J$ are given by
%\red{
\begin{align}
\s_1&=-\f{1}{2\mu\sqrt{T+1}}-\sqrt{\f{1}{4\mu^2(T+1)}+\f{1}{\lambda^2}}<0, \nonumber\\
\s_2&=\s_3=0,\nonumber\\
\s_4&=-\f{1}{2\mu\sqrt{T+1}}+\sqrt{\f{1}{4\mu^2(T+1)}+\f{1}{\lambda^2}}>0.\nonumber
\end{align}
%}
One has two eigenvectors associated with the zero eigenvalue:
$$
\CR_1=(1,1,0,0), \quad  \CR_2=(0,0,0,1).
$$
Then by using the {\it Centre Manifold Theorem} (e.g. Proposition 3.2 in \cite{M-P}), we have the following 2-d manifold $M^*$ which is invariant by the flow \eqref{2.4}:
\begin{equation}
\left\{
\begin{aligned}
U&=U(\eta, \tau)=U_-+\eta[1,1,0] +\CH(\eta, \tau),\nonumber\\
\tau&=\tau,\nonumber
\end{aligned}
\right.
\end{equation}
%for
with $|\eta|+|\tau-\sqrt{T+1}|\leq c $ %Here
for some constant $c>0$. % is a universal constant. and
Here
\begin{equation*}
%\label{ }
\CH(\eta,\tau)=[\CH^1,\CH^2,\CH^3](\eta,\tau)
\end{equation*}
is the higher order term with
\begin{align}\label{2.6}
\CH(0,\sqrt{T+1})=\CH_{\eta}(0,\sqrt{T+1})=\CH_{\tau}(0,\sqrt{T+1})=0.
\end{align}
If $|n_+-1|$ is small enough, the equilibria $[U_-,s]$ and $[U_+,s]$ belong to $M^*$. Therefore, there exists $\eta_+<0$ such that $U_+=U_-+\eta_+[1,1,0]+\CH(\eta_+,s)$. Here the sign of $\eta_+$ follows from the compressibility. Pick up the following curve
\begin{align}\label{2.6-1}
U(\eta,\tau)|_{\tau=s}=[n(\eta),Z(\eta),W(\eta)]:=U_-+\eta[1,1,0]+\CH(\eta,s), %\quad ,
\end{align}
for $\eta\in [\eta_+,0]$. It is direct to check that $U(\eta,s)$ is invariant by the flow \eqref{2.2}, so that %\Red{it meets the direction of the flow changes}
${\dd U(\eta,s)}/{\dd\eta}$, the direction field of $U(\eta, s)$, is parallel to $F(U(\eta,s))$.
%\marginpar{\Blue{ZHU: I am confused with this sentence....Could you make it more precise? RJ}} % also it connects $U_-$ and $U_+$ from $0$ to $\eta_+$.
Therefore, \eqref{2.2} admits a trajectory connecting $U_-$ and $U_+$ from {$\xi=-\infty$ to $\xi=+\infty$}, if and only if the following ODE
%and introduce a scalar function $g(\eta)$ determined by
%\begin{align}\label{2.7}
%g(\eta)U_{\eta}(\eta,s)=F(U(\eta,s)).
%\end{align}
%By using the  argument of \cite[Proposition 3.3]{M-P}, one can show that \eqref{2.2} and \eqref{2.3} admit a smooth solution $U(\eta(\xi),s)$ provided that the following ODE
\begin{equation}\label{2.8}
\left\{\begin{aligned}
&\f{\dd \eta}{\dd \xi}=g(\eta),-\infty<\xi<+\infty,\\
&\lim_{\xi\rightarrow -\infty}\eta(\xi)=0,\lim_{\xi\rightarrow +\infty}\eta(\xi)=\eta_+,
\end{aligned}\right.
\end{equation}
induces a well-defined trajectory $\eta=\eta(\xi)$ for all $\xi\in\R$, where $g(\eta)$ is determined by
\begin{align}\label{2.7}
g(\eta)U_{\eta}(\eta,s)=F(U(\eta,s)).
\end{align}
Therefore, in what follows we only focus on the existence of the solution to \eqref{2.8}. Note that since $\eta_+<0$, the standard ODE theory shows that \eqref{2.8} has a smooth solution if and only if $g(\eta)<0$ for all $\eta\in (\eta_+,0)$. Since $U_+$ is the unique state near $U_-$ such that $F(U_+)=0$, it follows from \eqref{2.7} that $g(\eta)$ vanishes only at two end points of $[\eta_+,0]$. To further show that $g(\eta)$ has the strictly negative sign in the open interval $(\eta_+,0)$, we calculate $\dot{g}(0)$ by differentiating \eqref{2.7} and then taking the inner product of the resulting equation  with $U_{\eta}(0,s)$ as
\begin{align}
\dot{g}(0)&=\f{U_{\eta}(0,s)dF(U_-)U_{\eta}(0,s)^T}{|{U}_{\eta}(0,s)|^2}\notag\\
&=\f{T+1}{2\mu s}(1-n_+)+O\left(H_{\eta}(0,s)\right)\nonumber\\
&=\f{{\sqrt{T+1}}}{2\mu }(1-n_+)+O\left((1-n_+)^2\right)>0,
\label{2.9}
\end{align}
%since $n_+<n_-=1$ and
for $|n_+-1|$ small enough. Here \eqref{2.6} has been used in the %second
last equality %line
in \eqref{2.9}. Therefore, it holds that $g(\eta)<0$ in $(\eta_+,0)$, %and
so that \eqref{2.8} admits a smooth solution $\eta(\xi)$ for $\xi\in \R$. Let
$$
[\b{n},\b{u},\b{\phi}](\xi):=[n,s(1-n^{-1}),\log Z]\left(\eta(\xi)\right),
$$
where $n(\cdot)$ and $Z(\cdot)$ are defined in \eqref{2.6-1}. It is direct to check that $[\b{n},\b{u},\b{\phi}]$ solves \eqref{1.2} with \eqref{1.3}. The monotonicity \eqref{1.1.1} follows from the following computations:
\begin{equation*}%\label{2.10}
\left\{\begin{aligned}
&\f{\dd\b{n}}{\dd\xi} =\dot{n}(\eta(\xi))\f{\dd\eta}{\dd\xi}=\left(1+H^1_{\eta}(\eta(\xi),s)\right)\f{\dd\eta}{\dd\xi}\sim\f{\dd\eta}{\dd\xi}<0,\\
&\f{\dd\b{u}}{{\dd}\xi} ={\f{s}{\b{n}^2}\f{\dd\b{n}}{\dd\xi}}<0,\\
&\f{\dd\b{\phi}}{\dd\xi} =Z^{-1}\dot{Z}(\eta(\xi))\f{\dd\eta}{\dd\xi}=Z^{-1}\{1+H^2_{\eta}(\eta(\xi),s)\}\f{\dd\eta}{\dd\xi}\sim\f{\dd\b{n}}{\dd\xi}<0,
\end{aligned}
\right.
\end{equation*}
provided that $|n_+-1|$ is sufficiently small. Furthermore, to verify \eqref{1.1.2} for $k=0$, it follows  from \eqref{2.9} that
%, it holds for $k=0,1$ that
$$
|\eta(\xi)|\leq C{|\eta_+|}e^{-\theta(1-n_+)|\xi|}\leq C(1-n_+)e^{-\theta(1-n_+)|\xi|},\quad  \xi<0,
$$
with some constants $C>0$ and $\theta>0$ independent of $\xi$. Then it follows from \eqref{2.6-1} that
$$
\big|[\b{n}-1,\b{u},\b{\phi}](\xi)\big|\leq C(1-n_+)e^{-\theta(1-n_+)|\xi|}, \quad \xi<0.
$$
Similarly, one also has
$$
\big|[\b{n}-n_+,\b{u}-u_+,\b{\phi}-\phi_+](\xi)\big|\leq C(1-n_+)e^{-\theta(1-n_+)|\xi|},\quad  \xi>0. %, k=0,1.
$$
This then proves \eqref{1.1.2} for $k=0$. Estimates on \eqref{1.1.2} with $k\geq 1$ for those high-order derivatives of $[\b{n},\b{u},\b{\phi}]$ can be similarly obtained by differentiating \eqref{1.2}, and the details of the proof are omitted for brevity. It is also straightforward to show the uniqueness (up to a shift) for the ODE system \eqref{1.2} and \eqref{1.3}. Therefore, we complete the proof of Theorem \ref{thm1.1}. \qed

\section{KdV-Burgers approximation to shock profiles}\label{sec3}

In this section we shall prove Theorem \ref{thm1.2} concerning the KdV-Burgers approximation of the smooth small-amplitude travelling shock profile under the scaling \eqref{def.scal} provided that $\delta>0$ given by \eqref{def.ratio} is small enough. As mentioned in Sections 1.3 and 1.4, it suffices to study the existence of solutions to the ODE system \eqref{3.1.13} for the remainders $n_R$ and $\phi_R$.

%\begin{subsection}
\subsection{Linear problem}
First of all, we start from the following linear inhomogeneous problem
\begin{equation}\label{3.2.1}
\left\{
\begin{aligned}
&\f{\dd n}{\dd z}=A(z)n+\f{\delta}{\sqrt{T+1}}\f{\dd^2 \phi}{\dd z^2}+h_1,\\
&-\varepsilon\delta\f{\dd^2\phi}{\dd z^2}=n-\phi+h_2,\\
&n(0)=0, \end{aligned}
\right.
\end{equation}
where $A(z):=2[1+\sqrt{T+1}n_1(z)].$ Recall the solution space %$\mathbf{X}_{\alpha}$
{$\mathbf{X}_{\alpha,k}$}
in \eqref{space}. The following result is concerned with the solvability and estimates of \eqref{3.2.1}, which is a crucial step for further treating \eqref{3.1.13}.

% is given by the following lemma.
%\begin{proposition}\label{lm3.2.1}
%If $0<\varepsilon\ll\delta\ll1$ and $\|h_1\|_{H^{1}_{\alpha}}+\|h_2\|_{H^{3}_{\alpha}}<\infty$, then for any $0<\alpha<2$, the linear ODE system \eqref{3.2.1} has a solution $U(z)=[n(z), \phi(z)]$ which is unique in $\mathbf{X}_{\alpha}$ and satisfies the following estimate:
%\begin{align}\label{3.2.2}
%\|U\|_{\mathbf{X}_{\alpha}}\leq C\left\{\|h_{1}\|_{H^{1}_{\alpha}}+\|h_2\|_{H^2_{\alpha}}+\delta\|h_2'''\|_{L^2_{\alpha}}\right\}
%\end{align}
%with a positive constant $C$ independent of $\varepsilon$ and $\delta$.
%\end{proposition}

\begin{proposition}\label{lm3.2.1}
Let $k\geq 2$ be an arbitrary integer and $0<\alpha<2$. There exist positive constants $\vep_1$ and $\delta_1$ such that if $0<\vep\leq \vep_1$, $0<\delta\leq \delta_1$, %\marginpar{\Blue{ZHU: Could you please make it more precise?}}
 and %$\|h_1\|_{H^{k-1}_{\alpha}}+\|h_2\|_{H^{k+1}_{\alpha}}<\infty$,
$$
\|h_1\|_{H^{k-1}_{\alpha}}+\|h_2\|_{H^{k+1}_{\alpha}}<\infty,
$$
then the linear ODE system \eqref{3.2.1} has a unique solution $U(z)=[n(z), \phi(z)]$ in $\mathbf{X}_{\alpha,k}$ satisfying the following estimate:
\begin{align}\label{3.2.2}
\|U\|_{\mathbf{X}_{\alpha,k}}\leq C\|h_1\|_{H^{k-1}_{\alpha}}+C\|h_2\|_{H^k_{\alpha}}+C\delta\left\|\f{\dd^{k+1}h_2}{\dd z^{k+1}}\right\|_{L^2_{\alpha}},
\end{align}
where $C>0$ is a generic constant independent of $\varepsilon$ and $\delta$.
\end{proposition}

\begin{proof}
We divide the proof by %two
three steps.

\medskip
\noindent\underline{Step 1}. In this step, we treat only the a priori estimates of solutions for the case $k=2$, that is to prove that any smooth solution $U(z)=[n(z),\phi(z)]$ to the system \eqref{3.2.1} enjoys the estimate \eqref{3.2.2} with $k=2$.
First of all, we estimate $n$ as follows. From the first equation of \eqref{3.2.1}, we can represent $n$ as \begin{align}\label{n}
n(z)=\int_0^ze^{\int_{z'}^zA(\tau)\dd \tau}\left[\f{\delta\phi''}{\sqrt{T+1}}+h_1\right](z')\dd z'.
\end{align}
It is direct to check that
$$\lim_{z\rightarrow+\infty}A(z)=-2<0, \quad\lim_{z\rightarrow-\infty}A(z)=2>0.$$ Then %from the standard ODE theory, it follows by the first equation of \eqref{3.2.1} that
from \eqref{n}, we have
\begin{align}\label{3.2.3}
\|n\|_{L^{2}_{\alpha}}\leq C\delta\|\phi'' \|_{L^{2}_{\alpha}}+C\|h_1\|_{L^{2}_{\alpha}},
\end{align}
for $\alpha\in (0,2)$. Here we emphasize that the constant $C>0$ is independent of $\vep$ and $\delta$. Then, again from  the first equation of \eqref{3.2.1}, one has
\begin{align}
\|n'\|_{L^2_{\alpha}}&\leq {|A|_{L^{\infty}}}\|n\|_{L^{2}_{\alpha}}+C\delta\|\phi''\|_{L^2_{\alpha}}+C\|h_1\|_{L^2_{\alpha}}\leq
C\delta\|\phi''\|_{L^2_{\alpha}}+C\|h_1\|_{L^2_{\alpha}}.\label{3.2.4}
\end{align}
Next, we turn to estimate $\phi$. Taking the inner product of the second equation of $\eqref{3.2.1}$ with $w^2_{\alpha}\phi$, one has
\begin{align}
\|\phi\|_{L^2_{\alpha}}^2=(\vep\delta\phi'',w^2_{\alpha}\phi)+(n+h_2, w^2_{\alpha}\phi).\nonumber
\end{align}
By Cauchy-Schwarz, the second inner product term is bounded by
$$
|(n+h_2,w^2_{\alpha}\phi)|\leq \eta\|\phi\|^2_{L^2_{\alpha}}+C_{\eta}\{\|n\|_{L^2_{\alpha}}^2+\|h_2\|_{L^2_{\alpha}}^2\},
$$
with an arbitrary constant $0<\eta<1$ to be chosen later. As to the first inner product term, it holds from integration by parts and Cauchy-Schwarz that
$$
\begin{aligned}
(\vep\delta\phi'',w^2_{\alpha}\phi)&=-\vep\delta\|\phi'\|_{L^2_{\alpha}}^2+(\vep\delta\phi',-2w_{\alpha}w_{\alpha}'\phi)\\
&\leq -\vep\delta\|\phi'\|_{L^2_{\alpha}}^2+\eta\vep\delta\|\phi'\|_{L^2_{\alpha}}^2
+C_{\eta}\vep\delta\|\phi\|_{L^2_{\alpha}}^2.
\end{aligned}
$$
Therefore, by taking $\eta>0$ suitably small, one has
\begin{align}\label{3.2.5}
\|\phi\|_{L^2_{\alpha}}^2+\vep\delta\|\phi'\|_{L^2_{\alpha}}^2\leq C\|n\|_{L^2_{\alpha}}^2+C\vep\delta
\|\phi\|_{L^2_{\alpha}}^2+C\|h_2\|_{L^2_{\alpha}}^2.
\end{align}
Similarly, taking the inner product of the second equation of $\eqref{3.2.1}$ with $w^2_{\alpha}\phi''$ and integrating by parts, one has
\begin{align}\label{3.2.6}
\|\phi'\|_{L^2_{\alpha}}^2+\vep\delta\|\phi''\|_{L^2_{\alpha}}^2\leq C\{\|\phi\|_{L^2_{\alpha}}^2+
\|n\|_{H^1_{\alpha}}^2+\|h_2\|_{H^1_{\alpha}}^2\}.
\end{align}

In what follows it is necessary to get the uniform-in-$\vep$ estimate for $\|[n'',\phi'']\|_{L^2_{\alpha}}$. Differentiating the first equation of $\eqref{3.2.1}$ with respect to $z$ and taking the inner product of the resultant equation with $w_{\alpha}^2n''$, one has
\begin{align}\label{3.2.7}
\|n''\|_{L^2_{\alpha}}^2=\big(A'n+An'+h_1',w^2_{\alpha}n''\big)+\delta\sqrt{T+1}^{-1}(\phi''',w^2_{\alpha}n'').
\end{align}
By Cauchy-Schwarz, the first inner product term is bounded by
$$
\eta\|n''\|_{L^2_{\alpha}}^2+C_{\eta}\{\|n\|_{H^1_{\alpha}}^2+\|h_1'\|_{L^2_{\alpha}}\}
$$
with an arbitrary constant $0<\eta<1$ to be chosen later. To estimate the last term on the right-hand side of \eqref{3.2.7}, we firstly differentiate the second equation of $\eqref{3.2.1}$ twice and then solve $n''$ as
\begin{align}\label{3.2.8-1}
n''(z)=\phi''(z)-\vep\delta\phi''''(z)-h_2''(z).
\end{align}
Thus, applying %Then replacing $n''$ by the R.H.S of
\eqref{3.2.8-1}, one has
\begin{align}\label{3.2.9}
(\delta\phi''',w_{\alpha}^2n'')=(\delta\phi''',w_{\alpha}^2\phi'')+(\delta\phi''',-\vep\delta w_{\alpha}^2\phi'''')+(\delta\phi''',-w_{\alpha}^2h_2'').
\end{align}
By integration by parts, the first term on the right is bounded as
$$
|(\delta\phi''',w_{\alpha}^2\phi'')|=|(\delta\phi'',w_{\alpha}w_{\alpha}'\phi'')|\leq C\delta\|\phi''\|_{L^2_{\alpha}}^2,
$$
the second term is bounded {as} %by $C\vep\delta^2\|\phi'''\|_{L^2_{\alpha}}^2$,
$$
{|(\delta\phi''',-\vep\delta w_{\alpha}^2\phi'''')|=\vep\delta^2|(\phi''',w_{\alpha}w_{\alpha}'\phi''')|\leq C\vep\delta^2\|\phi'''\|_{L^2_{\alpha}}^2,}
$$
and the last term is bounded as
$$
|(\delta\phi''',-w_{\alpha}^2h_2'')|=|(\delta\phi'',w_{\alpha}\{w_{\alpha}h_2'''+2w_{\alpha}'h_2''\})|\leq \eta
\|\phi''\|_{L^2_{\alpha}}^2+C_{\eta}\delta^2\|h_{2}''\|_{H^1_{\alpha}}^2.
$$
This completes all estimates on the right-hand side of \eqref{3.2.9}. Plugging those estimates back to \eqref{3.2.7}, one has
\begin{multline}
\|n''\|_{L^2_{\alpha}}\leq C\{\|n\|_{H^1_{\alpha}}+(\sqrt{\delta}+\sqrt{\eta})\|\phi''\|_{L^2_{\alpha}}+\sqrt{\vep}\delta\|\phi'''\|_{L^2_{\alpha}}\}\\
+C\|h_1'\|_{L^2_{\alpha}}+C_{\eta}\delta\|h_2''\|_{H^1_{\alpha}}.
\label{3.2.10}
\end{multline}
Here the constant $0<\eta<1$ can be chosen  small enough. For the estimate of $\|\phi''\|_{L^2_{\alpha}}$, we take the inner product of \eqref{3.2.8-1} with $w_{\alpha}^2\phi''$, which yields that
\begin{align}\label{3.2.10-1}
\|\phi''\|_{L^2_{\alpha}}^2=(\vep\delta\phi'''',w_{\alpha}^2\phi'')+(h_2'',w_{\alpha}^2\phi'')+(n'',w_{\alpha}^2\phi'').
\end{align}
From integration by parts again, one has
\begin{align*}
(\vep\delta\phi'''',w_{\alpha}^2\phi'')&=-\vep\delta\|\phi'''\|_{L^2_{\alpha}}^2+(-\vep\delta\phi''',2w_{\alpha}w_{\alpha}'\phi'')\\
&\leq -\vep\delta\|\phi'''\|_{L^2_{\alpha}}^2+\eta\|\phi''\|_{L^2_{\alpha}}^2+C_{\eta}\vep^2\delta^2\|\phi'''\|_{L^2_{\alpha}}^2.
\end{align*}
Moreover, by Cauchy-Schwarz, the last two terms on the right-hand side of \eqref{3.2.10-1} is bounded by
$$
\eta\|\phi''\|_{L^2_{\alpha}}^2+C_{\eta}\{\|n''\|_{L^2_{\alpha}}+\|h_2''\|_{L^2_{\alpha}}^2\}.
$$
Therefore, by collecting all estimates and taking $0<\eta<1$ suitably small, it follows from \eqref{3.2.10-1} that
\begin{align}\label{3.2.11}
\|\phi''\|_{L^2_{\alpha}}+\sqrt{\vep\delta}\|\phi'''\|_{L^2_{\alpha}}\leq C\vep\delta\|\phi'''\|_{L^2_{\alpha}}+C\{\|n''\|_{L^2_{\alpha}}+\|h_2''\|_{L^2_{\alpha}}\}.
\end{align}
Furthermore, taking the inner product of \eqref{3.2.8-1} with %$\phi''''w_{\alpha}^2$
$\vep\delta\phi''''w_{\alpha}^2$, one has
\begin{align}\label{3.2.12}
\vep\delta\|\phi''''\|_{L^2_{\alpha}}\leq C\{\|\phi''\|_{L^2_{\alpha}}+\|n''\|_{L^2_{\alpha}}+\|h_2''\|_{L^2_{\alpha}}\}.
\end{align}
Finally, a suitable linear combination of \eqref{3.2.3}, \eqref{3.2.4}, \eqref{3.2.5}, \eqref{3.2.6}, \eqref{3.2.10}, \eqref{3.2.11} and \eqref{3.2.12} yields that
\begin{align}
&\|[n,\phi]\|_{H^2_{\alpha}}+\sqrt{\vep\delta}\|\phi'''\|_{L^2_{\alpha}}+\vep\delta\|\phi''''\|_{L^2_{\alpha}}\notag\\
&\leq C(\delta+\sqrt{\delta})\|\phi''\|_{L^2_{\alpha}}+C\sqrt{\vep\delta}(\sqrt{\delta}+\sqrt{\vep\delta})\|\phi'''\|_{L^2_{\alpha}}\notag\\
&\quad +C\{\|h_1\|_{H^1_{\alpha}}+\|h_2\|_{H^2_{\alpha}}+\delta\|h_2'''\|_{L^2_{\alpha}}\}.\label{3.2.13}
\end{align}
Therefore, \eqref{3.2.2} with $k=2$ follows from  \eqref{3.2.13} by taking $\delta>0$ and $\vep>0$ suitably small.

\medskip
\noindent\underline{Step 2}. In this step, we use the induction argument to show that the estimates \eqref{3.2.2} is valid for any $k\geq 2$. Notice that \eqref{3.2.2} for $k=2$ has been proved in Step 1. Assume that this is valid for $k\geq 2$. Differentiating the first equation of \eqref{3.2.1} $k$-times with respect to $z$ yields %\marginpar{\gr{Drop $y$ in \eqref{nk}}}
\begin{align}\label{nk}
\f{\dd^{k+1} n}{\dd z^{k+1}}=\sum_{0\leq k'\leq k}\f{\dd^{k'} A}{\dd z^{k'}}\cdot\f{\dd^{k-k'}n}{\dd z^{k-k'}}+\f{\delta}{\sqrt{T+1}}\f{\dd^{k+2}\phi}{\dd z^{k+2}}+\f{\dd^{k}h_1}{\dd z^k}.
\end{align}
Taking the inner product of \eqref{nk} with $w_{\alpha}^2\f{\dd^{k+1}n}{\dd z^{k+1}}$, we have
\begin{align}\label{nk1}
\left\|\f{\dd^{k+1}n}{\dd z^{k+1}}\right\|_{L^2_{\alpha}}^2=&\sum_{0\leq k'\leq k}\left(\f{\dd^{k+1}n}{\dd z^{k+1}}\text{ },\text{ }w_{\alpha}^2\f{\dd^{k'} A}{\dd z^{k'}}\cdot\f{\dd^{k-k' }n}{\dd z^{k-k'}}\right)+\left(\f{\dd^kh_1}{\dd z^k}\text{ },\text{ }w_{\alpha}^2\f{\dd^{k+1}n}{\dd z^{k+1}}\right)\nonumber\\
&
+\left(\f{\delta}{\sqrt{T+1}}\f{\dd^{k+2}\phi}{\dd z^{k+2}}\text{ },\text{ }w_{\alpha}^2\f{\dd^{k+1}n}{\dd z^{k+1}}\right):=J_1+J_2+J_3.
\end{align}
Using \eqref{3.1.9-1}, the first inner product term on the right is bounded as
$$\begin{aligned}
|J_1|&\leq C\sum_{0\leq k'\leq k }\left|\f{\dd^{k'}A}{\dd z^{k'}}\right|_{L^{\infty}}\cdot\left\|\f{\dd^{k-k'}n}{\dd z^{k-k'}}\right\|_{L^2_{\alpha}}\cdot\left\|\f{\dd^{k+1}n}{\dd z^{k+1}}\right\|_{L^2_{\alpha}}\\
&\leq \eta\left\|\f{\dd^{k+1}n}{\dd z^{k+1}}\right\|_{L^2_{\alpha}}^2+C_{\eta,k}\|n\|^2_{H^{k}_{\alpha}}.
\end{aligned}
$$
Here the positive constant $\eta>0$ can be chosen to be arbitrarily small. By Cauchy-Schwarz, the second inner product is bounded as
$$|J_2|\leq\eta\left\|\f{\dd^{k+1}n}{\dd z^{k+1}}\right\|_{L^2_{\alpha}}^2+C_{\eta}\left\|\f{\dd^{k}h_1}{\dd z^k}\right\|_{L^2_{\alpha}}^2.
$$
To further estimate $J_3$, in the similar way as before, we differentiate the second equation of \eqref{3.2.1} $k+1$ times and represent $\f{\dd^{k+1} n}{\dd z^{k+1}}$ as
\begin{align}\label{nk4}\f{\dd^{k+1}n}{\dd z^{k+1}}=\f{\dd^{k+1}\phi}{\dd z^{k+1}}-\vep\delta\f{\dd^{k+3}\phi}{\dd z^{k+3}}-\f{\dd^{k+1}h_2}{\dd z^{k+1}}.
\end{align}
Substituting this into $J_3$, we have
\begin{align}\label{phik}
J_3=&\f{\delta}{\sqrt{T+1}}\left(\f{\dd^{k+2}\phi}{\dd z^{k+2}}\text{ },\text{ }w_\alpha^2\f{\dd^{k+1}\phi}{\dd z^{k+1}}\right)+\f{\vep\delta^2}{\sqrt{T+1}}\left(-\f{\dd^{k+2}\phi}{\dd z^{k+2}}\text{ },\text{ }w_\alpha^2\f{\dd^{k+3}\phi}{\dd z^{k+3}}\right)\nonumber\\
&+\f{\delta}{\sqrt{T+1}}\left(-\f{\dd^{k+2}\phi}{\dd z^{k+2}}\text{ },\text{ }w_\alpha^2\f{\dd^{k+1}h_2}{\dd z^{k+1}}\right).
\end{align}
Then by integration by parts, it follows from \eqref{phik} that
\begin{align*}
|J_3|&\leq  C(\delta+\eta)\left\|\f{\dd^{k+1}\phi}{\dd z^{k+1}}\right\|_{L^2_{\alpha}}^2+C\vep\delta^2\left\|\f{\dd^{k+2}\phi}{\dd z^{k+2}}\right\|_{L^2_{\alpha}}^2\\
&\quad+C_{\eta}\delta^2\left(
\left\|\f{\dd^{k+1}h_2}{\dd z^{k+1}}\right\|_{L^2_{\alpha}}^2+\left\|\f{\dd^{k+2}h_2}{\dd z^{k+2}}\right\|_{L^2_{\alpha}}^2\right).
\end{align*}
Substituting estimates of $J_1$ to $J_3$ into \eqref{nk1}, we have, for any small $\eta>0$, that
%\marginpar{\gr{Typos corrected in \eqref{nk2}.}}
\begin{align}\label{nk2}
\left\|\f{\dd^{k+1}n}{\dd z^{k+1}}\right\|_{L^2_{\alpha}}\leq& C\|n\|_{H^k_{\alpha}}+C(\sqrt{\delta}+\eta)\left\|\f{\dd^{k+1}\phi}{\dd z^{k+1}}\right\|_{L^2_{\alpha}}+C\vep^{1/2}\delta\left\|\f{\dd^{k+2}\phi}{\dd z^{k+2}}\right\|_{L^2_{\alpha}}\nonumber\\
&+C_{\eta}\left\|\f{\dd^{k}{h_1}}{\dd z^{k}}\right\|_{L^2_{\alpha}}+C_{\eta}\delta\left(\left\|\f{\dd^{k+1}h_2}{\dd z^{k+1}}\right\|_{L^2_{\alpha}}+\left\|\f{\dd^{k+2}h_2}{\dd z^{k+2}}\right\|_{L^2_{\alpha}}\right).
\end{align}
By using the induction assumption, \eqref{nk2} implies that %\marginpar{\gr{Here in \eqref{nk3} I missed a $\delta$ in the previous version. Also typos are corrected.}}
\begin{align}
\label{nk3}
\left\|\f{\dd^{k+1}n}{\dd z^{k+1}}\right\|_{L^2_{\alpha}}&\leq C(\sqrt{\delta}+\eta)\left\|\f{\dd^{k+1}\phi}{\dd z^{k+1}}\right\|_{L^2_{\alpha}}+C\vep^{1/2}{\delta}\left\|\f{\dd^{k+2}\phi}{\dd z^{k+2}}\right\|_{L^2_{\alpha}}\nonumber\\
&\quad +C_{\eta}\left\{\|h_1\|_{H^{{k}}_{\alpha}}+\|h_2\|_{H^{{k+1}}_{\alpha}}+\delta\left\|\f{\dd^{{k+2}}h_2}{\dd z^{k+1}}\right\|_{L^2_{\alpha}}\right\}.
\end{align}
%By using
From \eqref{nk4}, we obtain that
\begin{multline}\label{nk5}
\left\|\f{\dd^{k+1}\phi}{\dd z^{k+1}}\right\|_{L^2_{\alpha}}+\sqrt{\vep\delta}\left\|\f{\dd^{k+2}\phi}{\dd z^{k+2}}\right\|_{L^2_{\alpha}}+\vep\delta\left\|\f{\dd^{k+3}\phi}{\dd z^{k+3}}\right\|_{L^2_{\alpha}}\\
\leq C\left\|\f{\dd^{k+1}n}{\dd z^{k+1}}\right\|_{L^2_{\alpha}}+C\left\|\f{\dd^{k+1}h_2}{\dd z^{k+1}}\right\|_{L^2_{\alpha}}.
\end{multline}
Therefore, \eqref{3.2.2} for $k+1$ follows from a suitable combination of \eqref{nk3} and \eqref{nk5} and taking both $\eta$ and $\delta$ suitably small. This completes the proof of estimate \eqref{3.2.2}.

\medskip
\noindent\underline{Step 3}. In this step, we construct the solution to \eqref{3.2.1} by using the approximation sequence $[n^{\vep'},\phi^{\vep'}]$ in terms of  solutions to the following ODE system:
\begin{equation}\label{3.2.18}
\left\{
\begin{aligned}
&\f{\dd n^{\varepsilon'}}{\dd z}=A(z)n^{\varepsilon'}+\f{\varepsilon'\delta}{\sqrt{T+1}}\f{\dd^2 \phi^{\varepsilon'}}{\dd z^2}+h_1,\\
&-\varepsilon\delta\f{\dd^2\phi^{\varepsilon'}}{\dd z^2}=n^{\varepsilon'}-\phi^{\varepsilon'}+h_2,\\
&n^{\vep'}(0)=0,
\end{aligned}
\right.
\end{equation}
where $0\leq \varepsilon'\leq 1$. Note that when $\vep'=1$, the system \eqref{3.2.18} is exactly \eqref{3.2.1} under consideration. In what follows let $L^{-1}_{\varepsilon',\varepsilon,\delta}$ formally denote the solution operator for the problem \eqref{3.2.18}.

%\begin{itemize}
%\item
\medskip
\noindent(i) Firstly we start with the case of $\vep'=0$. From the first equation of $\eqref{3.2.18}$, one can solve $n^0(z)$ as
\begin{equation}\nonumber
n^0(z)=\int_0^ze^{\int_{y}^zA(\tau)\dd \tau}h_1(y)\dd y.
\end{equation}
Since it holds that
$$
\lim_{z\rightarrow+\infty}A(z)=-2<0,\quad \lim_{z\rightarrow-\infty}A(z)=2>0,
$$
one has
$$
\|n^0\|_{H^{k}_{\alpha}}\leq C\|h_1\|_{H^{k-1}_{\alpha}}<\infty.
$$
The existence of the solution $\phi^0$ to the second equation of $\eqref{3.2.18}$ in case of $\varepsilon'=0$ can be shown by the Lax-Milgram Theorem and the ${H^{k+2}_{\alpha}}$-regularity can be shown by using the ${H^{k}_{\alpha}}$-estimate of $n^0$ and $h_2$. Here, the details of the proof are omitted for brevity. Therefore, the solution $U^0(z)=[n^0(z),\phi^{0}(z)]$ is well defined in the function space ${\mathbf{X}_{\alpha,k}}$. One can thereby use the similar argument in {previous steps} to deduce that the solution $U^0(z)=[n^0(z),\phi^{0}(z)]$ also satisfies the estimate \eqref{3.2.2}. Hence the solution operator $L^{-1}_{0,\varepsilon,\delta}$ in ${\mathbf{X}_{\alpha,k}}$ has been constructed.

%\item
\medskip
\noindent(ii) Next, we construct the solution of \eqref{3.2.18} when $\varepsilon'>0$ is small enough. For any %$U=[n,\phi]\in \mathbf{X}_{\alpha}$,
{$U=[n,\phi]\in \mathbf{X}_{\alpha,k}$, }we introduce the linear mapping
$$
T_{\varepsilon'}U:=L^{-1}_{0,\varepsilon,\delta}[\frac{\varepsilon'\delta}{\sqrt{T+1}}\phi''+h_1,h_2].
$$
Then for any $U^1=(\tilde{n}_1,\tilde{\phi}_1)$ and $U^2=(\tilde{n}_2,\tilde{\phi}_2)$ in ${\mathbf{X}_{\alpha,k}}$,
%$X_{\varepsilon,\delta}$,
one has from \eqref{3.2.2} that
%\begin{align}\label{3.2.20}
%\|T_{\varepsilon'}(U^1-U^2)\|_{\mathbf{X}_{\alpha,k}}\leq C\varepsilon'\delta\|\tilde{\phi}_1''-\tilde{\phi}_2''\|_{H^1_{\alpha}}\leq C_{\delta,\varepsilon}\varepsilon'\|U^1-U^2\|_{\mathbf{X}_{\alpha}},
%\end{align}
\begin{align}\label{3.2.20}
\|T_{\varepsilon'}(U^1-U^2)\|_{\mathbf{X}_{\alpha,k}}&\leq C\varepsilon'\delta\|[\tilde{\phi}_1-\tilde{\phi}_2]\|_{H^{k+1}_{\alpha}}\leq {C\vep^{-1/2}}\varepsilon'\|U^1-U^2\|_{\mathbf{X}_{\alpha,k}},
\end{align}
where the constant %$C_{\delta,\varepsilon}$
{$C>0$} is independent of {$\delta$, $\vep$, and }$\varepsilon'$. We now choose {$\varepsilon'_0=\vep^{1/2}/2C$.} %such that $C\vep^{-1/2}\varepsilon'_0<1/2$.
Then it follows from \eqref{3.2.20} that $T_{\varepsilon'}$ is a contraction mapping on %$\mathbf{X}_{\alpha}$
${\mathbf{X}_{\alpha,k}}$
for any $0<\varepsilon'\leq\varepsilon'_0$. Thus $T_{\varepsilon'}$ has a unique fixed point $U^{\vep'}:=[n^{\vep'},\phi^{\vep'}]\in{\mathbf{X}_{\alpha,k}}$. It is direct to check that $U^{\varepsilon'}=[n^{\varepsilon'},\phi^{\varepsilon'}]$ is the solution to \eqref{3.2.18}.  % and thereby satisfies \eqref{3.2.2} from Step 1.
Therefore, $L^{-1}_{\vep',\vep,\delta}$ is well-defined for any $0<\vep'\leq \vep_0'$. {Moreover, the solution also satisfies the estimates in \eqref{3.2.2}.}

\medskip
\noindent(iii) Lastly, we introduce
    $$T_{\varepsilon'_0+\varepsilon'}U=T_{\varepsilon'_0+\varepsilon'}[n,\phi]=L^{-1}_{\varepsilon'_0,\varepsilon,\delta}[\varepsilon'\delta \phi''+h_1,h_2].$$
Notice that by the uniform estimate \eqref{3.2.2}, the upper bounds on the norm of the solution in terms of $L^{-1}_{\vep',\vep,\delta}$ is independent of $\vep$ and $\delta$. Then by using the same argument as in (ii), one can show that $T_{\varepsilon'_0+\varepsilon'}$ is a contraction mapping on ${\mathbf{X}_{\alpha,k}}$ and thereby has a unique fixed point in
${\mathbf{X}_{\alpha,k}}$ for $0<\varepsilon'\leq \varepsilon'_0$. Therefore, the solution operator $L^{-1}_{2\varepsilon'_0,\varepsilon,\delta}$ has been constructed. Now we can repeat the same procedure and finally construct the solution operator $L^{-1}_{\varepsilon,\delta}\equiv L^{-1}_{1,\varepsilon,\delta}$ for the original problem \eqref{3.2.1} in ${\mathbf{X}_{\alpha,k}}$.
%\end{itemize}
The proof of Proposition \ref{lm3.2.1} is then completed.
\end{proof}
%\end{subsection}

%\begin{subsection}
\subsection{Justification of the approximation}
%Recall the remaining terms $r_i$ $(1\leq i\leq 6)$ given in \eqref{def.r1-6}. Direct computations yields the following %estimates.
We have the following estimates on the remaining terms $r_1$ to $r_6$ given in \eqref{def.r1-6}.

\begin{lemma}\label{lm3.1.2}
Let %\gr{$k\geq0$} be an integer and
$0<\alpha<2$. Then there exist positive constants $\vep_2$ and $\delta_2$ %and $C_k$, \gr{$k=0,1\cdots$}, %\Red{$0<\vep\ll\delta\ll1$}
such that if $0<\vep\leq \vep_2$ and $0<\delta\leq \delta_2$, then the following estimates hold:
%\marginpar{\Blue{ZHU: As before, please make it more precise. RJ}} %then there exists a positive constants $C_k>0$ independent of $\vep$ and $\delta$ such that
\begin{align}
\left|\f{\dd^kr_1}{\dd z^{k}}\right|\leq C_ke^{-\alpha|z|},\quad\left|\f{\dd^kr_4}{\dd z^{k}}\right|\leq C_ke^{-\alpha|z|},\nonumber
\end{align}
for any integer $k\geq 0$, where each $C_{k}>0$ is a generic constant independent of $\vep$ and $\delta$.
Moreover, %for $\gr{k\geq 2}$,
we have
$$
\begin{aligned}
&\|r_1\|_{H^{k-1}_{\alpha}}\leq {C_k},\\
&\|r_2[n,\phi]\|_{H^{k-1}_{\alpha}}\leq {C_k}\sqrt{\vep}\|[n,\phi]\|_{\mathbf{X}_{\alpha,k}},\\
&\|r_3\|_{H^{k-1}_{\alpha}}\leq {C_k}\vep\|[n,\phi]\|_{\mathbf{X}_{\alpha,k}}^2\{1+\|[n,\phi]\|^2_{\mathbf{X}_{\alpha,k}}\},\\
&\|r_4\|_{H^{k}_{\alpha}}\leq {C_k},\\
&\|r_5[\phi]\|_{H^{k}_{\alpha}}\leq {C_k}\vep\|\phi\|_{H^{k}_{\alpha}},\\
&\|r_6[\phi]\|_{H^{k}_{\alpha}}\leq {C_k}\vep^2e^{\vep^2\|\phi\|_{L^\infty}}\|\phi\|^2_{H^{k}_{\alpha}},
\end{aligned}
$$
for any integer $k\geq 2$, where each $C_{k}>0$ is a generic constant independent of $\vep$ and $\delta$.
\end{lemma}

\begin{proof}
We first consider $r_4$. Note from \eqref{ad-ffdr} that
$$
1+\vep n_{1,\vep}(\pm\infty)-e^{\vep\phi_{1,\vep}(\pm\infty)}=0.
$$
Hence, by \eqref{def.r1-6}, one can rewrite $r_4$ as
%A direct computation shows, for $z\gtrless 0$, that
$$
r_4=\vep^{-2}\left\{\vep\left[n_{1,\vep}-n_{1,\vep}(\pm\infty)\right]-\left[e^{\vep\phi_{1,\vep}}-e^{\vep\phi_{1,\vep}(\pm\infty)}\right]\right\}.
$$
Using \eqref{3.1.6}, it further reduces to
\begin{align*}
r_4
%=&\vep^{-2}\left\{\vep[n_{1,\vep}-n_{1,\vep}(\pm\infty)]-[e^{\vep\phi_{1,\vep}}-e^{\vep\phi_{1,\vep}(\pm\infty)}]\right\}\nonumber\\
%&+
%\vep^{-2}\big\{\underbrace{1+\vep n_{1,\vep}(\pm\infty)-e^{\vep\phi_{1,\vep}(\pm\infty)}}_{=0 \text{ by } \eqref{ad-ffdr} }\big\}\nonumber\\
=&\vep^{-1}\big\{[n_{1,\vep}-n_{1,\vep}(\pm\infty)]-[n_{1}-n_{1,\pm}]\big\}\\
&-\vep^{-1}\big\{[\phi_{1,\vep}-\phi_{1,\vep}(\pm\infty)]-[\phi_1-\phi_{1,\pm}]\big\}\nonumber\\
&
%+\vep^{-1}\big\{\underbrace{[n_{1}-n_{1,\pm}]-[\phi_1-%\phi_{1,\pm}]}_{=0 \text{ by \eqref{3.1.6} }}\big\}
+O(1)|\phi_{1,\vep}-\phi_{1,\vep}(\pm\infty)|.\nonumber
\end{align*}
Then we use Lemma \ref{lem3.1.0} and Lemma \ref{lm3.1.1} to conclude the desired estimates on $r_4$. Similarly, it holds that
$$\begin{aligned}
r_1=&\frac{\delta(1+\vep n_{1,\vep})}{\sqrt{T+1}}\vep^{-1}[\phi_{1,\vep}''-n_{1,\vep}'']+(\frac{1}{\sqrt{T+1}}+n_{1,\vep})n_{1,\vep}'\\
&+\f{\delta(1+\vep n_{1,\vep})n_{1,\vep}\phi_{1,\vep}''}{\sqrt{T+1}}-\f{\delta(1+\vep n_{1,\vep})^2}{2\sqrt{T+1}}(\phi_{1,\vep}')^2\\
=&\frac{\delta(1+\vep n_{1,\vep})}{\sqrt{T+1}}\vep^{-1}[\phi_{1,\vep}''-\phi_{1}''-(n_{1,\vep}''-n_1'')+\underbrace{\phi_1''-n_1''}_{=0 \text { by \eqref{3.1.6}}}]+(\frac{1}{\sqrt{T+1}}+n_{1,\vep})n_{1,\vep}'\\
&+\f{\delta(1+\vep n_{1,\vep})n_{1,\vep}\phi_{1,\vep}''}{\sqrt{T+1}}-\f{\delta(1+\vep n_{1,\vep})^2}{2\sqrt{T+1}}(\phi_{1,\vep}')^2.
\end{aligned}
$$
Notice that each term on the right contains derivatives, so that all the right-hand  terms and hence $r_1$ vanish at $z\to\pm\infty$. Then the estimates on $r_1$ directly follow from Lemma \ref{lem3.1.0} and Lemma \ref{lm3.1.1}. Estimates on other terms can be treated with the help of the  Sobolev inequality; we omit the details of the proof for brevity. The proof of Lemma \ref{lm3.1.2} is then complete.
\end{proof}

Now we are in position to prove Theorem \ref{thm1.2}. We start from the  approximation sequence %$U_{k}=[n_k,\phi_k]$
{$U_{i}=[n_i,\phi_i]$ %$k=0,1,\cdots$)
($i=0,1,\cdots$)} in terms of
%$$
%\left\{
%\begin{aligned}
%&n_{k+1}'=A(z)n_{k+1}+\f{\delta}{\sqrt{T+1}}\phi_{k+1}''+r_1+r_2[n_k,\phi_k]+r_3[n_k,\phi_k],\\
%&-\vep\delta\phi_{k+1}''=n_{k+1}-\phi_{k+1}+r_4+r_5[\phi_k]+r_6[\phi_k],\\
%&n_{k+1}(0)=0, U_0=[0,0].
%\end{aligned}
%\right.
%$$
$$
\left\{
\begin{aligned}
&n_{i+1}'=A(z)n_{i+1}+\f{\delta}{\sqrt{T+1}}\phi_{i+1}''+r_1+r_2[n_i,\phi_i]+r_3[n_i,\phi_i],\\
&-\vep\delta\phi_{i+1}''=n_{i+1}-\phi_{i+1}+r_4+r_5[\phi_i]+r_6[\phi_i],\\
&n_{i+1}(0)=0, U_0=[0,0].
\end{aligned}
\right.
$$
Note that the existence of the sequence %$\{U_k\}_{k\geq 0}$
{$\{U_i\}_{i\geq 0}$}
is assured by Proposition \ref{lm3.2.1}. By induction, we claim to have the uniform bound of %$U_k$
{$U_i$} as
%\begin{align}\label{ad.Ubd}
%\|U_k\|_{\mathbf{X}_{\beta}}\leq K\delta^{-1},\quad k=0,1,\cdots
%\end{align}
\begin{align}\label{ad.Ubd}
\|U_i\|_{\mathbf{X}_{\alpha,k}}\leq K,\quad i=0,1,\cdots
\end{align}
for a suitably chosen constant $K>0$ independent of $\vep$, $\delta$ and %$k$
{$i$}. Indeed, \eqref{ad.Ubd} is obviously true for $i=0$, since $U_0=[0,0]$. To proceed, we assume that  \eqref{ad.Ubd} is true up to %$k\geq 0$
{$i\geq 0$}.
Applying Proposition \ref{lm3.2.1} to %$U=U_{k+1}$
{$U=U_{i+1}$}
with
%$$
%h_1=r_1+r_2[n_k,\phi_k]+r_3[n_k,\phi_k],\quad h_2=r_4+r_5[\phi_k]+r_6[\phi_k],
%$$
$$
h_1=r_1+r_2[n_i,\phi_i]+r_3[n_i,\phi_i],\quad h_2=r_4+r_5[\phi_i]+r_6[\phi_i],
$$
and further using Lemma \ref{lm3.1.2} to estimate the right-hand side of \eqref{3.2.2} as
\begin{align}
\|h_1\|_{H^{k-1}_{\alpha}}&\leq \|r_1\|_{H^{k-1}_\alpha}+\|r_2\|_{H^{k-1}_{\alpha}}+\|r_3\|_{H^{k-1}_\alpha}\nonumber\\
&\leq C_k+C_k\sqrt{\vep}\|U_i\|_{\mathbf{X}_{\alpha,k}}\left(1+\|U_i\|^3_{\mathbf{X}_{\alpha,k}}\right),\nonumber\\
\|r_5[\phi_i]\|_{H^{k}_\alpha}+\delta\left\|\f{\dd^{k+1}r_5}{\dd z^{k+1}}\right\|_{L^2_\alpha}&\leq C_k\vep\|\phi_i\|_{H^{k}_\alpha}+C_k\vep\delta\left\|\f{\dd^{k+1}\phi_i}{\dd z^{k+1}}\right\|_{L^2_\alpha}\leq C_k\sqrt{\vep}\|U_i\|_{\mathbf{X}_{\alpha,k}},\nonumber\\
\|r_6[\phi_i]\|_{H^{k}_\alpha}+\delta\left\|\f{\dd^{k+1}r_6}{\dd z^{k+1}}\right\|_{L^2_\alpha}&\leq C_k\vep^2e^{\vep^2\|\phi_i\|_{H^{k}_{\alpha}}}\left(\|\phi_i\|^2_{H^{k}_\alpha}+\delta\left\|\f{\dd^{k+1}\phi_i}{\dd z^{k+1}}\right\|^2_{L^2_\alpha}\right)\nonumber\\
&\leq C_k\vep e^{\vep^2\|\phi_i\|_{H^{k}_{\alpha}}}\|U_i\|_{\mathbf{X}_{\alpha,k}}^2,\nonumber
\end{align}
one can conclude that %$\|U_{k+1}\|_{\mathbf{X}_\beta}$
$\|U_{i+1}\|_{\mathbf{X}_{\alpha,k}}$
is bounded by
%
%\begin{equation}\label{3.3.1}
%\|U_{k+1}\|_{\mathbf{X}_\beta}\leq
%C_1\delta^{-1}
%+C_{\delta}\sqrt{\vep}\|U_k\|_{\mathbf{X}_{\beta}}\left\{1+\|U_{k}\|_{\mathbf{X}_{\beta}}
%\left(1+e^{\vep^2\|U_k\|_{\mathbf{X}_{\beta}}}\right)+\|U_k\|_{\mathbf{X}_{\beta}}^3\right\},
%\end{equation}
\begin{equation}\label{3.3.1}
%\|U_{k+1}\|_{\mathbf{X}_\beta}\leq
B
+C\sqrt{\vep}\|U_i\|_{\mathbf{X}_{\alpha,k}}\left\{1+\|U_{i}\|_{\mathbf{X}_{\alpha,k}}
\left(1+e^{\vep^2\|U_i\|_{\mathbf{X}_{\alpha,k}}}\right)+\|U_i\|_{\mathbf{X}_{\alpha,k}}^3\right\},
\end{equation}
for a generic constant $B>0$ independent of $\vep$, %and
$\delta$ and $i$. In terms of the induction hypothesis, it follows from \eqref{3.3.1} that
$$
\|U_{i+1}\|_{\mathbf{X}_{\alpha,k}}\leq K
$$
by taking $K=2B$ and $\vep>0$ small enough. This then proves \eqref{ad.Ubd}.

By a similar argument, one
can further show that  the estimate
$$
\|U_{i+1}-U_{i}\|_{\mathbf{X}_{\alpha,k}}\leq \frac{1}{2}\|U_{i}-U_{i-1}\|_{\mathbf{X}_{\alpha,k}}
$$
holds true for all $i\geq 1$, provided that $\vep>0$ is small enough. Thus, $\{U_i\}_{i\geq 0}$ is a Cauchy sequence in $\mathbf{X}_{\alpha,k}$, and hence there is $U\in \mathbf{X}_{\alpha,k}$ such that $U_i\rightarrow U$  as $i\to\infty$ in terms of the norm of $\mathbf{X}_{\alpha,k}$. It is straightforward to check that the limit function $U:=[n_R,\phi_R]$  solves the problem \eqref{3.1.13} and satisfies \eqref{thm1.2.estnp} by choosing $C_k=K$.

Once $n_R$ is solved, $u_R$ can be solved according to \eqref{3.1.u} and it follows that
$$
\|u_R\|_{H^{k}_{\alpha}}\leq C\|[n_{R},\phi_R\|_{\mathbf{X}_{\alpha,k}}+C. %\leq CK+C.
$$
This then proves \eqref{thm1.2.estu}  due to \eqref{thm1.2.estnp} by re-choosing $C_k$ suitably large.
%It is direct to check $[n_\vep,u_{\vep},\phi_{\vep}]$ defined in \eqref{1.1.3} satisfies \eqref{1.1.2-1}.
Therefore, \eqref{1.1.3} is justified with the uniform estimates \eqref{thm1.2.estnp} and \eqref{thm1.2.estu} for the remaining terms. The proof of Theorem \ref{thm1.2} is complete. \qed
%\end{subsection}

\section{Dynamical stability of shock profiles}\label{sec4}

In this section we turn to the proof of Theorem \ref{thm1.3} for the large time asymptotic stability of the smooth small-amplitude shock profile obtained in Theorem \ref{thm1.1} under suitably small smooth perturbations. The proof is based on the anti-derivative technique and the elementary energy method. Compared to the classical result for the Navier-Stokes equations, the main difficulty is to treat the extra effect of the self-consistent force.

%\begin{subsection}
\subsection{Reformulation}
%We

Recall the coordinate $(t,y)=(t,x-st)$. We define the perturbation around the shock profile $[\b{v},\b{u},\b{\phi}](y)$ as
\begin{align}
[\tilde{v},\tilde{u},\tilde{\phi}]:=[v-\bar{v}, u-\bar{u}, \phi-\bar{\phi} ]. \nonumber
\end{align}
Then by \eqref{4.2}, $[\tilde{v},\tilde{u},\tilde{\phi}](t,y)$ satisfies
\begin{equation}\label{4.1.1}
\left\{
\begin{aligned}
&\tilde{v}_{t}-s\tilde{v}_{y}-\tilde{u}_{y}=0,\\
&\tilde{u}_{t}-s\tilde{u}_{y}+T\left(\frac{1}{v}-\frac{1}{\bar{v}}\right)_{y}
=\mu\left(\frac{{u}_{y}}{v}-
\frac{\bar{u}_{y}}{\bar{v}}\right)_{y}-\left(\frac{\phi_{y}}{v}-\frac{\bar{\phi}_{y}}{\bar{v}}\right),\\
&-\lambda^{2}\left(\frac{\phi_y}{v}-\frac{\bar{\phi}_y}{\bar{v}}\right)_{y}=\bar{v}e^{\bar{\phi}}-ve^{\phi}.
\end{aligned}
\right.
\end{equation}
As for obtaining \eqref{4.3c}, the second equation of \eqref{4.1.1} can be rewritten as
\begin{multline}\label{4.1.2}
%\left\{
%\begin{aligned}
%&\tilde{v}_{t}-s\tilde{v}_{y}-\tilde{u}_{y}=0,\\
\tilde{u}_{t}-s\tilde{u}_{y}+(T+1)\left(\frac{1}{v}-\frac{1}{\bar{v}}\right)_{y}
-\mu\left(\frac{{u}_{y}}{v}-
\frac{\bar{u}_{y}}{\bar{v}}\right)_{y}\\
=\f{\lambda^{2}}{2}\left[\left(\frac{\phi_{y}}{v}\right)^{2}-
	\left(\frac{\bar{\phi}_{y}}{\bar{v}}\right)^{2}\right]_{y}-
\lambda^{2}\left[\frac{1}{v}\left(\frac{\phi_{y}}{v}\right)_{y}-\frac{1}{\bar{v}}
	\left(\frac{\bar{\phi}_{y}}{\bar{v}}\right)_{y}\right]_{y}.
%\\
%&-\lambda^{2}\left(\frac{\phi_y}{v}-\frac{\bar{\phi}_y}{\bar{v}}\right)_{y}=\b{v}e^{\b{\phi}}-{v}e^{{\phi}}.\\
%\end{aligned}
%\right.
\end{multline}
%and by \eqref{4.3} that
%\begin{equation}\label{4.1.2}
%\left\{
%\begin{aligned}
%&\tilde{v}_{t}-s\tilde{v}_{y}-\tilde{u}_{y}=0,\\
%&\tilde{u}_{t}-s\tilde{u}_{y}+(T+1)\left(\frac{1}{v}-\frac{1}{\bar{v}}\right)_{y}
%-\mu\left(\frac{{u}_{y}}{v}-
%\frac{\bar{u}_{y}}{\bar{v}}\right)_{y}\\
%&\qquad\qquad\qquad=\f{\lambda^{2}}{2}\left[\left(\frac{\phi_{y}}{v}\right)^{2}-
%	\left(\frac{\bar{\phi}_{y}}{\bar{v}}\right)^{2}\right]_{y}-
%\lambda^{2}\left[\frac{1}{v}\left(\frac{\phi_{y}}{v}\right)_{y}-\frac{1}{\bar{v}}
%	\left(\frac{\bar{\phi}_{y}}{\bar{v}}\right)_{y}\right]_{y},\\
%&-\lambda^{2}\left(\frac{\phi_y}{v}-\frac{\bar{\phi}_y}{\bar{v}}\right)_{y}=\b{v}e^{\b{\phi}}-{v}e^{{\phi}}.\\
%\end{aligned}
%\right.
%\end{equation}
Recall \eqref{1.1.4} for a formal definition of $[\Phi,\Psi]$. Then, from \eqref{4.1.1} as well as \eqref{4.1.2}, by formally taking integration of $[\tilde{v},\tilde{u}](t,\cdot)$  from $-\infty$ to $y$,
%where in \eqref{4.1.2}, we have used $\eqref{4.3}_3$ and replaced $e^{\phi}$ by $v^{-1}+\lambda^2 v^{-1}\pa_x(v^{-1}\pa_x\phi)$. Recall from \eqref{1.1.4} that $
%[\tilde{v},\tilde{u}]=[\Phi_y, \Psi_y].$ Then we integrate $\eqref{4.1.2}_1$, $\eqref{4.1.2}_2$ from $-\infty$ to $y$ and obtain the following equation for
$[\Phi,\Psi,{\phi}](t,y)$ satisfies
\begin{equation}\label{4.1.3}
\left\{
\begin{aligned}
&\Phi_{t}-s\Phi_{y}-\Psi_{y}=0,\\
&\Psi_{t}-s\Psi_{y}+(T+1)\left(\frac{1}{\bar{v}+\Phi_y}-\frac{1}{\bar{v}}\right)
-\mu\left(\frac{\Psi_{yy}+\bar{u}_{y}}{\bar{v}+\Phi_y}-
\f{\bar{u}_y}{{\bar{v}}}\right)\\
&\qquad\qquad\qquad=\f{\lambda^{2}}{2}\left[\left(\frac{\phi_{y}}{v}\right)^{2}-
	\left(\frac{\bar{\phi}_{y}}{\bar{v}}\right)^{2}\right]-
\lambda^{2}\left[\frac{1}{v}\left(\frac{\phi_{y}}{v}\right)_{y}-\frac{1}{\bar{v}}
	\left(\frac{\bar{\phi}_{y}}{\bar{v}}\right)_{y}\right],\\
&-\lambda^{2}\left(\frac{\phi_y}{v}-\frac{\bar{\phi}_y}{\bar{v}}\right)_{y}=\b{v}e^{\b{\phi}}-{v}e^{\phi},
\end{aligned}
\right.
\end{equation}
supplemented with the initial data of $[\Phi,\Psi]$ given in \eqref{def.id.PP}. We regard the Cauchy problem \eqref{4.1.3} and \eqref{def.id.PP} on $[\Phi,\Psi,\phi](t,y)$ as an auxiliary problem for obtaining the  existence of the original solution $[v,u,\phi](t,y)$ by defining
\begin{equation}
\label{def.relvu}
[v,u]=[\bar{v},\bar{u}]+[\Phi_y,\Psi_y],
\end{equation}
and the uniqueness of solutions $[v,u,\phi]$ in the prescribed function space can be independently proved. Thus there is actually no need to justify if the right-hand terms of  \eqref{1.1.4} are well defined. Since it is a standard procedure, in what follows we will only focus on the existence of smooth solutions to  the Cauchy problem \eqref{4.1.3} and \eqref{def.id.PP} by the energy method.

%Note that once the uniqueness of \eqref{4.1.1} is shown, the solutions to \eqref{4.1.1}, \eqref{4.1.2} and \eqref{4.1.3} are coincide.
%\end{subsection}

%\begin{subsection}
\subsection{A priori estimates}
We are now devoted to obtaining the a priori estimates of solutions to the Cauchy problem \eqref{4.1.3} and \eqref{def.id.PP}.

\begin{proposition}\label{prop4.1}
Let $M>0$ be an arbitrary constant and $[\Phi,\Psi,\tilde{\phi}]$ be a smooth solution to the Cauchy problem \eqref{4.1.3} on $[0,M]$ with initial data $[\Phi_0,\Psi_0]\in H^2$. Then there exist positive constants {$e_1$ and $\tilde{\vep}_1$}
%$\varepsilon_2$ and $\delta_2$ %and $C_2$
independent of $M$ such that if
\begin{equation}\label{4.2.1}
\sup_{0\leq t\leq M}\|[\Phi,\Psi,\tilde{\phi}](t)\|_{H^2}\leq {e_1}%\varepsilon_2,%\quad t\in[0,T]
\end{equation}
and
\begin{equation}\label{4.2.2}
|v_+-v_-|\leq {\tilde{\vep}_1} %\vep,
\end{equation}
then it holds that
\begin{multline}\label{4.2.3}
\|[\Phi,\Psi,\tilde{\phi}](t)\|_{H^2}^2+\int_0^t\|\sqrt{s\b{v}\bar{v}_y}\Psi(\tau)\|_{L^2}^2\\
%+\|\tilde{\phi}_t(\tau)\|_{H^1}^2
+\|[\Phi_y,\tilde{\phi}_t](\tau)\|_{H^1}^2+\|[\Psi_y,\tilde{\phi}](\tau)\|_{H^2}^2\dd \tau
\leq C\|[\Phi_0,\Psi_0]\|_{H^2}^2,
\end{multline}
for all $t\in[0,M]$.
\end{proposition}

We will devote the rest of this subsection to prove Proposition \ref{prop4.1}. Firstly we estimate the zero-order energy of $[\Phi,\Psi,\tilde{\phi}]$. For this, we rewrite \eqref{4.1.3} as follows:
\begin{equation}\label{4.2.4}
\left\{
\begin{aligned}
&\Phi_{t}-s\Phi_{y}-\Psi_{y}=0,\\
&\Psi_{t}-s\Psi_{y}-\frac{(T+1)}{\b{v}^{2}}\Phi_{y}-\frac{\mu}{\b{v}}\Psi_{yy}=-\lambda^{2}\bigg[\f{1}{v}
\left(\frac{{\phi}_{y}}{v}\right)_y-\f{1}{\bar{v}}\left(\f{\bar{\phi}_y}{\bar{v}}\right)_y\bigg]+\CJ_{1}+\CN_{1},\\
&-\lambda^{2}\left(\frac{\phi_y}{v}-\frac{\bar{\phi}_y}{\bar{v}}\right)_{y}=-e^{\b{\phi}}\Phi_y-\bar{v}e^{\bar{\phi}}\tilde{\phi}+\CN_2,
\end{aligned}
\right.
\end{equation}
where we have denoted
\begin{align}
\CJ_1&\equiv\lambda^2\bigg(\f{\bar{\phi}_y\tilde{\phi}_y}{v^2}-\f{\bar{\phi}_{y}^2\Phi_y}{v^2\bar{v}}\bigg)-\f{\mu \b{u}_y\Phi_y}{v\b{v}},\notag\\
\CN_1&\equiv\f{-(T+1)\Phi_y^2}{\bar{v}^2v}-\f{\mu\Psi_{yy}\Phi_y}{v\bar{v}}+\f{\lambda^2}{2}\bigg(\f{\tilde{\phi}_y^2}{v^2}-\f{\bar{\phi}^2_y
\Phi_y^2}{v^2\bar{v}^2}\bigg),\nonumber\\
\CN_2&\equiv e^{\bar{\phi}}(1-e^{\tilde{\phi}})\Phi_y+\bar{v}e^{\bar{\phi}}(1-e^{-\tilde{\phi}}+\tilde{\phi})\nonumber.
\end{align}

\begin{lemma}\label{lem4.2}
%There exist positive constants $\varepsilon_3$ and $\delta_3$, such that if
%\begin{align}\label{4.2.5}
%|v_+-v_-|\leq \delta_3,
%\end{align}
%and for $t\in[0,T]$,
%\begin{align}\label{4.2.6}
%\CE(t)\leq \varepsilon_3,
%\end{align}
%where $\CE(t)$ is defined by \eqref{E}, then
Under the assumptions of Proposition \ref{prop4.1}, it holds that
\begin{align}
&\|[\Phi,\Psi,\tilde{\phi},\tilde{\phi}_y](t)\|_{L^{2}}^2+\int_0^t\left\{\|\sqrt{s\b{v}\bar{v}_y}\Psi(\tau)\|_{L^2}^2+\|\Psi_y(\tau)\|_{L^2}^2
\right\}\dd \tau\notag \\
&\leq C \|[\Phi_0,\Psi_0]\|_{L^2}^2+C\|\tilde{\phi}(0)\|^2_{H^1}\nonumber\\
&\quad+ C %\{\sup_{0\leq \tau\leq t}\sqrt{\CE(\tau)}+|v_+-v_-|\}
%(\varepsilon_2+\delta_2)
{(e_1+\tilde{\vep}_1)}\int_{0}^t\left\{\|[\Phi_y,\Psi_y](\tau)\|_{H^1}^2+\|\tilde{\phi}(\tau)\|_{H^2}^2+\|\tilde{\phi}_t(\tau)\|_{H^1}^2\right\}
\dd \tau,\label{4.2.8}
\end{align}
for all $t\in [0,M]$.
\end{lemma}

\begin{proof}
Firstly, it holds from Sobolev embedding $H^1(\R)\hookrightarrow L^{\infty}(\R)$ as well as the a priori assumption \eqref{4.2.1} that
\begin{align}\label{4.2.7}
\|[\Phi,\Psi,\tilde{\phi}](t)\|_{L^{\infty}}+\|[\Phi_y,\Psi_y,\tilde{\phi}_y](t)\|_{L^{\infty}}%\leq C\sqrt{\CE(t)}\leq
%C{\vep_2}.
\leq {C e_1},
\end{align}
with a generic constant $C>0$.
Then, for %$\vep_2>0$
{$e_1$} suitably small, in terms of \eqref{def.relvu} and $\tilde{v}=\Phi_y$, we have
%\begin{equation}\label{4.2.7-1}
%\underline{V}\leq v=\b{v}+\tilde{v}\leq \overline{V},
%\end{equation}
\begin{equation}\label{4.2.7-1}
\ubar{V}\leq v=\b{v}+\tilde{v}\leq \bar{V},
\end{equation}
for two positive constants %$\overline{V},\underline{V}>0$.
{$\ubar{V}$, $\b{V}>0$}. Multiplying the first and second equations of $\eqref{4.2.4}$
%and $\eqref{4.2.4}_2$
by $(T+1)\Phi$ and $\bar{v}^2\Psi$ respectively and adding them up, we have
\begin{multline}\label{4.2.9}
%\f{\pa}{\pa t}
\left\{\f{T+1}{2}\Phi^2+\f{\b{v}^2}{2}\Psi^2\right\}_t+\{\cdots\}_y+s\b{v}\b{v}_y\Psi^2+\m\b{v}\Psi_y^2\\
=-\mu\b{v}_y\Psi\Psi_y+(\CJ_1+\CN_1)\b{v}^2\Psi-\lambda^2\b{v}^2\Psi\left[\f{1}{v}\left(\f{\phi_y}{v}\right)_y-\f{1}{\b{v}}\left(\f{\b{\phi}_y}{\b{v}}\right)_y\right],
\end{multline}
where the second term $\{\cdots\}_y$ on the left stands for the total derivative term and will disappear after taking integration with respect to $y$. Note that the coefficient of $\Psi^2$ in the third term on the left is positive due to \eqref{4.4} for the compressibility of the shock profile. Now we estimate the right-hand side of \eqref{4.2.9} term by term. By Cauchy-Schwarz, the first term is bounded as
\begin{align}\label{4.2.10}
|\mu\b{v}_y\Psi\Psi_y|\leq \eta s\b{v}\b{v}_y\Psi^2+C_{\eta}|v_+-v_-|\cdot |\Psi_y|^2,
\end{align}
with an arbitrary constant $0<\eta<1$ to be chosen later.
The second term on the right-hand side of \eqref{4.2.9} comes from inhomogeneous and nonlinear contributions. Therefore, it holds from \eqref{4.4}, \eqref{4.2.7} and \eqref{4.2.7-1} that
\begin{equation}\label{4.2.11}
\big|(\CJ_1+\CN_1)\b{v}^2\Psi\big|\leq \eta s\b{v}\b{v}_y\Psi^2+\{C_\eta|v_+-v_-|+C\sqrt{\CE(t)}\}\left(\Phi_y^2+\Psi_{yy}^2+\tilde{\phi}_y^2\right).
\end{equation}
%where $\eta>0$ can be chosen arbitrarily small.
Here and in the sequel we have used the notation $\CE(t)$ given in \eqref{E}.
To estimate the last term on the right-hand side of \eqref{4.2.9}, we first rewrite it as
\begin{equation}\label{4.2.12}
\lambda^2\bigg[\f{1}{v}\left(\f{\phi_y}{v}\right)_y-\f{1}{\b{v}}\left(\f{\b{\phi}_y}{\b{v}}\right)_y\bigg]\b{v}^2\Psi
=\lambda^2\bigg[\left(\f{\phi_y}{v^2}-\f{\b{\phi}_y}{\b{v}^2}\right)\b{v}^2\Psi\bigg]_y-\lambda^2\tilde{\phi}_y\Psi_y+\CN_3,
\end{equation}
where $\CN_3$ is denoted by
\begin{equation}
\CN_3\equiv-\lambda^2\b{v}^2\left(\f{1}{v^2}-\f1{\b{v}^2}\right)\phi_y\Psi_y-2\lambda^2\b{v}\b{v}_y\Psi\left(\f{\phi_y}{v^2}-\f{\b{\phi}_y}{\b{v}^2}\right)+
\lambda^2\b{v}^2\Psi\left(\f{v_y\phi_y}{v^3}-\f{\b{v}_y\b{\phi}_y}{\b{v}^3}\right).\nonumber
\end{equation}
Using \eqref{4.2.7} and \eqref{4.2.7-1}, it is direct to show that $\CN_3$ is bounded by
\begin{align}\label{4.2.13}
|\CN_3|\leq \eta|\b{v}_y|\Psi^2+\{C_{\eta}|v_+-v_-|+C\sqrt{\CE(t)}\}\left(\Phi_{y}^2+\Psi_y^2+\tilde{\phi}_y^2+\Phi_{yy}^2\right).
\end{align}
Next, substituting the first equation of \eqref{4.2.4} into the second term on the right-hand side of \eqref{4.2.12}, one has
\begin{eqnarray}
%\begin{aligned}
-\lambda^2\tilde{\phi}_y\Psi_y&=&-\lambda^2\tilde{\phi}_y\Phi_t+s\lambda^2\tilde{\phi}_y\Phi_y\notag\\
&=&-\lambda^2%\f{\dd}{\dd t}
\left(\tilde{\phi}_y\Phi\right)_t+\lambda^2\tilde{\phi}_{yt}\Phi+s\lambda^2\tilde{\phi}_y\Phi_y\notag\\
&=&-\lambda^2%\f{\dd}{\dd t}
\left(\tilde{\phi}_y\Phi\right)_t+\lambda^2\left(\tilde{\phi}_t\Phi\right)_y-\lambda^2\tilde{\phi}_t\Phi_y+s\lambda^2\tilde{\phi}_y\Phi_y.\label{4.2.14}
%\end{aligned}
\end{eqnarray}
Now it remains to deal with the last two terms in the last line of \eqref{4.2.14}. For this, one should turn to the Poisson equation. In fact, %we solve $\Phi_y$ from $\eqref{4.2.4}_3$ as
it follows from the third equation of \eqref{4.2.4} that
\begin{equation}\label{4.2.15}
\Phi_y=\lambda^2e^{-\b{\phi}}\left(\f{\phi_y}{v}-\f{\b{\phi}_y}{\b{v}}\right)_y-\b{v}\tilde{\phi}+e^{-\b{\phi}}\CN_2.
\end{equation}
Then one has
\begin{align}\label{4.2.16}
-\lambda^2\tilde{\phi}_t\Phi_y&=-\lambda^4e^{-\b{\phi}}\left(\f{\phi_y}{v}-\f{\b{\phi}_y}{\b{v}}\right)_y\tilde{\phi}_t
+\lambda^2\b{v}\tilde{\phi}_t\tilde{\phi}-\lambda^2e^{-\b{\phi}}\tilde{\phi}_t\CN_2\nonumber\\
&=%\f{\dd}{\dd t}
\bigg[\f{\lambda^2\b{v}}{2}\tilde{\phi}^2+\f{\lambda^4e^{-\b{\phi}}}{2\b{v}}\tilde{\phi}_y^2\bigg]_t-
\bigg[\lambda^4e^{-\b{\phi}}\left(\f{\phi_y}{v}-\f{\b{\phi}_y}{\b{v}}\right)\tilde{\phi}_t\bigg]_y+\CI_1,
\end{align}
where  $\CI_1$ is denoted by
$$
\CI_1\equiv\lambda^4e^{-\b{\phi}}\bigg[\left(\f1v-\f{1}{\bar{v}}\right)\phi_y\tilde{\phi}_{ty}-\b{\phi}_y\left(
\f{\phi_y}{v}-\f{\b{\phi}_y}{\b{v}}\right)\tilde{\phi}_t\bigg]-\lambda^2e^{-\b{\phi}}\tilde{\phi}_t\CN_2.
$$
Using \eqref{4.2.7} for %$\vep_2>0$
{$\tilde{\vep}_1$} suitably small, it is straightforward to bound $\CI_1$ by
\begin{equation}\label{4.2.17}
|\CI_1|\leq C\{|v_+-v_-|+\sqrt{\CE(t)}\}\left(|\tilde{\phi}|^2+|\tilde{\phi}_t|^2+|\tilde{\phi}_{ty}|^2+|\Phi_y|^2\right).
\end{equation}
%for suitably small $\varepsilon_3>0$.
In the same way as before, the last term on the right-hand side of \eqref{4.2.14} can be computed as %follows.
\begin{align}\label{4.2.18}
s\lambda^2\tilde{\phi}_y\Phi_y&=s\lambda^4e^{-\b{\phi}}\left(\f{\phi_y}{v}-\f{\b{\phi}_y}{\b{v}}\right)_y\tilde{\phi}_y
-s\lambda^2\b{v}\tilde{\phi}\tilde{\phi}_y+s\lambda^2e^{-\b{\phi}}\CN_2\tilde{\phi}_y\nonumber\\
&=\bigg[s\lambda^4e^{-\b{\phi}}\left(\f{\phi_y}{v}-\f{\phi_y}{\b{v}}\right)\tilde{\phi}_y+
\f{s\lambda^4e^{-\b{\phi}}\tilde{\phi}_y^2}{2\b{v}}-\f{s\lambda^2\b{v}\tilde{\phi}^2}{2}\bigg]_y+\CI_2,
\end{align}
where  $\CI_2$ is denoted by
\begin{multline*}
%$$
\CI_2\equiv s\lambda^4e^{-\b{\phi}}\left\{\left(\f{\b{\phi}_y}{2\b{v}}-{\f{\b{v}_y}{2\b{v}^2}}\right)\tilde{\phi}_y^2+
\left(\f{\phi_y}{v}-\f{\phi_y}{\b{v}}\right)\left(\b{\phi}_y\tilde{\phi}_y-\tilde{\phi}_{yy}\right)
\right\}\\
+s\lambda^2e^{-\b{\phi}}\CN_2\tilde{\phi}_y+
\f{s\lambda^2\b{v}_y\tilde{\phi}^2}{2}.
%$$
\end{multline*}
%and
One can bound $\CI_2$ by
\begin{align}\label{4.2.18-1}
|\CI_2|\leq C\{|v_+-v_-|+\sqrt{\CE(t)}\}\left(|\tilde{\phi}|^2+|\tilde{\phi}_y|^2+|\tilde{\phi}_{yy}|^2+|\Phi_y|^2\right).
\end{align}
In sum, collecting all the above estimates \eqref{4.2.10}, \eqref{4.2.11}, \eqref{4.2.12}, \eqref{4.2.13}, \eqref{4.2.14}, \eqref{4.2.16}, \eqref{4.2.17}, \eqref{4.2.18} and \eqref{4.2.18-1}, we bound the right-hand side of \eqref{4.2.9} by
\begin{align}\label{4.2.19}
%&\text{R.H.S of \eqref{4.2.9} }\leq
&-%\f{\dd}{\dd t}
\left\{\f{\lambda^2\b{v}\tilde{\phi}^2}{2}+\f{\lambda^4e^{-\b{\phi}}\tilde{\phi}_y^2}{2\b{v}}
-\lambda^2\tilde{\phi}_y\Phi\right\}_t+\{\cdots\}_y
+\eta s\b{v}\b{v}_y\Psi^2\nonumber\\
&+\{C_{\eta}|v_+-v_-|+C\sqrt{\CE(t)}\}\bigg\{\sum_{i=1}^2\big|\pa^i_y{[\Phi,\Psi]}\big|^2+\sum_{i=0}^2\big|\pa^i_y\tilde{\phi}\big|^2
+\sum_{i=0}^1\big|\pa_y^i\tilde{\phi}_t\big|\bigg\},
\end{align}
with an arbitrary constant $0<\eta<1$ to be chosen later.
%where $\eta>0$ can be chosen arbitrary small.
Substituting \eqref{4.2.19} into \eqref{4.2.9}, integrating it with respect to $y$, and taking a suitably small constant $\eta>0$, one obtains that
\begin{multline}\label{4.2.20}
%\f{\dd}{\dd t}\CE_1(t)
{\CE_1'(t)}+\int_{\mathbb{R}}\left(s\b{v}\b{v}_y\Psi^2+\mu\b{v}\Psi_y^2\right)\dd y\\
\leq C\{|v_+-v_-|+\sqrt{\CE(t)}\}\left(\|[\Phi_y,\Psi_y,\tilde{\phi}_t]\|_{H^1}^2+\|\tilde{\phi}\|_{H^2}^2\right),
\end{multline}
for all $t\in[0,M]$, where $\CE_1(t)$ is denoted by
\begin{align}\label{E1}
\CE_1(t)\equiv\int_{\mathbb{R}}\left(\f{\b{v}^2\Psi^2}{2}+\f{\lambda^2\b{v}\tilde{\phi}^2}{2}+\f{(T+1)\Phi^2}{2}+
\f{\lambda^4e^{-\b{\phi}}\tilde{\phi}^2_y}{2\b{v}}-\lambda^2\tilde{\phi}_y\Phi\right)\dd y.
\end{align}
%\marginpar{\gr{Here we number above equation as $E1$.}}
Finally, one can check that $\CE_1(t)$ is a nonnegative energy functional. Indeed, by the Poisson equation, one has
\begin{equation}
\label{ad.est.p}
\left|\b{v}^{-1}e^{-\b{\phi}}-1\right|=\left|\lambda^2\b{v}^{-1}e^{-\b{\phi}}(\b{v}^{-1}\b{\phi}_y)_y\right|\leq C|v_+-v_-|.
\end{equation}
%for some constants $C>0$.
Then, due to \eqref{4.2.2} with %$\delta_2>0$
{$\tilde{\vep}_1$} suitably small,
%for $|v_+-v_-|\ll 1$,
the quadratic integrand of $\CE_1(t)$ has a lower bound as
\begin{equation}
\label{ad.qfT}
\f{(T+1)\Phi^2}{2}+
\f{\lambda^4e^{-\b{\phi}}\tilde{\phi}^2_y}{2\b{v}}-\lambda^2\tilde{\phi}_y\Phi\geq c
%\gr{\kappa}
\left(|\Phi|^2+|\tilde{\phi}_y|^2\right),
\end{equation}
for a generic constant $c>0$. {Here we have essentially used the condition $T>0$.} Therefore, \eqref{4.2.8} follows from integrating \eqref{4.2.20} over $[0,t]$. This completes the proof of Lemma \ref{lem4.2}.
\end{proof}

Next, we need to derive the dissipation terms $\|[\Phi_y,\tilde{\phi},\tilde{\phi}_y,\tilde{\phi}_{yy}]\|_{L^2}$ as well as the dissipation of $\tilde{\phi}$ in $H^2$.

\begin{lemma}\label{lem4.3}
Under the assumptions of Proposition \ref{prop4.1},
%There exists constant $C>0$, such that
it holds that
\begin{align}
&\|\Phi_y(t)\|_{L^2}^2+\int_0^t\left\{\|\Phi_y(\tau)\|_{L^2}^2+\|\tilde{\phi}(\tau)\|_{H^2}^2\right\}\dd \tau\notag\\
&\leq C\|[\Phi_{0y},\Psi_0]\|_{L^2}^2 %+\|\Psi_0\|_{L^2}^2
%\nonumber\\
%&\quad
+%C(\varepsilon_2+\delta_2)
{C(e_1+\tilde{\vep}_1)}
%\{\sup_{0\leq \tau\leq t}\sqrt{\CE(\tau)}+|v_+-v_-|\}
\int_0^t\|[\Phi_{yy},\Psi_{yy}](\tau)\|_{L^2}^2
%+\|\Phi_{yy}(\tau)\|_{L^2}^2
%+\|\tilde{\phi}(\tau)\|_{H^2}^2
%\right\}
\dd \tau\nonumber\\
&\quad+C\left(\|\Psi(t)\|_{L^2}^2
+\int_0^t\left\{\|\sqrt{s\b{v}\b{v}_y}\Psi(\tau)\|_{L^2}^2
+\|\Psi_y(\tau)\|_{L^2}^2\right\}\dd \tau\right),
\label{4.2.21}
\end{align}
and
\begin{equation}\label{4.2.21-1}
\|\tilde{\phi}(t)\|_{H^2}^2\leq C\|\Phi_y(t)\|_{L^2}^2%C(\varepsilon_2+\delta_2)
+{C(e_1+\tilde{\vep}_1)}
%\{|v_+-v_-|+\sqrt{\CE(t)}\}
%\left\{
\|\Phi_{yy}(t)\|_{L^2}^2,
%+\|\tilde{\phi}(t)\|_{H^2}^2\right\},
\end{equation}
for all $t\in [0,M]$.
\end{lemma}

\begin{proof}
Differentiating the first equation of \eqref{4.1.3} with respect to $y$, one has
\begin{align}\label{4.2.22-0}
\Phi_{ty}-s\Phi_{yy}-\Psi_{yy}=0.
\end{align}
Then, multiplying \eqref{4.2.22-0} and the second equation of \eqref{4.1.3} by $\Phi_y$ and $-\mu^{-1}\b{v}\Phi_y$ respectively and adding them together, it holds that
\begin{equation}\label{4.2.22}
%\f{\dd}{\dd t}
\left(\f{\Phi_y^2}{2}\right)_t-\left(\f{s\Phi_y^2}{2}\right)_y+\sum_{j=3}^{8}\CI_j=0,
\end{equation}
where $\CI_j$ $(3\leq j\leq 8)$ are denoted by
\begin{equation}\nonumber
\begin{aligned}
&\CI_3\equiv-(T+1)\left(\f{1}{\b{v}+\Phi_y}-\f{1}{\b{v}}\right)\f{\b{v}\Phi_y}{\mu},\quad \CI_4\equiv-\mu^{-1}(\Psi_t-s\Psi_y){\b{v}\Phi_y},\\
&\CI_5\equiv -\Psi_{yy}\Phi_y+\b{v}\Phi_y\left(\f{\Psi_{yy}+\b{u}_y}{\b{v}+\Phi_y}-\f{\b{u}_y}{\b{v}}\right),\quad
\CI_6\equiv\f{\lambda^2\b{v}\Phi_y}{2\mu}\left[\left(\f{\phi_y}{v}\right)^2-\left(\f{\b{\phi}_y}{\b{v}}\right)^2\right],\\
&\CI_7\equiv-\mu^{-1}\lambda^2\b{v}(\b{v}^{-1}\b{\phi}_y)_y(\f1v-\f1{\b{v}})\Phi_y,\quad
\CI_8\equiv-\mu^{-1}\lambda^2v^{-1}\b{v}\Phi_y\left(\f{\phi_y}{v}-\f{\b{\phi}_y}{\b{v}}\right)_y.
\end{aligned}
\end{equation}
Now we estimate $\CI_3$ to $\CI_8$ term by term. Firstly, it holds that
\begin{equation}\nonumber
\CI_3=\f{(T+1)\Phi_y^2}{\mu(\b{v}+\Phi_y)}=\f{(T+1)\Phi_y^2}{\mu\bar{v}}
-\f{(T+1)\Phi_y^3}{\mu(\b{v}+\Phi_y)\b{v}}\geq \f{(T+1)\Phi_y^2}{\mu\bar{v}}-C\sqrt{\CE(t)}\Phi_y^2.
\end{equation}
For $\CI_4$, it is direct to compute %it as follows.
\begin{align}
\CI_4&=-%\f{\dd}{\dd t}
\left(\f{\b{v}\Phi_y\Psi}{\mu}\right)_t+\f{\b{v}\Phi_{yt}\Psi}{\mu}+\f{s\b{v}\Psi_y\Phi_y}{\mu}\notag\\
&=-%\f{\dd}{\dd t}
\left(\f{\b{v}\Phi_y\Psi}{\mu}\right)_t+\f{\b{v}\Psi}{\mu}(s\Phi_{yy}+\Psi_{yy})+\f{s\b{v}\Psi_y\Phi_y}{\mu}\nonumber\\
&=-%\f{\dd}{\dd t}
\left(\f{\b{v}\Phi_y\Psi}{\mu}\right)_t+\{\cdots\}_y-\left\{\f{\b{v}_y}{\mu}(s\Psi\Phi_y+\Psi\Psi_y)
+\f{\b{v}\Psi_y^2}{\mu}\right\}.\label{4.2.22-1}
\end{align}
By Cauchy-Schwarz, the last term of \eqref{4.2.22-1} is bounded by
$$
\eta\Phi_y^2+C_{\eta}\{s\b{v}\b{v}_y\Psi^2+\Psi_y^2\},
$$
where $\eta>0$ can be small enough to be chosen later.
%can be chosen arbitrary small.
As for $\CI_5$ to $\CI_7$, we have
\begin{equation}\nonumber
|\CI_5|=\left|\left(\f{\b{v}}{v}-1\right)\left(\Psi_{yy}\Phi_y+\b{u}_y\Phi_y\right)\right|\leq C\{\sqrt{\CE(t)}+|v_+-v_-|\}\{|\Psi_{yy}|^2+|\Phi_y|^2\},
\end{equation}
and
\begin{equation}\nonumber
|\CI_6|+|\CI_7|\leq C\{\sqrt{\CE(t)}+|v_+-v_-|\}\{|\Phi_{y}|^2+|\tilde{\phi}_y|^2\}.
\end{equation}
Now it remains to estimate $\CI_8$, which is delicate.
% the most delicate term .
Direct computations show that
\begin{multline}\label{4.2.22-2}
\CI_8=-\f{\lambda^2\Phi_y\tilde{\phi}_{yy}}{\mu\b{v}}\\
+\f{\lambda^2\b{v}\Phi_y}{\mu v}
\left\{\left(\f{\phi_{yy}}{\b{v}}-\f{\phi_{yy}}{v}\right)-\f{v}{\b{v}}\left(\f{\tilde{\phi}_{yy}}{v}
-\f{\tilde{\phi}_{yy}}{\b{v}}\right)+\left(\f{\phi_yv_y}{v^2}-
\f{\b{\phi}_y\b{v}_y}{\b{v}^2}\right)\right\}.
\end{multline}
The last term of \eqref{4.2.22-2} is bounded by
$$
C\{\sqrt{\CE(t)}+|v_+-v_-|\}\left\{|\Phi_y|^2+|\Phi_{yy}|^2+|\tilde{\phi}_y|^2+|\tilde{\phi}_{yy}|^2\right\}.
$$
Note that the first term on the right-hand  of \eqref{4.2.22-2} can not be  controlled by directly replacing $\Phi_y$ from \eqref{4.2.15} like what has been done earlier in \eqref{4.2.16}. Indeed, one has to include some estimates on $\tilde{\phi}_{yy}$ simultaneously so that the quadratic form consisting of $\Phi_y$ and $\tilde{\phi}_{yy}$ is strictly positive. For this purpose, multiplying \eqref{4.2.15} by $-\lambda^2\tilde{\phi}_{yy}/\mu\b{v}$, one has
\begin{align}\label{4.2.22-3}
&\f{\lambda^4\tilde{\phi}_{yy}^2}{\mu\b{v}}+\f{\lambda^2\tilde{\phi}_y^2}{\mu}
-\f{\lambda^2\tilde{\phi}_{yy}\Phi_y}{\mu\b{v}}-\left(\f{\lambda^2\tilde{\phi}\tilde{\phi}_y}{\mu}\right)_y\nonumber\\
&=\f{\lambda^2e^{-\b{\phi}}\tilde{\phi}_{yy}}{\mu \b{v}}\left\{
\lambda^2\left(\f{\phi_{yy}}{\b{v}}-\f{\phi_{yy}}{v}\right)\right.\notag\\
&\qquad\qquad\qquad\ \left.+\lambda^2\left(\f{\phi_y v_y}{v^2}-\f{\b{\phi}_y\b{v}_y}{\b{v}^2}\right)+\lambda^2(e^{\b{\phi}}-\b{v}^{-1})\tilde{\phi}_{yy}-\CN_2\right\}.
\end{align}
%Since $|e^{\b{\phi}}-\f1{\b{v}}|=|\b{v}^{-1}\lambda^2(\b{v}^{-1}\b{\phi}_y)_y|\leq C|v_+-v_-|\ll1$,
Due to \eqref{ad.est.p}, the right-hand side of \eqref{4.2.22-3} is bounded by
$$
C\{|v_+-v_-|+\sqrt{\CE(t)}\}\left\{|\tilde{\phi}|^2+|\tilde{\phi}_y|^2+|\tilde{\phi}_{yy}|^2+%|\Phi|^2
{|\Phi_{yy}|^2}+|\Phi_y|^2\right\}.
$$
Collecting all the above estimates for $\CI_3$ to $\CI_8$ as well as  \eqref{4.2.22-3}, it follows from \eqref{4.2.22} that
\begin{align}\label{4.2.22-4}
&%\f{\dd}{\dd t}
\left(\f{\Phi_y^2}2-\f{\b{v}\Phi_y\Psi}{\mu}\right)_t+
\left\{\f{(T+1)\Phi_y^2}{\mu\b{v}}+\f{\lambda^4\tilde{\phi}_{yy}^2}{\mu\b{v}}
-\f{2\lambda^2\tilde{\phi}_{yy}\Phi_y}{\mu\b{v}}\right\}
+\f{\lambda^2\tilde{\phi}_y^2}{\mu}+\{\cdots\}_y\nonumber\\
&\leq \eta|\Phi_y|^2+C_{\eta}\left(s\b{v}\b{v}_y\Psi^2+\Psi_y^2\right)\notag\\
&\quad+C\{\sqrt{\CE(t)}+|v_+-v_-|\}
\left\{\sum_{i=0}^2|\pa_y^i\tilde{\phi}|^2+\sum_{i=1}^2|\pa_y^i\Phi|^2+|\Psi_{yy}|^2\right\},
\end{align}
for an arbitrary constant $0<\eta<1$. Note that the quadratic term  on the left-hand side has the lower bound as
$$
\f{\lambda^4\tilde{\phi}_{yy}^2}{\mu\b{v}}+\f{(T+1)\Phi_y^2}{\mu\b{v}}
-\f{2\lambda^2\tilde{\phi}_{yy}{\Phi_y}}{\mu\b{v}}\geq c
%\gr{\kappa}
\left(\f{\lambda^4{\tilde{\phi}_{yy}^2}}{\mu\b{v}}+\f{(T+1)\Phi_y^2}{\mu \b{v}}\right),
$$
for a generic positive constant $c$.
%$\gr{\kappa}$.
Therefore, integrating \eqref{4.2.22-4} with respect to $y$ and taking $\eta>0$ suitably small, one has
\begin{align}\label{4.2.23}
&\f{\dd}{\dd t}\bigg(\f12\|\Phi_y(t)\|_{L^2}^2-\int_{\mathbb{R}}\f{\b{v}\Phi_y\Psi}{\mu}\dd y\bigg)+
\|\Phi_y\|_{L^2}^2+\|\tilde{\phi}_y\|_{H^1}^2\notag\\
&\leq C\int_{\mathbb{R}}\left(s\b{v}\b{v}_y\Psi^2+\Psi_y^2\right)\dd y\nonumber\\
&\quad+C\{\sqrt{\CE(t)}+|v_+-v_-|\}\left\{\|\Psi_{yy}\|_{L^2}^2+\|\Phi_y\|_{H^1}^2+\|\tilde{\phi}\|_{H^2}^2\right\}.
\end{align}
Moreover, multiplying the third equation of \eqref{4.2.4} by $\tilde{\phi}$, we obtain that
\begin{align}\label{4.2.24}
\b{v}e^{\b{\phi}}\tilde{\phi}^2+\f{\tilde{\phi}_y^2}{v}+\{\cdots\}_y=-\lambda^2\b{\phi}_y
\left(\f1v-\f1{\b{v}}\right)\tilde{\phi}_y-e^{\b{\phi}}\Phi_y\tilde{\phi}+\CN_2\tilde{\phi}. \end{align}
Integrating \eqref{4.2.24} with respect to $y$ and using Cauchy-Schwarz, it holds that
\begin{align}\label{4.2.25}
\|\tilde{\phi}(t)\|_{H^1}^2\leq C\|\Phi_y(t)\|_{L^2}^2+C\sqrt{\CE(t)}\|\tilde{\phi}(t)\|_{L^2}^2.
\end{align}
Recall \eqref{4.2.1} and \eqref{4.2.2}. Then, \eqref{4.2.21} follows from %multiplying \eqref{4.2.25} by a suitably small constant $\ka>0$,  taking the further summation with \eqref{4.2.23}, and also taking integration  over $[0,t]$.
a suitable linear combination of \eqref{4.2.23} and \eqref{4.2.25} as well as letting $e_1$ and $\tilde{\vep}_1$ be small enough.
%Integrating \eqref{4.2.23}+$\ka$\eqref{4.2.25} with suitably small $\ka>0$ over $[0,t]$,
As to the $H^2$ estimate of $\tilde{\phi}$, we note that \eqref{4.2.22-3} gives
\begin{equation}\label{4.2.23-1}
\|\tilde{\phi}_y(t)\|_{H^1}^2\leq C\|\Phi_y(t)\|_{L^2}^2+C\{\sqrt{\CE(t)}+|v_+-v_-|\}\{\|\Phi_y(t)\|_{H^1}^2+
\|\tilde{\phi}\|_{H^2}^2\}.
\end{equation}
Therefore, \eqref{4.2.21-1} follows from combining \eqref{4.2.23-1} and \eqref{4.2.25} and letting $e_1$ and $\tilde{\vep}_1$ be further small enough. The proof for Lemma \ref{lem4.3} is complete.
\end{proof}

Now we are prepared to derive the higher order energy estimates on $[\Phi, \Psi]$ in Lemma \ref{lem4.4} and Lemma \ref{lem4.5} whose proof will be postponed to Section \ref{sec.a4} in Appendix. In fact, with $L^2_tH^2_x$ estimate of $\tilde{\phi}$ on hand, one can regard the terms induced by the self-consistent force $\tilde{\phi}$ as the inhomogeneous sources. %\gr{\cancel{Therefore}, \cancel{by a standard parabolic estimate,}\cancel{ we have the following}} %lemma.

\begin{lemma}\label{lem4.4}
%There exists a constant $C>0$, such that for any $t \in [0,T]$,
Under the assumptions of Proposition \ref{prop4.1}, it holds that
\begin{equation}\label{4.2.26}
\|\tilde{u}(t)\|_{L^2}^2+\int_0^t\|\tilde{u}_y(s)\|_{L^2}^2\dd s %\\
\leq
 %\gr{\lesssim}
 C
\|\tilde{u}_0\|_{L^2}^2+C
\int_0^t %\left(\|\tilde{\phi}_y(s)\|_{L^2}^2+
\|[\Phi_y,\Psi_y,\tilde{\phi}_y](s)\|_{L^2}^2%\right)
\dd s,
\end{equation}
and
\begin{equation}\label{4.2.27}
\|\tilde{u}_y(t)\|_{L^2}^2+\int_0^t\|\tilde{u}_{yy}(s)\|_{L^2}^2\dd s \leq
% \gr{\lesssim}
C \|\tilde{u}_{0y}\|_{L^2}^2 %+%C
%\sup_{0\leq s\leq t}\sqrt{\CE(s)}
%\varepsilon_2
%C{e_1}\int_0^t\|\pa_{yy}\tilde{u}(s)\|_{L^2}^2\dd s
+C
\int_0^t%\left(\|\tilde{v}(s)\|_{H^1}^2+
\|[\tilde{v},\tilde{v}_y,\tilde{u}_y,\tilde{\phi}_y](s)\|_{L^2}^2%+\|\pa_y\tilde{\phi}(s)\|_{L^2}^2
%\right)
\dd s,
\end{equation}
for all $t\in [0,M]$.
\end{lemma}

Next, we derive the energy dissipation term $\|\tilde{v}_y\|_{L^2}$.

\begin{lemma}\label{lem4.5}
%There exists constants $C>0$, such that
Under the assumptions of Proposition \ref{prop4.1}, it holds that
\begin{align}\label{4.2.32}
&\|\tilde{v}_y(t)\|_{L^2}^2+\int_0^t\|\tilde{v}_y(s)\|_{L^2}^2\dd s\notag\\
&\leq C\|[\tilde{v}_{0y},\tilde{u}_0]\|_{L^2}^2
+C\|\tilde{u}(t)\|_{L^2}^2
+Ce_1\int_0^t \|\tilde{u}_{yy}(s)\|_{L^2}^2\dd s \notag\\
&\quad +C\int_0^t\left\{\|\tilde{u}(s)\|_{H^1}^2+\|\tilde{v}(s)\|_{L^2}^2+
\|\tilde{\phi}_y(s)\|_{L^2}^2\right\}\dd s,
\end{align}
for all $t\in [0,M]$.
\end{lemma}

Finally, to close the a priori assumption \eqref{4.2.1}, we need to estimate the time derivative $\tilde{\phi}_t$. In fact, we have the following

\begin{lemma}\label{lem4.6}
%There exists a constant $C>0$, such that
Under the assumptions of Proposition \ref{prop4.1}, it holds that
\begin{equation}\label{4.2.35}
\|\tilde{\phi}_t\|_{H^1}^2\leq C %\{1+|v_+-v_-|+\sqrt{\CE(t)}\}
\{\|\tilde{u}_y\|_{L^2}^2+\|\tilde{v}_{y}\|_{L^2}^2\},
%+C\sqrt{\CE(t)}\|\tilde{\phi}_t\|_{H^1}^2,
\end{equation}
for all $t\in[0,M]$.
\end{lemma}

\begin{proof}
Differentiate the third equation of \eqref{4.1.1} with respect to $t$ and taking the inner product of the resultant equation with $\tilde{\phi}_t$, one has
\begin{multline}\label{4.2.36}
(ve^{\phi}\tilde{\phi}_t,\tilde{\phi}_t)+(\lambda^2v^{-1}\tilde{\phi}_{ty},\tilde{\phi}_{ty})=
\left(\pa_y(-\lambda^2v^{-2}\phi_y\tilde{v}_t),\tilde{\phi}_t\right)+(-e^{\phi}\tilde{v}_t,\tilde{\phi}_t)\\
=\left(\lambda^2v^{-2}\phi_y(s\tilde{v}_y+\tilde{u}_y),\tilde{\phi}_{ty}\right)
+\left(-e^{\phi}(s\tilde{v}_y+\tilde{u}_y),\tilde{\phi}_t\right), \end{multline}
where the first equation of \eqref{4.1.1} has been used for obtaining the second equality. By Cauchy-Schwarz, it is direct to bound the right-hand side of \eqref{4.2.36} by
$$
\{\eta+C\sqrt{\CE(t)}\}\|\tilde{\phi}_t\|_{H^1}^2+\{C_\eta+C|v_+-v_-|+C\sqrt{\CE(t)}\}\{\|\tilde{u}_y\|_{L^2}^2+\|\tilde{v}_y\|_{L^2}^2\}.
$$
Recall \eqref{4.2.1} and \eqref{4.2.2}. Thus, \eqref{4.2.35} follows by taking $\eta>0$ suitably small and also  letting $%\varepsilon_2>0
e_1$ and $\tilde{\vep}_1$ be small enough. The proof of Lemma \ref{lem4.6} is then complete.
\end{proof}

%Now we are in the position of proving Proposition \ref{prop4.1}.

%\medskip
\noindent{\it Proof of Proposition \ref{prop4.1}}: Letting positive constants %$\varepsilon_2>0$ and $\delta_2>0$
{$e_1$ and $\tilde{\vep}_1$} be small enough, a suitable linear  combination of all estimates \eqref{4.2.8}, \eqref{4.2.21}, \eqref{4.2.26}, \eqref{4.2.27}, \eqref{4.2.32} and \eqref{4.2.35} yields that
\begin{equation}\label{4.2.37}
\|[\Phi,\Psi](t)\|_{H^2}^2+\|\tilde{\phi}(t)\|_{H^1}^2+\int_0^t\CD(s)\dd s
\leq
C\|[\Phi_0,\Psi_0]\|_{H^2}^2+C\|\tilde{\phi}(0,\cdot)\|_{H^1}^2,
%+C %\{\sup_{0\leq s\leq t}\sqrt{\CE(s)}+|v_+-v_-|\}
%(\varepsilon_2+\delta_2)\int_0^t\CD(s)\dd s,
\end{equation}
where $\CD(s)$ is defined in \eqref{D}. Since $e_1$ and $\tilde{\vep}_1$ can be further small enough, by \eqref{4.2.21-1} one has
%there exists $\delta_4\leq \delta_3$ and
%$\varepsilon_4\leq \varepsilon_3$, such that if $|v_+-v_-|\leq \delta_4$ and
%$\sup_{0\leq s\leq T}\CE(s)\leq \varepsilon_4 $, then
\begin{equation}\label{4.2.38}
\|\tilde{\phi}(t)\|_{H^2}^2\leq C\|\Phi_y(t)\|_{{H^1}}^2,
\end{equation}
for all $t\in [0,M]$. Then, from  \eqref{4.2.37} together with \eqref{4.2.38}, one has
\begin{equation*}
\|[\Phi,\Psi,\tilde{\phi}](t)\|_{H^2}^2+\int_0^t\CD(s)\dd s
\leq C\|[\phi_0,\Psi_0]\|_{H^2}^2,
%+ C\{\sup_{0\leq s\leq t}\sqrt{\CE(s)}+|v_+-v_-|\}\int_0^t\CD(s)\dd s.
\end{equation*}
which proves \eqref{4.2.3}.
%Then \eqref{4.2.3} follows by taking $\varepsilon_2>0$ and $\delta_2>0$ suitably small.
Therefore, the proof of Proposition \ref{prop4.1} is  complete. \qed

%\end{subsection}
%\begin{subsection}
\subsection{Global existence and large time behavior}
This part is devoted to proving Theorem \ref{thm1.3}. First, the local-in-time existence and uniqueness of solutions $[\Phi,\Psi,\tilde{\phi}]$ to the Cauchy problem on the system \eqref{4.1.3} with initial data $[\Phi_0,\Psi_0]$ can be obtained in a usual way;
%is quite standard and
we omit the details by brevity. Furthermore, by a  continuity argument, the global existence of the solution  $[\Phi,\Psi,\tilde{\phi}]$
%can be proved by
%combining the local existence theory with
follows from the uniform a priori estimates obtained in Proposition \ref{prop4.1}.  %we have done in last subsection.
As a consequence, the solution to \eqref{4.1.1} with the corresponding initial data is given by $[\tilde{v},\tilde{u},\tilde{\phi}]=[\Phi_y,\Psi_y,\tilde{\phi}]$. As mentioned before, we also omit the proof of uniqueness for brevity.
%Next, the uniqueness for \eqref{4.1.1} is standard.
Therefore, it remains to
show the large time behaviour \eqref{LT}. To do this,
%we first show the time derivative $\f{\dd}{\dd t}\|[\tilde{v},\tilde{u},\tilde{\phi}](t)\|_{L^2}^2\in L^1(\mathbb{R}^+)$. In fact, we have
we see from \eqref{4.1.1} as well as \eqref{thm1.3.est} that
%\begin{multline*}
%\int_0^{\infty}|(\tilde{v},\tilde{v}_t)(t)|\dd t=\int_0^{\infty}|(\tilde{v},s\tilde{v}_y+\tilde{u}_y)(t)|\dd t\\
%\leq
%C\int_0^{\infty}\{\|\tilde{v}(t)\|_{H^1}^2+\|\tilde{u}_{y}(t)\|_{L^2}^2\}\dd t\leq C\int_0^{\infty}\CD(t)\dd t,
%\end{multline*}
%\begin{align*}
%&\int_0^{\infty}|(\tilde{u},\tilde{u}_t)(t)|\dd t\notag\\
%%&\leq T\int_0^{\infty}|\big(({v}^{-1}-\b{v}^{-1})_y,\tilde{u}\big)(t)|\dd t
%%+\mu\int_0^{\infty}|\big((v^{-1}u_y-\b{v}^{-1}\b{u}_y)_y,\tilde{u}\big)(t)|\dd t+
%%\int_0^{\infty}|\big(v^{-1}\phi_y-\b{v}^{-1}\b{\phi}_y,\tilde{u}\big)(t)|\dd t\\
%&\leq \int_0^{\infty}\{T|\big({v}^{-1}-\b{v}^{-1},\tilde{u}_y\big)(t)|
%+\mu|\big(v^{-1}u_y-\b{v}^{-1}\b{u}_y,\tilde{u}_y\big)(t)|\notag\\
%&\qquad\qquad+|\big(v^{-1}\phi_y-\b{v}^{-1}\b{\phi}_y,\tilde{u}\big)(t)|\}\dd t\notag\\
%&\leq C\int_0^{\infty}\|\tilde{v}(t)\|_{L^2}^2+\|\tilde{u}_y(t)\|_{L^2}^2+\|\tilde{\phi}_y(t)\|_{L^2}^2\dd t\notag\\
%&\leq  C\int_0^{\infty}\CD(t)\dd t,
%\end{align*}
%and
%$$
%\int_0^{\infty}|(\tilde{\phi},\tilde{\phi}_t)(t)|\dd t\leq C\int_0^{\infty}\{\|\tilde{\phi}(t)\|_{L^2}^2+\|\tilde{\phi}_t(t)\|_{L^2}^2\}\dd t\leq
%C\int_0^{\infty}\CD(t)\dd t.
%$$
%Then, from the above estimates as well as \eqref{thm1.3.est}, it holds that
\begin{multline*}
\int_0^{\infty}\bigg|\f{\dd }{\dd t}\|[\tilde{v},\tilde{u},\tilde{\phi}](t)\|_{L^2}^2\bigg|\dd t
\leq 2\int_0^{\infty}\{|(\tilde{v},\tilde{v}_t)|+|(\tilde{u},\tilde{u}_t)|+|(\tilde{\phi},\tilde{\phi}_t)|\}\dd t\\
\leq C\int_0^{\infty}\CD(t)\dd t\leq CE_0.
%\nonumber \\
%&\leq C\int_0^{\infty}\CD(t)\dd t
%\leq CE_0.\nonumber
\end{multline*}
Moreover, since it also holds that
%Then combining it with the fact that
$$
\int_0^t\|[\tilde{v},\tilde{u},\tilde{\phi}](t)\|_{L^2}^2\dd t\leq C\int_0^{\infty}\CD(t)\dd t\leq CE_0,
$$
one can see that $\|[\tilde{v},\tilde{u},\tilde{\phi}](t)\|_{L^2}$ tends to zero as $t\to \infty$. Hence
%$ \lim_{t\rightarrow \infty}\|[\tilde{v},\tilde{u},\tilde{\phi}](t)\|_{L^2}=0,
%$ which,
by Sobolev inequality, one has
% $\|f\|_{L^{\infty}}\leq C\|f\|^{1/2}_{L^2}\|f_y\|^{1/2}_{L^2}$, further implies
$$
\|[\tilde{v},\tilde{u},\tilde{\phi}](t)\|_{L^\infty}\leq \sqrt{2}\|[\tilde{v},\tilde{u},\tilde{\phi}](t)\|_{L^2}^{1/2}
\|[\tilde{v}_y,\tilde{u}_y,\tilde{\phi}_y](t)\|_{L^2}^{1/2}\leq CE_0^{1/4}\|[\tilde{v},\tilde{u},\tilde{\phi}](t)\|_{L^2}^{1/2},
%\rightarrow 0
$$
which goes to zero as $t\rightarrow \infty$. This proves \eqref{LT}.  Therefore, the proof of Theorem \ref{thm1.3} is complete.\qed
%\end{subsection}

%\begin{section}

\section{Appendix}

\subsection{KdV-Burgers shock profile}
Concerning the shock profiles for the KdV-Burgers equation, we have the following result, cf.~\cite{BS}.

\begin{lemma}\label{lem3.1.0}
Let $0<\alpha<2$. Then if $\delta>0$ is sufficiently small, the equations \eqref{3.1.8-1} with \eqref{3.1.8-1.1}, \eqref{3.1.8-2} with \eqref{3.1.8-2.1}, and \eqref{3.1.8-3} with \eqref{3.1.8-3.1} have the smooth solutions $n_1$, $u_1$ and $\phi_1$, respectively, which are unique up to a spatial shift and satisfy the following properties:
\begin{equation}
n_1', u_1', \phi_1'<0, \label{3.1.9-0}
\end{equation}
and
\begin{equation}
\left|\f{\dd^k}{\dd z^k}\big[n_1-n_{1,\pm},u_1-u_{1,\pm},\phi_1-\phi_{1,\pm}\big]\right|\leq C_ke^{-\alpha|z|},\quad z\lessgtr0,\label{3.1.9-1}
\end{equation}
{for any integer $k\geq 0$,} where each positive constant $C_k$  is independent of $\delta.$
\end{lemma}
\begin{proof}
The existence and uniqueness of the smooth shock profile with properties \eqref{3.1.9-0} have been proved in \cite{BS}. We only show \eqref{3.1.9-1}. Integrating \eqref{3.1.8-1} from $-\infty$ to $z$, we have
\begin{align}\label{k1}
2\sqrt{T+1}n_1+(T+1)n_1^2-\sqrt{T+1}n_1'+\delta n_{1}''=0.
\end{align}
Let $q=n_1'$. Then, \eqref{k1} is equivalent to the following 1st-order ODE system for $[n_1,q]$:
\begin{equation*}%\label{k2}
\left\{\begin{aligned}
n_1'&=q,\\
q'&=\delta^{-1}\left\{\sqrt{T+1}q-(T+1)n_1^2-2\sqrt{T+1}n_1\right\}.
\end{aligned}\right.
\end{equation*}
The Jacobian at the far fields $(0,0)$ and $(-\f{2}{\sqrt{T+1}},0)$ can be directly computed as
$$J_{\pm}=\left(
 \begin{array}{cc}
  0\quad & 1 \\
    \pm2\delta^{-1}\sqrt{T+1}\quad & \delta^{-1}\sqrt{T+1} \\
    \end{array}
 \right).$$
The eigenvalues are given by
\begin{eqnarray*}
\lambda_{\pm,1}&=&\f{\sqrt{T+1}-\sqrt{T+1\pm8\sqrt{T+1}\delta}}{2\delta},\\
\lambda_{\pm,2}&=&\f{\sqrt{T+1}+\sqrt{T+1\pm8\sqrt{T+1}\delta}}{2\delta},
\end{eqnarray*}
with
$$
\la_{-,1}>0,\ \la_{-,2}>0,\ \la_{+,1}<0,\ \la_{+,2}>0.
$$
%$$
%\begin{aligned}
%&\lambda_{-,1}=\f{\sqrt{T+1}-\sqrt{T+1-8\sqrt{T+1}\delta}}{2\delta}>0,\quad
%\lambda_{-,2}=\f{\sqrt{T+1}+\sqrt{T+1-8\sqrt{T+1}\delta}}{2\delta}>0,\\
%&\lambda_{+,1}=\f{\sqrt{T+1}-\sqrt{T+1+8\sqrt{T+1}\delta}}{2\delta}<0,\quad
%\lambda_{+,2}=\f{\sqrt{T+1}+\sqrt{T+1+8\sqrt{T+1}\delta}}{2\delta}>0.
%\end{aligned}
%$$
Hence we have
$$
\begin{aligned}
\lim_{z\rightarrow-\infty}\f{q}{n_1}&=\lambda_{-,1}=\f{4\sqrt{T+1}}{\sqrt{T+1}+\sqrt{T+1-8\sqrt{T+1}\delta}}=2+O(\delta),\\
\lim_{z\rightarrow+\infty}\f{q}{n_1+\f{2}{\sqrt{T+1}}}&=\lambda_{+,1}=\f{-4\sqrt{T+1}}{\sqrt{T+1}+\sqrt{T+1-8\sqrt{T+1}\delta}}=-2+O(\delta).
\end{aligned}
$$
This implies that, for any $0<\alpha<2$,
\begin{align}\label{k3}
|n_1-n_{1,\pm}|\leq Ce^{-\alpha|z|},\quad z\gtrless0,
\end{align}
provided that $\delta>0$ is suitably small. Next, to estimate the  derivatives of $n_1$. Taking the inner product of \eqref{k1} with $w_{\alpha}^2n_1'$ gives
\begin{align}
\sqrt{T+1}\|n_1'\|_{L^2_{\alpha}}^2=(\delta n_1'',w_{\alpha}^2n_1')+(T+1)\left(n_1(n_1+\f{2}{\sqrt{T+1}}),w_{\alpha}^2n_1'\right)\nonumber.
\end{align}
From integration by parts, the first inner product term is equal to
$$
-\delta(n_1',w_{\alpha}w_{\alpha}'n_1')\leq C\delta\|n_{1}'\|_{L^2_{\alpha}}^2.
$$
By Cauchy-Schwarz, the second one is bounded by
$$
\eta\|n_1'\|_{L^2_{\alpha}}^2+C_{\eta}\left\|n_1(n_1+\f{2}{\sqrt{T+1}})\right\|_{L^2_{\alpha}}^2,
$$
for $\eta>0$.
Therefore, by taking both $\eta>0$ and $\delta>0$ suitably small, we have
$$
\|n_1'\|_{L^2_{\alpha}}\leq C\left\|n_1(n_1+\f{2}{\sqrt{T+1}})\right\|_{L^2_{\alpha}}\leq C.
$$
Here we have used the exponential decay property \eqref{k3} in the last inequality. The higher-order derivatives can be treated similarly. The proof of Lemma \ref{lem3.1.0} is complete.
\end{proof}

\subsection{Error estimates}
The following result gives the estimates on errors between the first-order approximation $[n_1,u_1,\phi_1]$ and the modified one $[n_{1,\vep},u_{1,\vep},\phi_{1,\vep}]$ defined in \eqref{3.1.9}. %It can be shown by combining
%in terms of the asymptotic structure \eqref{3.1.9-1}. The proof is also omitted for brevity.
{It can be shown by the same energy method as the one used for proving Lemma \ref{lem3.1.0}. So the proof is omitted for brevity.}
%and theory of differentiable dependence on parameters.
%Since the proof is quite standard, we omit it for brevity.

\begin{lemma}\label{lm3.1.1}
Let $0<\alpha<2$. Assume that both $\vep>0$ and $\delta>0$ are suitably small. {For any integer $k\geq 0$, there exists a constant}
%\marginpar{\Blue{ZHU: $C_k$ also depends on $\alpha$, right? RJ}} %$C_k>0$
$C_{k,\alpha}>0$ independent of $\delta$ and $\vep$ such that
\begin{align*}
\bigg|\f{\dd^k}{\dd z^k}\big[n_{1,\vep}-n_{1,\vep}(\pm\infty)-(n_1-n_{1,\pm})\big]\bigg|\leq C_{k,\alpha}\vep e^{-\alpha|z|},\\
\bigg|\f{\dd^k}{\dd z^k}\big[u_{1,\vep}-u_{1,\vep}(\pm\infty)-(u_1-u_{1,\pm})\big]\bigg|\leq C_{k,\alpha}\vep e^{-\alpha|z|},\\
\bigg|\f{\dd^k}{\dd z^k}\big[\phi_{1,\vep}-\phi_{1,\vep}(\pm\infty)-(\phi_1-\phi_{1,\pm})\big]\bigg|\leq C_{k,\alpha}\vep e^{-\alpha |z|},
\end{align*}
for $z\lessgtr0$. Moreover, let
%\begin{align*}%\label{n2}
$[n_2,u_2,\phi_2]:=\vep^{-1}[n_{1,\vep}-n_1,u_{1,\vep}-u_1,\phi_{1,\vep}-\phi_1]$,
%\end{align*}
then it holds that
$$
\left|\f{\dd^k}{\dd z^k}[n_2,u_2,\phi_2](z)\right|\leq C_{k,\alpha},\quad z\in \R.
$$
\end{lemma}

%\end{section}

\subsection{Explicit formulas of $r_2$ and $r_3$}
For completeness, we write down the explicit formulas of $r_2$ and $r_3$ as
%\textcolor[rgb]{1.00,0.00,0.00}{Exact formula of $r_2$ and $r_3$:
\begin{align}\label{5.1}
r_2=&\f{\vep n_R}{\sqrt{T+1}}\bigg\{2(T+1)\vep^{-1}(n_{1,\vep}-n_1)-1+2(2\sqrt{T+1}-\vep)n_{1,\vep}\nonumber\\
&+3(T+1)n_{1,\vep}^2+2\delta(1+\vep n_{1,\vep})\phi_{1,\vep}''-\vep\delta(1+\vep n_{1,\vep})(\phi_{1,\vep}')^2\bigg\}\nonumber\\
&+\f{\vep \delta(2n_{1,\vep}+\vep n_{1,\vep}^2)}{\sqrt{T+1}}\phi_{R}''+\f{\vep n_R'}{\sqrt{T+1}}-\f{\vep\delta\phi_{1,\vep}'(1+\vep n_{1,\vep})^2}{\sqrt{T+1}}\phi_{R}',
\end{align}
and
\begin{align}\label{5.2}
r_3=&\f{\vep[(T+1)(2+3\vep n_{1,\vep})-s^2+\vep^2\delta\phi_{1,\vep}''-\f{\delta}{2}\vep^3(\phi_{1,\vep}')^2]n_{R}^2}{\sqrt{T+1}}\nonumber\\
&+(T+1)^{3/2}\vep^3n_{R}^3+\f{2\vep^2\delta(1+\vep n_{1,\vep})n_R\phi_{R}''}{\sqrt{T+1}}+\f{\vep^4\delta n_{R}^2\phi_{R}''}{\sqrt{T+1}}\nonumber\\
&-\f{\vep^3\delta\phi_{1,\vep}'\phi_{R}'[2n_{R}(1+\vep n_{1,\vep})+%\vep
 {\vep^2}n_{R}^2]}{\sqrt{T+1}}-\f{\vep^2{\delta}(\phi_{R}')^2(1+\vep n_{1,\vep}+\vep^2 n_R)^2}{2\sqrt{T+1}},
\end{align}
%\marginpar{\gr{Here we miss a $\delta$.}}
respectively.

\subsection{Higher order energy estimates}\label{sec.a4}
We give the detailed proof of Lemma \ref{lem4.4} and Lemma \ref{lem4.5} as follows.

%\begin{lemma}\label{lem4.4}
%%There exists a constant $C>0$, such that for any $t \in [0,T]$,
%Under the assumptions of Proposition \ref{prop4.1}, it holds that
%\begin{multline}\label{4.2.26}
%\|\tilde{u}(t)\|_{L^2}^2+\int_0^t\|\pa_y\tilde{u}(s)\|_{L^2}^2\dd s\\
%\leq C\|\tilde{u}_0\|_{L^2}^2+C\int_0^t\left(\|\pa_y\tilde{\phi}(s)\|_{L^2}^2+\|\partial_y[\Phi,\Psi,\tilde{\phi}](s)\|_{L^2}^2\right)\dd s,
%\end{multline}
%and
%\begin{multline}\label{4.2.27}
%\|\pa_y\tilde{u}(t)\|_{L^2}^2+\int_0^t\|\pa_{yy}\tilde{u}(s)\|_{L^2}^2\dd s\leq  \|\pa_{y}\tilde{u}_0\|_{L^2}^2+C
%%\sup_{0\leq s\leq t}\sqrt{\CE(s)}
%\varepsilon_2
%\int_0^t\|\pa_{yy}\tilde{u}(s)\|_{L^2}^2\dd s\\
%+C\int_0^t\left(\|\tilde{v}(s)\|_{H^1}^2+\|\pa_y\tilde{u}(s)\|_{L^2}^2+\|\pa_y\tilde{\phi}(s)\|_{L^2}^2\right)\dd s,
%\end{multline}
%for all $t\in [0,M]$.
%\end{lemma}

\medskip
\noindent{\it Proof of Lemma \ref{lem4.4}:}
Taking the inner product of the second equation of \eqref{4.1.1} with $\tilde{u}$ with respect to $y$ over $\mathbb{R}$, one has
\begin{equation}\label{4.2.28}
(\tilde{u}_t-s\tilde{u}_y,\tilde{u})+\left(T\left(\f1v-\f1{\b{v}}\right)_y-
\mu\left(\f{u_y}{v}-\f{\b{u}_y}{v}\right)_y,\tilde{u}\right)-
\left(\f{\phi_y}{v}-\f{\b{\phi}_y}{\b{v}},\tilde{u}\right)=0.
\end{equation}
We estimate the left-hand inner products term by term. The first term is equal to $\f12\f{\dd}{\dd t}\|\tilde{u}(t)\|_{L^2}^2$. From integration by parts, the second term is computed as
\begin{multline}
\left(T\left(\f1v-\f1{\b{v}}\right)_y-
\mu\left(\f{u_y}{v}-\f{\b{u}_y}{v}\right)_y,\tilde{u}\right)=\left(-T\left(\f1v-\f1{\b{v}}\right)+
\mu\left(\f{u_y}{v}-\f{\b{u}_y}{v}\right),\tilde{u}_y\right)\nonumber\\
=\left(\mu v^{-1} \tilde{u}_y,\tilde{u}_y\right)+\left((-T+\mu\b{u}_y)\left(\f1v-\f1{\b{v}}\right),\tilde{u}_y\right),\nonumber
\end{multline}
where the first term in the last line above is a good one and the second term is bounded by $\eta\|\tilde{u}_y\|_{L^2}^2+C_{\eta}\|\Phi_y\|_{L^2}^2$ with an arbitrary constant $0<\eta<1$. The third term on the left-hand side of \eqref{4.2.28} is bounded by $C\{\|\Phi_y\|_{L^2}^2+\|\tilde{u}\|_{L^2}^2+{\|\tilde{\phi}_y\|_{L^2}^2}\}$. Plugging these estimates back into \eqref{4.2.28} and letting $0<\eta<1$ be suitably small, one has
\begin{align}\label{4.2.29}
\frac{1}{2}\f{\dd}{\dd t}\|\tilde{u}(t)\|_{L^2}^2+c\|\tilde{u}_y\|_{L^2}^2\leq C\left(\|\Phi_y\|_{L^2}^2+\|\tilde{u}\|_{L^2}^2+\|\tilde{\phi}_y\|_{L^2}^2\right).
\end{align}
Then \eqref{4.2.26} follows from integrating \eqref{4.2.29} over $[0,t]$.

Next, we show \eqref{4.2.27}. Taking the inner product of the second equation of \eqref{4.1.1} with $-\tilde{u}_{yy}$ with respect to $y$ over $\mathbb{R}$ gives that
\begin{multline}\label{4.2.30}
\f12\f{\dd}{\dd t}\|\tilde{u}_y\|_{L^2}^2+\underbrace{\left(T\left(\f1v-\f1{\b{v}}\right)_y,-\tilde{u}_{yy}\right)}_{\CI_9}
+\underbrace{\left(\mu\left(\f{u_y}{v}-\f{\b{u}_y}{\b{v}}\right)_y,\tilde{u}_{yy}\right)}_{\CI_{10}}\\
+\underbrace{\left(\f{\phi_y}{v}-\f{\b{\phi}_y}{\b{v}},\tilde{u}_{yy}\right)=0}_{\CI_{11}}.
\end{multline}
The inner product terms $\CI_9$, $\CI_{10}$ and $\CI_{11}$ above are computed as follows. By Cauchy-Schwarz, $\CI_9$ and $\CI_{11}$ can be bounded respectively as
$$
|\CI_9|=\left|T\left(\f{-v_y}{v^2}+\f{\b{v}_y}{\b{v}^2},-\tilde{u}_{yy}\right)\right|\leq
\eta\|\tilde{u}_{yy}\|_{L^2}^2+C_{\eta}(\|\tilde{v}\|_{L^2}^2+\|\tilde{v}_y\|_{L^2}^2),
$$
and
$$
|\CI_{11}|\leq \eta\|\tilde{u}_{yy}\|_{L^2}^2+C_{\eta}(\|\Phi_y\|_{L^2}^2+\|\tilde{\phi}_y\|_{L^2}^2),
$$
with an arbitrary constant $0<\eta<1$. As to $\CI_{10}$, we rewrite it as %follows.
\begin{align}
\CI_{10}&=\left(\mu v^{-1}\tilde{u}_{yy},\tilde{u}_{yy}\right)%+
{-}
\left({\mu}v^{-2}\tilde{v}_y\tilde{u}_y,\tilde{u}_{yy}\right)+\left({\mu}\b{u}_{yy}\left(\f1v-\f1{\b{v}}\right),\tilde{u}_{yy}\right)\nonumber\\
&\quad
+\left({\mu}\b{u}_y\b{v}_y\left(\f{1}{\b{v}^2}-\f1{v^2}\right),\tilde{u}_{yy}\right)-
\left({\mu}\f{\b{v}_y\tilde{u}_y+\b{u}_y\tilde{v}_y}{v^2},\tilde{u}_{yy}\right).\label{ad.i10}
\end{align}
On the right-hand side of \eqref{ad.i10}, the first term is a good one, and the second term is bounded as
$$
\begin{aligned}
\left|\left(v^{-2}\tilde{v}_y\tilde{u}_y,\tilde{u}_{yy}\right)\right|&\leq C\|\tilde{u}_{y}\|_{L^{\infty}}\|\tilde{v}_{y}\|_{L^2}\|\tilde{u}_{yy}\|_{L^2}\\
&\leq
C\|\tilde{u}_y\|_{L^2}^{\f12}\|\tilde{v}_y\|_{L^2}^{\f14}
\|\tilde{u}_{yy}\|_{L^2}^{\f32}\|\tilde{v}_y\|^{\f34}_{L^2}\\
&\leq C %\red{\frac{C}{4}}
\|\tilde{v}_y\|_{L^2}\|\tilde{u}_y\|_{L^2}^2+ C %\f{3C}{4}
\|\tilde{v}_y\|_{L^2}
\|\tilde{u}_{yy}\|_{L^2}^2\\
&\leq C\sqrt{\CE(t)}\left(\|\tilde{u}_y\|_{L^2}^2+\|\tilde{u}_{yy}\|_{L^2}^2\right),
\end{aligned}
$$
where we have used the Sobolev inequality
%$\|f\|_{L^{\infty}}\leq C\|f\|_{L^2}^{\f12}\|f_y\|_{L^2}^{\f12}$
in the second line and Young's inequality
%$ab\leq \f34a^{\f43}+\f14b^4$ with $a,b\geq0$
in the third line. Also, the last three terms on the right-hand side of \eqref{ad.i10} are bounded by
$$
C|v_+-v_-|\left\{\|\tilde{u}_{y}\|_{L^2}^2+\|\tilde{u}_{yy}\|_{L^2}^2+\|\tilde{v}\|_{L^2}^2{+\|\tilde{v}_y\|_{L^2}^2}\right\}.
$$
Plugging those estimates on $\CI_9$ to $\CI_{11}$ back into \eqref{4.2.30} and taking $\eta>0$ suitably small, one has
\begin{equation}\label{4.2.31}
\frac{1}{2}\f{\dd}{\dd t}\|\tilde{u}_y\|_{L^2}^2+c\|\tilde{u}_{yy}\|_{L^2}^2
\leq C\|[\tilde{v},\tilde{v}_y,\tilde{u}_y,\tilde{\phi}_y]\|_{L^2}^2
%+\|\tilde{v}_y\|_{L^2}^2+\|\tilde{u}_y\|_{L^2}^2+\|\tilde{\phi}_y\|_{L^2}^2
+C\{|v_+-v_-|+\sqrt{\CE(t)}\}\|\tilde{u}_{yy}\|_{L^2}^2.
\end{equation}
Recall \eqref{4.2.1} and \eqref{4.2.2}. Then \eqref{4.2.27} follows from integrating \eqref{4.2.31} over $[0,t]$ and  letting $e_1$ and $\tilde{\vep}_1$ be small enough. The proof of Lemma \ref{lem4.4} is complete.\qed

\medskip
\noindent{\it Proof of Lemma \ref{lem4.5}:}
By taking the inner products of the first and second equations of \eqref{4.1.1}
%and the second equation of $\eqref{4.1.1}_2$
with $-\mu\tilde{v}_{yy}$
and $-v\tilde{v}_y$ respectively and adding the resultant equations together, we obtain that
%\marginpar{\gr{In \eqref{4.2.33}, there is a minus sign with red color.}}
\begin{multline}\label{4.2.33}
\f{\mu}{2}\f{\dd}{\dd t}\|\tilde{v}_y\|_{L^2}^2+\underbrace{(\tilde{u}_t-s\tilde{u}_y,-v\tilde{v}_y)}_{\CI_{12}}
+\underbrace{\left(T\left(\f1v-\f1{\b{v}}\right)_y,-v\tilde{v}_y\right)}_{\CI_{13}} \\
+\underbrace{(\mu \tilde{u}_y,\tilde{v}_{yy})+\left(\left(\f{\mu u_y}{v}-\f{\mu \b{u}_y}{\b{v}}\right)_y,v\tilde{v}_y\right)}_{\CI_{14}}{-}
\underbrace{\left(\f{\phi_y}{v}-\f{\b{\phi}_y}{\b{v}},v\tilde{v}_y\right)}_{\CI_{15}}=0.
\end{multline}
We estimate terms $\CI_{12}$ to $\CI_{15}$ as follows. Firstly, $\CI_{12}$ is computed as
\begin{align}
\CI_{12}&=\f{\dd}{\dd t}(\tilde{u},-v\tilde{v}_y)+(\tilde{u},\tilde{v}_t\tilde{v}_y)+
(\tilde{u},v\tilde{v}_{ty})+(s\tilde{u}_y,v\tilde{v}_y)\nonumber\\
&=\f{\dd}{\dd t}(\tilde{u},-v\tilde{v}_y)+(\tilde{u},-\b{v}_y\tilde{v}_t)+(\tilde{u}_y,-v\tilde{v}_t)
+(s\tilde{u}_y,v\tilde{v}_y).\nonumber
\end{align}
Replacing $\tilde{v}_t$ by the first equation of \eqref{4.1.1}, $\CI_{12}$ is further equal to
\begin{align}
\CI_{12}=\f{\dd}{\dd t}(\tilde{u},-v\tilde{v}_y)+(\tilde{u},-\b{v}_y(s\tilde{v}_y+\tilde{u}_y))+
(\tilde{u}_y,-v(s\tilde{v}_y+\tilde{u}_y))+s(\tilde{u}_y,v\tilde{v}_y),\nonumber
\end{align}
where the last three terms are bounded by $\eta\|\tilde{v}_y\|_{L^2}^2+C_{\eta}\{\|\tilde{u}\|_{L^2}^2+
\|\tilde{u}_y\|_{L^2}^2\}$ with an arbitrary constant $0<\eta<1$. As to $\CI_{13}$, it follows that
\begin{align}
\CI_{13}=\left(-\f{Tv_y}{v^2}+\f{T\b{v}_y}{\b{v}^2},-v\tilde{v}_y\right)=(Tv^{-1}\tilde{v}_y,\tilde{v}_y)
+\left(T\b{v}_y\left(\f{1}{\b{v}^2}-\f{1}{v^2}\right),-v\tilde{v}_y\right),\nonumber
\end{align}
where the first term on the right is good and the second inner product is bounded by
$\eta\|\tilde{v}_y\|_{L^2}^2+C_{\eta}\|\tilde{v}\|_{L^2}^2$ with an arbitrary constant $0<\eta<1$.
For $\CI_{14}$, one has
\begin{align}
\CI_{14}&=\left(\f{\mu u_{yy}}{v}-\f{\mu\b{u}_{yy}}{\b{v}},v\tilde{v}_y\right)+(\mu\tilde{u}_y,\tilde{v}_{yy})
+\left(\f{-\mu u_yv_y}{v^2}+\f{\mu\b{u}_y\b{v}_y}{\b{v}^2},v\tilde{v}_y\right)\nonumber\\
&=\left(\mu \b{u}_{yy}(v^{-1}-\b{v}^{-1}),v\tilde{v}_y\right)
+\left(\f{-\mu u_yv_y}{v^2}+\f{\mu\b{u}_y\b{v}_y}{\b{v}^2},v\tilde{v}_y\right).\nonumber
\end{align}
By Cauchy-Schwarz, the first inner product term on the right is bounded as
$$
\left|\left(\mu \b{u}_{yy}(v^{-1}-\b{v}^{-1}),v\tilde{v}_y\right)\right|\leq \eta{\|\tilde{v}_y\|_{L^2}^2}+C_{\eta}\|\tilde{v}\|_{L^2}^2.
$$
And the second one is computed as
\begin{multline}\label{ad.p1}
\left(\f{-\mu u_yv_y}{v^2}+\f{\mu\b{u}_y\b{v}_y}{\b{v}^2},v\tilde{v}_y\right)= \big(-\mu v^{-1}(\b{u}_y\tilde{v}_y+\b{v}_y\tilde{u}_y),\tilde{v}_y\big)\\
-\left(\mu\b{u}_y\b{v}_y\left(\f1{v^2}-\f{1}{\b{v}^2}\right),v\tilde{v}_y\right)
%+
-(\mu v^{-1}\tilde{v}_y\tilde{u}_y,\tilde{v}_y).
\end{multline}
On the right-hand side of \eqref{ad.p1}, the last inner product term is bounded as
$$
\begin{aligned}
|(\mu v^{-1}\tilde{v}_y\tilde{u}_y,\tilde{v}_y)|&\leq \|\tilde{u}_y\|_{L^{\infty}}\|\tilde{v}_y\|_{L^2}^2\leq C\|\tilde{u}_y\|_{H^1}\|\tilde{v}_y\|_{L^2}^2\\
&\leq C\|\tilde{v}_y\|_{L^2}\|\tilde{u}_y\|_{H^1}^2+\|\tilde{v}_y\|_{L^2}^3\leq C\sqrt{\CE(t)}\{\|\tilde{u}_y\|_{H^1}^2+\|\tilde{v}_y\|_{L^2}^2\},
\end{aligned}
$$
and the rest terms are bounded by $C\{|v_+-v_-|+\sqrt{\CE(t)}\}\{\|\tilde{u}_y\|_{L^2}^2+\|\tilde{v}\|_{H^1}^2\}$. Thus, it follows from the above estimates that
$$
|\CI_{14}|\leq\eta\|\tilde{v}_y\|_{L^2}^2+C_{\eta}\|\tilde{v}\|_{L^2}^2+
C\{|v_+-v_-|+\sqrt{\CE(t)}\}\{\|\tilde{u}_y\|_{H^1}^2+\|\tilde{v}_y\|_{L^2}^2\},
$$
with an arbitrary constant $0<\eta<1$. For $\CI_{15}$, it holds by Cauchy-Schwarz that
\begin{align}
|\CI_{15}|\leq \eta\|\tilde{v}_y\|_{L^2}^2+C_{\eta}\{\|\tilde{v}\|_{L^2}^2+\|\tilde{\phi}_y\|_{L^2}^2\}.\nonumber
\end{align}
Plugging those estimates on $\CI_{12}$ to $\CI_{15}$ back into \eqref{4.2.33} and letting $\eta>0$ be chosen suitably small, we obtain that
\begin{multline}\label{4.2.34}
\f{\dd}{\dd t}\left\{\frac{\mu}{2}\|\tilde{v}_y\|_{L^2}^2+(\tilde{u},-v\tilde{v}_y)\right\}+c\|\tilde{v}_y\|_{L^2}^2\leq C\{\|\tilde{u}\|_{H^1}^2+\|\tilde{v}\|_{L^2}^2+\|\tilde{\phi}_y\|_{L^2}^2\}\\
+C\{|v_+-v_-|+\sqrt{\CE(t)}\}\{\|\tilde{v}_y\|_{L^2}^2+
\|\tilde{u}_{yy}\|_{L^2}^2\}.
\end{multline}
Recall \eqref{4.2.1} and \eqref{4.2.2}. Then \eqref{4.2.32} follows from integrating \eqref{4.2.34} over $[0,t]$ and letting $e_1$ and $\tilde{\vep}_1$ be small enough. The proof of Lemma \ref{lem4.5} is complete.
\qed

\medskip
\noindent{\bf Acknowledgments:}
Renjun Duan was supported by the General Research Fund (Project No.~14301515) from RGC of Hong Kong. Shuangqian Liu was supported by the grants from the National Natural Science Foundation of China under contracts 11471142, 11731008 and 11571136.  Zhu Zhang would thank the Institute for Analysis and Scientific Computing, Vienna University of Technology for their kind hospitality and also gratefully acknowledge the support from the Eurasia-Pacific Uninet Ernst Mach Grant.

\end{document}